\documentclass{etds}

\usepackage[francais]{babel}
\usepackage[utf8]{inputenc}
\usepackage[T1]{fontenc}
\usepackage{graphicx}

\usepackage{amsmath,amssymb,bbm}

\usepackage{float, mathtools}

\usepackage{tikz}
\usetikzlibrary{arrows,shapes,decorations,automata,backgrounds,petri,trees}

\usepackage{tikzexternal}

\tikzsetexternalprefix{}

\newtheorem{thm}{Th\'eor\`eme}[section]
\newtheorem{prop}[thm]{Proposition}
\newtheorem{lemme}[thm]{Lemme}

\newtheorem{souslemme}[thm]{Sous lemme}
\newtheorem{props}[thm]{Propri\'et\'es}
\newtheorem{define}[thm]{D\'efinition}
\newtheorem{def-prop}[thm]{D\'efinition-proposition}
\newtheorem{cor}[thm]{Corollaire}

\newtheorem{ex}[thm]{Exemple}

\newtheorem{contr-ex}[thm]{Contre-exemple}
\newtheorem{rem}[thm]{Remarque}
\newtheorem{conj}[thm]{Conjecture}
\newtheorem{quest}{Question}

\newcommand\Part{\mathcal P}

\newcommand\N{\mathbb N}
\newcommand\Z{\mathbb Z}
\newcommand\Q{\mathbb Q}
\newcommand\R{\mathbb R}
\newcommand\C{\mathbb C}
\newcommand\D{\mathbb D}

\newcommand\bord[1]{\partial #1}
\newcommand\bX{\bord{X}}
\newcommand\cX{\overline{X}}

\newcommand\hyp{{\mathbb H}^2_\R} 
\newcommand\hypt{{\mathbb H}^3_\R} 
\newcommand\hypn{{\mathbb H}^n_\R} 

\newcommand\abs[1]{\left|#1\right|}
\newcommand\norm[1]{\|#1\|}
\newcommand\eq{\Leftrightarrow}
\newcommand\imp{\Rightarrow}

\newcommand\longeq{\Longleftrightarrow}

\newcommand\floor[1]{\left\lfloor #1 \right\rfloor}
\newcommand\ceil[1]{\left\lceil #1 \right\rceil}

\newcommand\diam{\operatorname{diam}}
\newcommand\dv{\dim_{\operatorname{vis}}}

\newcommand\vol{\operatorname{vol}}
\newcommand\supp{\operatorname{supp}}
\newcommand\Isom{\operatorname{Isom}}
\newcommand\Aff{\operatorname{Aff}}

\newcommand\defi[1]{\textbf{#1}}
\newcommand\Adh[1]{\operatorname{Adh\left(#1\right)}}
\newcommand\interieur[1]{\overset{\circ}{#1}}

\def\restriction#1#2{\mathchoice
              {\setbox1\hbox{${\displaystyle #1}_{\scriptstyle #2}$}
              \restrictionaux{#1}{#2}}
              {\setbox1\hbox{${\textstyle #1}_{\scriptstyle #2}$}
              \restrictionaux{#1}{#2}}
              {\setbox1\hbox{${\scriptstyle #1}_{\scriptscriptstyle #2}$}
              \restrictionaux{#1}{#2}}
              {\setbox1\hbox{${\scriptscriptstyle #1}_{\scriptscriptstyle #2}$}
              \restrictionaux{#1}{#2}}}
\def\restrictionaux#1#2{{#1\,\smash{\vrule height .8\ht1 depth .85\dp1}}_{\,#2}}

\newcommand{\convfaible}[1] 
{
	\ensuremath{\displaystyle{\smash{\,\mathop{\longrightarrow}\limits_{n\to\infty}^{#1}\,}}}
}

\newcommand*{\EnsembleQuotient}[2]
{
	\ensuremath
	{
		#1/\!\raisebox{-.65ex}{\ensuremath{\mathcal{#2}}}
	}
}

\usepackage{hyperref}

\usepackage{verbatim}

\newwrite\SORTIE


\newwrite\SORTIE
\makeatletter
\newcommand{\dpo}[1]
{
	\begingroup
	\@bsphack
	\immediate\openout\SORTIE in.txt 
	\let\do\@makeother\dospecials
	\catcode`\^^M\active
	\immediate\write\SORTIE{\pgfactualjobname} 
	\immediate\write\SORTIE{#1} 
	\immediate\closeout\SORTIE 
	 \@esphack
	\endgroup
	
	\immediate\write18{./dp} 
	\input{out} 
}
\makeatother

\newenvironment{proof}[1][Preuve]
{ 
	\textbf{#1} ---
}{ 
	\hfill $\square$
}

\begin{document}	
	\immediate\write18{./dp -init}
	
	\ETDS{12}{}{XX}{2013}

	\runningheads{P.\ Mercat}{Semi-groupes d'isométries d'un espace Gromov-hyperbolique} 
	
	\title{Entropie des semi-groupes d'isométries d'un espace Gromov-hyperbolique}
	
	\author{Paul Mercat}
	
	\address{Aix-Marseille Université, C.M.I., 39 rue F. Joliot Curie, 13453 Marseille, FRANCE \\
	\email{paul.mercat@univ-amu.fr}}
	
	\recd{Décembre $2013$}
	
	\begin{abstract}
		Nous généralisons aux semi-groupes convexes co-compacts un très joli théorème de Patterson-Sullivan, donnant l'égalité entre exposant critique (c'est-à-dire la vitesse exponentielle de croissance) et dimension de Hausdorff de l'ensemble limite (c'est-à-dire la taille du plus petit fermé invariant non vide), d'un groupe discret d'isométrie d'un espace hyperbolique.
		Nous démontrons ce résultat dans le cadre général des semi-groupes d'isométries d'un espace Gromov-hyperbolique propre à bord compact.
		Pour cela, nous introduisons une notion d'entropie, qui généralise la notion d'exposant critique des groupes discrets, et nous montrons que celle-ci est égale à la borne supérieure des exposants critiques des sous-semi-groupes de Schottky (c'est-à-dire les semi-groupes ayant la dynamique la plus simple). Nous obtenons ainsi plusieurs autres corollaires tels que la semi-continuité inférieure de l'entropie, le fait que l'exposant critique d'un semi-groupe séparé, qui est définit comme une limite supérieure, soit en fait une vraie limite, et enfin l'existence de \og gros \fg\ sous-groupes de Schottky dans les groupes discrets d'isométries.
	\end{abstract}

	\tableofcontents
%
	
	
	\section{Introduction}
	
	Le premier cas intéressant de la théorie de Patterson-Sullivan, est l'étude des sous-groupes du groupe $SL(2, \R)$.
	\'Etant donné un sous-groupe $\Gamma$ de $SL(2,\R)$, on s'intéresse à une donnée dynamique, l'exposant critique, et une donnée géométrique, l'ensemble limite, qui sont définis de la façons suivante :
	\begin{define}
		L'\defi{exposant critique} est le réel noté $\delta_\Gamma$, définit par
		\[
			\delta_\Gamma := \limsup_{n \to \infty} \frac{1}{2n} \log(\#\{ \gamma \in \Gamma | \log(\norm{\gamma}) \leq n \}).
		\]
		Cela correspond à la vitesse exponentielle à laquelle croît le groupe.
		
		L'\defi{ensemble limite} est la partie de $\hat{\R} := \R \cup \{ \infty \}$ notée $\Lambda_\Gamma$, qui est le plus petit fermé invariant non vide pour l'action de $\Gamma$ sur $\hat{\R}$ par homographie.
	\end{define}
	
	Les ensembles limites de sous-groupes discrets de $SL(2, \R)$ ressemblent en général à des ensembles de Cantor, et leur exposant critique est un nombre entre $0$ et $1$.
	
	On a alors le très joli résultat suivant, qui relie cette donnée dynamique et cette donnée géométrique, qui n'avaient à priori rien à voir.
	
	\begin{thm}[Patterson-Sullivan]
		Si $\Gamma$ est un sous-semi-groupe de type fini et non élémentaire de $SL(2,\R)$, alors on a l'égalité
		\[
			\dim_H(\Lambda_\Gamma) = \delta_\Gamma.
		\]
	\end{thm}
	
	Bien que l'énoncé ne fasse pas intervenir de géométrie hyperbolique, la preuve utilise de façon cruciale le fait que $SL(2,\R)$ agisse par isométrie sur le plan hyperbolique $\hyp$.
	
	\begin{figure}[H]
		\centering
		\caption{Action du groupe $SL(2,\Z)$ sur le disque de Poincaré}
		\includegraphics{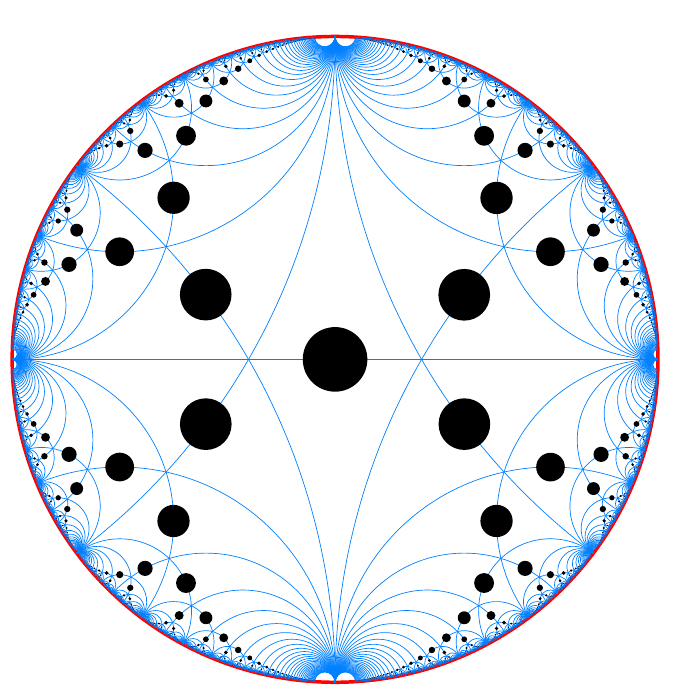}
	\end{figure}
	
	\subsection{Présentation de mes résultats}
	
	Pour un sous-semi-groupe $\Gamma$ du groupe d'isométries $\Isom(X)$ d'un espace Gromov-hyperbolique $X$ propre, nous définissons une notion d'entropie (par analogie avec l'entropie volumique), qui généralise la notion d'exposant critique des groupes discrets.
	L'entropie est une façon de mesurer \og l'espace occupé \fg\ par l'orbite d'un semi-groupe dans l'espace $X$, tandis que l'exposant critique mesure la quantité d'éléments.
	L'énoncé suivant donne une caractérisation de l'entropie qui justifie son intérêt :
	
	\begin{thm} \label{thm0}
		Soit $X$ un espace Gromov-hyperbolique d'adhérence compacte, et soit $\Gamma$ un semi-groupe d'isométries de $X$ dont l'ensemble limite contient au moins deux points, alors on a
		\[ \sup_{\Gamma' < \Gamma \atop \Gamma' \text{ Schottky}} \delta_{\Gamma'} = h_\Gamma. \]
	\end{thm}
	
	Autrement dit, l'entropie $h_\Gamma$ est la borne supérieure des exposants critiques des sous-semi-groupes de Schottky du semi-groupe $\Gamma$ (c'est-à-dire des sous-semi-groupes ayant la dynamique la plus simple (voir \ref{schottky} pour une définition précise)).
	
	Nous supposons ici que l'espace est propre et de bord compact), mais nous verrons que cette hypothèse n'est pas beaucoup plus forte que de demander seulement la propreté de l'espace $X$ (voir le paragraphe \ref{scompacite}).

	
	\begin{rem}
		Lorsque $\Gamma$ est un groupe, si l'on remplace les semi-groupes de Schottky par des \emph{groupes} de Schottky au sens classique, le résultat devient faux d'après un théorème de P.Doyle (voir \cite{doyle}).
	\end{rem}
	
	Le théorème \ref{thm0} permet d'étudier la \og dimension à l'infini \fg\ du semi-groupe, puisque l'on obtient en corollaire une généralisation d'un résultat de F.Paulin (voir \cite{paulin}) :
	
	\begin{cor} \label{cpaulin}
		Soit $X$ un espace Gromov-hyperbolique propre à bord compact, et soit $\Gamma$ un semi-groupe d'isométries de $X$ dont l'ensemble limite contient au moins deux points. Alors la dimension visuelle de l'ensemble limite radial du semi-groupe $\Gamma$ est égale à l'entropie du semi-groupe :
		\[ \dv {\Lambda^c_\Gamma} = h_\Gamma. \]
	\end{cor}
	
	La dimension visuelle $\dv$ est une généralisation naturelle de la dimension de Hausdorff au bord d'un espace hyperbolique. Et l'ensemble limite radial $\Lambda^c_\Gamma$ (appelé aussi ensemble limite conique) du semi-groupe $\Gamma$ est l'ensemble des points $\xi$ du bord qui sont limite d'une quasi-géodésique de l'orbite $\Gamma o$. 
	
	Ceci nous permet de calculer la dimension de Hausdorff de certains ensemble auto-similaires pour lesquels on ne savait pas encore faire à ma connaissance :
	
	\begin{cor}
		Si $\beta$ est un nombre de Salem et si $\Gamma$ est le sous-semi-groupe du groupe affine de $\C$ engendré par les applications
		\[ x \mapsto \frac{x}{\beta} + t \]
		où $t \in A$ pour $A$ une partie finie de $\Q(\beta)$,
		alors on a l'égalité
		\[ \dim_H (\Lambda_\Gamma) = \delta_\Gamma. \]
	\end{cor}
	
	Le résultat était connu pour un nombre de Pisot, 
	mais semble nouveau pour un nombre de Salem. 
	Une telle égalité est toujours à l'état de conjecture pour les semi-groupes de Kenyon (voir sections \ref{dev-beta} et \ref{semi-ken}).
	
	Voici d'autres corollaires du théorème \ref{thm0} :
	
	\begin{cor} \label{lim}
		Soit $X$ un espace Gromov-hyperbolique propre à bord compact et soit $\Gamma$ un semi-groupe discret d'isométries de $X$ dont l'ensemble limite contient au moins deux points.
		La limite supérieure dans la définition de l'exposant critique du semi-groupe $\Gamma$ est une vraie limite :
		\[ \delta_\Gamma = \lim_{n \to \infty} \frac{1}{n} \log( \#\{ \gamma \in \Gamma | d(o, \gamma o) \leq n \}). \]
	\end{cor}
	
	\begin{cor} \label{cont-inf}
		L'entropie est semi-continue inférieurement en les semi-groupes dont l'ensemble limite contient au moins deux points.
	\end{cor}
	
	Voir corollaire \ref{semi-cont} pour un énoncé plus précis.
	Ce résultat généralise celui que donne F. Paulin à la fin de son article \cite{paulin}, puisque l'on ne fait ni l'hypothèse que le semi-groupe soit un groupe, ni qu'il soit discret, ni qu'il soit de type fini, ni que l'espace soit géodésique ou quasi-géodésique, et cela fonctionne aussi bien pour la convergence algébrique que pour la convergence géométrique. 
	
	Nos résultats sur les semi-groupes permettent d'obtenir un résultat sur les groupes :
	
	\begin{cor} \label{sous-groupes}
		Soit $X$ un espace Gromov-hyperbolique propre à bord compact, et soit $\Gamma$ un groupe discret et sans torsion d'isométries de $X$ ne fixant pas de point au bord, alors on a
		\[ \sup_{\Gamma' < \Gamma \atop \Gamma' \text{ Schottky}} \delta_{\Gamma'} \geq \frac{1}{2} \delta_\Gamma, \]
		où la borne supérieure est prise sur l'ensemble des sous-groupes de Schottky du groupe $\Gamma$.
	\end{cor}
	
	
	\subsection{Organisation de l'article}
	Dans la section \ref{cadre}, nous donnons les définitions et outils qui serviront dans la suite.
	La section \ref{contract} est consacrée à définir et donner des propriétés sur les semi-groupes contractants.
	On y démontre par exemple que l'ensemble limite d'un semi-groupe contractant de type fini est toujours radial.
	On construit dans la section \ref{pthm1} une grosse partie contractante d'un semi-groupe d'isométries.
	C'est l'étape principale pour démontrer le théorème \ref{thm0}.
	Dans la section \ref{schottky}, on définit ce qu'est un semi-groupe de Schottky, et l'on démontre l'existence de gros sous-semi-groupes de Schottky.
	Cela permettra d'avoir une preuve du théorème \ref{thm0}.
	Les sections \ref{dim-vis} à \ref{semi-continuite} sont consacrées à des corollaires du théorème.
	Dans la première, nous obtenons une généralisation d'un résultat de F. Paulin (corollaire \ref{cpaulin}, voir section \ref{paulin}), que nous appliquons à l'étude des semi-groupes de développement en base $\beta$ (voir section \ref{dev-beta}). Dans la deuxième, nous voyons un résultat sur les sous-groupes de Schottky d'un groupe discret (corollaire \ref{sous-groupes}, voir section \ref{groupes}). Nous montrons ensuite dans la section \ref{car-ent} que l'exposant critique est une vraie limite. Et pour finir nous montrons la semi-continuité inférieure de l'entropie des semi-groupes (corollaire \ref{cont-inf}, voir section \ref{semi-continuite}).
	
	\section{Le cadre} \label{cadre}
	
	Nous définissons ici les objets que nous manipulerons dans tout l'article, en commençant par les espaces Gromov-hyperboliques et leur bord.
	Nous verrons en particulier la définition et des propriétés de l'entropie.
	
	On notera $d(x,y)$ la distance entre deux point $x,y \in X$ d'un espace métrique $X$.
	
	\subsection{Espaces hyperboliques}
	
	\'Etant donné un espace métrique, on peut définir le produit de Gromov, qui permet de mesurer le défaut d'égalité triangulaire de trois points $x$, $y$ et $o$ :
	
	\begin{define}
		Soit $X$ un espace métrique de point base $o$.
		On appelle \defi{produit de Gromov} de deux points $x,y \in X$ le réel
		\[ (x|y) := \frac{1}{2} \left( d(x, o) + d(y,o) - d(x,y) \right). \]
	\end{define}
	
	On peut alors définir la Gromov-hyperbolicité :
	
	\begin{define}
		On dit qu'un espace $X$ est \defi{$\delta$-hyperbolique} pour un réel $\delta \geq 0$, si c'est un espace métrique vérifiant l'inégalité
		\[ (x|z) \geq \min\{(x|y), (y|z)\} - \delta \]
		pour tous $x, y$ et $z$ $\in X$.
		
		On dit qu'un espace $X$ est \defi{Gromov-hyperbolique} s'il existe un réel $\delta$ tel que l'espace $X$ est $\delta$-hyperbolique.
	\end{define}
	
	Un espace Gromov-hyperbolique est un espace qui ressemble, vu de loin, à un arbre. Les arbres sont d'ailleurs des espaces $0$-hyperboliques.
	Dans un espace hyperbolique, un grand produit de Gromov caractérise des points qui sont \og proches vus de $o$ \fg.
	
	Définissons l'exposant critique. Il s'agit de la vitesse exponentielle de croissance d'une partie de $X$.
	
	\begin{define}
		On appelle \defi{exposant critique} d'une partie $P \subset X$ d'un espace métrique $X$ de point base $o$, le réel (éventuellement infini)
		\[ \delta_P := \limsup_{n \rightarrow \infty} \frac{1}{n}\log(\#(B(o, n) \cap P)). \]
	\end{define}
	
	Par inégalité triangulaire, l'exposant critique ne dépend pas du point $o$ choisi.
	
	\begin{rem} \label{delta-ann}
		On peut aussi considérer seulement les éléments d'un anneau.
		On a
		\[ \delta_P = \limsup_{n \rightarrow \infty} \frac{1}{n}\log(\#((B(o, n+1) \backslash B(o, n)) \cap P)) \]
		si $P$ est une partie non bornée.
	\end{rem}
	
	On va maintenant définir l'entropie de n'importe quelle partie $P$ de $X$.
	Cela correspond à la vitesse exponentielle de croissance du point de vue de l'espace occupé, et non plus du point de vu du comptage comme pour l'exposant critique.
	
	\begin{define}
		Soit $X$ un espace métrique.
		On dit qu'une partie $P \subseteq X$ est \defi{séparée} s'il existe un réel $\epsilon > 0$ tel que la partie $P$ soit \defi{$\epsilon$-séparée}, c'est-à-dire tel que
		\[ d(x,y) > \epsilon \]
		pour tous $x \neq y$ $\in P$. \\
		On dit qu'une partie $P \subseteq X$ est une partie \defi{couvrante} de $Y \subseteq X$ s'il existe un réel $\epsilon > 0$ telle que la partie soit \defi{$\epsilon$-couvrante} de $Y$, c'est-à-dire telle que
		pour tout $y \in Y$, il existe $x \in P$ tel que
		\[ d(x,y) \leq \epsilon. \]
	\end{define}
	
	\begin{define}
		On appelle \defi{entropie} d'une partie $P \subseteq X$ d'un espace métrique $X$, le réel (éventuellement infini)
		\[ h_P := \sup_S \delta_{P \cap S}, \]
		où la borne supérieure est prise sur les parties $S$ séparées de $X$. \\
		On pose $h_{\emptyset} = 0$ par convention.
	\end{define}
	
	La remarque suivante justifie le nom d'\og entropie \fg\ :
	\begin{rem}
		Quand $X = \hypn$, l'entropie d'une partie $P \subseteq X$ est égale à l'entropie volumique de l'ensemble $\{ x \in X | d(x, P) \leq 1 \}$.
	\end{rem}
	
	\begin{proof}[Idée de preuve]
		On peut trouver une partie $S$ $\epsilon$-séparée et $r$-couvrante de la partie $P$ pour un réel $1 > \epsilon > 0$ assez petit et un réel $r$ assez grand (voir \ref{epsc}), et on peut montrer que son exposant critique est alors égale à l'entropie de $P$. Soit $o$ un point de $X = \hypn$.
		On a les inégalités
		\[
			\vol (B(o, \epsilon)) \cdot \# B \cap S = \vol( \bigcup_{x \in S \cap B} B(x, \epsilon) ) \leq \vol (B(o, n + \epsilon) \cap \{ x \in X | d(x, P) \leq 1 \}),
		\]
		\begin{align*}
			\vol (B(o, n) \cap \{ x \in X | d(x, P) \leq 1 \})	&\leq \vol( \bigcup_{x \in S \cap B(o, n+1+r)} B(x, r+1) ) \\
												&\leq \vol(o, r+1) \cdot \# B(o, n+1+r) \cap S.
		\end{align*}
		En passant à la limite après avoir pris le $\log$ et divisé par $n$, on obtient le résultat annoncé.
	\end{proof}
	
	\begin{ex}
		Dans l'espace $X = \hyp$ muni de la métrique usuelle, l'entropie de $\hyp$ vaut $1$ et l'entropie d'une horoboule vaut $\frac{1}{2}$.
	\end{ex}
	
	Voici quelques propriétés de l'entropie.
	\begin{props} \textbf{Propriétés de l'entropie} \label{pptex} \\
		Soit une partie $A \subseteq X$ d'un espace métrique $X$ de point base $o$.
		\begin{enumerate}
			\item L'entropie ne dépend pas du point base $o \in X$ choisi.
			\item On a l'inégalité
				$ h_A \leq \delta_A$.
			\item Si la partie $A$ est séparée, alors on a l'égalité $h_A = \delta_A$.
			\item L'entropie est croissante : si $A \subseteq B \subseteq X$, alors on a
				\[ h_A \leq h_B \leq h_X. \]
		\end{enumerate}
	\end{props}
	
	\begin{proof}
		\begin{enumerate}
			\item Cela découle du fait que l'exposant critique ne dépende pas du point base $o$ choisi, ce qui découle de l'inégalité triangulaire.
			\item Pour toute partie séparée $S$, on a $\delta_{A \cap S} \leq \delta_{A}$. D'où le résultat en passant à la borne supérieure.
			\item Si la partie $A$ est séparée, on a $h_A = \sup_S \delta_{A \cap S} \geq \delta_{A \cap A} = \delta_A$.
			\item Pour toute partie séparée $S$ on a $A \cap S \subseteq B \cap S$, donc $\delta_{A \cap S} \leq \delta_{B \cap S}$. D'où le résultat en passant à la borne supérieure.
		\end{enumerate}
	\end{proof}
	
	\begin{define} \label{propre}
		On dit qu'un espace métrique $X$ est \defi{propre} si ses boules fermées sont compactes.
	\end{define}
	
	
	%
	\subsection{Bord d'un espace hyperbolique} \label{sbord}
	
	Nous allons voir qu'à chaque espace hyperbolique on peut ajouter un bord sur lequel le produit de Gromov s'étend naturellement.
	L'espace hyperbolique muni de son bord est en quelque sorte une \og compactification \fg, puisque nous verrons que cette union est compacte si l'espace est propre et vérifie une propriété supplémentaire.  
	
	Soit $X$ un espace Gromov-hyperbolique.
	
	\begin{define} \label{def_bord}
		On dit qu'une suite $(x_i) \in X^\N$ est \defi{convergente} si l'on a
		\[ \lim_{i,j \rightarrow \infty} (x_i | x_j) = \infty. \]
		On définit le \defi{bord} $\bord A$ d'une partie $A$ l'espace hyperbolique $X$ comme quotient de l'ensemble des suites convergentes
		\[ \bord A := \EnsembleQuotient{\{(x_i) \in A^\N | \lim_{i,j \rightarrow \infty}(x_i | x_j) = \infty \}}{\sim} \]
		par la relation d'équivalence $\sim$ définie par
		\[ (x_i)_{i \in \N} \sim (y_j)_{j \in \N} \quad \text{ si } \quad \lim_{i,j \rightarrow \infty} (x_i | y_j) = \infty. \]
	\end{define}
	
	\begin{rem}
		Si l'on n'avait pas supposé l'espace $X$ Gromov-hyperbolique, la relation $\sim$ définie ci-dessus ne serait pas nécessairement transitive. Par exemple elle ne l'est pas pour $X = \R^2$.
	\end{rem}
	
	\begin{define}
		On appelle \defi{adhérence de Gromov} d'une partie $A \subseteq X \cup \bX$ l'ensemble
		\[ \overline{A} := \Adh{A} \cup \bord A, \]
		où $\Adh{A}$ est l'adhérence de $A$ dans $X$.
	\end{define}
	
	\begin{rem}
		\begin{enumerate}
			\item Le groupe $\Isom(X)$ agit naturellement sur le bord $\bX$.
			\item Le produit de Gromov s'étend naturellement à l'adhérence de Gromov :
				\[ (\xi | \eta) := \sup_{x_i \to \xi \atop y_j \to \eta} \liminf_{i,j \to \infty} (x_i | y_j), \]
				pour $\xi, \eta \in \overline{X}$.
			\item Un espace métrique $X$ est propre à bord compact si et seulement si l'adhérence $\overline{X}$ est compacte.
		\end{enumerate}
	\end{rem}
	
	Le premier point ci-dessus permet de généraliser la notion de suite convergentes à toute suite de $\overline{X}$.
	Cela définit la topologie que l'on considère sur l'espace $\overline{X}$.
	Voir aussi dans la partie \og Compacité de l'adhérence \fg ci-dessous pour une caractérisation de la topologie de l'adhérence de Gromov de $X$.

	\begin{define}
		Soit $X$ un espace métrique.
		\'Etant donné un réel $C > 0$, on dit que deux parties $A$ et $B$ de $X $ sont \defi{$C$-disjointes} si elles sont d'adhérence dans $X$ disjointes et que l'on a
		\[ \sup_{(a,b) \in A \times B} (a | b) < C. \]
		On dit que les parties $A$ et $B$ sont \defi{Gromov-disjointes} s'il existe une constante $C > 0$ telle que les parties $A$ et $B$ sont $C$-disjointes.
		Ces notions se généralisent de façon claire à des partie de l'espace $\cX$ quand l'espace $X$ est Gromov-hyperbolique.
	\end{define}
	
	\begin{rem}
		Si l'espace $X$ est propre et à bord compact, alors des parties $A$ et $B$ de $\cX $ sont Gromov-disjointes si et seulement si leur adhérences de Gromov $\overline{A}$ et $\overline{B}$ sont disjointes, mais ceci est faux en général.
	\end{rem}
	
	Dire que deux partie $A$ et $B$ sont Gromov-disjointes revient à dire qu'elles sont disjointes en tant qu'ensembles, et \og disjointes en l'infini vues du point base $o$ \fg.
	Cette notion permettra de définir ce qu'est un semi-groupe contractant (voir \ref{contract}) et ce qu'est un semi-groupe de Schottky (voir \ref{schottky}).
	
	\subsubsection*{Compacité de l'adhérence} \label{scompacite}
	
	Nous aurons besoin de la compacité de l'adhérence $\overline{X}$ pour contrôler la façon dont l'entropie part à l'infini (voir la sous-section \ref{supp}). 
	
	On muni l'adhérence de Gromov $\overline{X}$ d'un espace Gromov-hyperbolique $X$, de la base de voisinages formée des boules ouvertes $B(x, r)$ de $X$ et des boules
	\[ \beta(\xi, r) := \{ x \in \overline{X} | (x | \xi) > -\log(r) \}, \]
	pour les points $\xi \in \bX$. Cela définit bien la même topologie que celle donnée par les suites convergentes.
	
	Voici une condition sur l'espace $X$ qui sert à avoir la compacité de l'adhérence :
	\begin{define}
		Soit $X$ un espace métrique de point base $o$.
		On dit que l'espace $X$ est $C$-étoilé si pour tout point $x \in X$, il existe une $C$-géodésique de $o$ à $x$, c'est-à-dire une suite finie $(x_k)_{k = 1}^n$ d'éléments de $X$ telle que
		$x_0 = o$, $x_n = x$,
		\[ d(x_k, x_{k+1}) \leq C, \]
		\[ d(x_i, x_j) \leq \abs{d(o, x_i) - d(o, x_j)} + C, \]
		pour tous $0 \leq k \leq n-1$ et $0 \leq i, j \leq n$.
	\end{define}
	
	\begin{rem}
		Un espace géodésique est $C$-étoilé pour tout $C > 0$.
	\end{rem}
	
	\begin{ex}
		$\Z^n$ muni de la métrique euclidienne est $\sqrt{n}$-étoilé.
	\end{ex}
	
	Dans cet exemple, il suffit de considérer les points $x_i$ de $\Z^n$ qui sont à distance inférieure ou égale à $\sqrt{n}/2$ du segment $[o, x]$ dans $\R^n$.
	
	\begin{prop} \label{compacite}
		Soit $X$ un espace Gromov-hyperbolique, propre, et $C$-étoilé, alors son adhérence de Gromov $\overline{X}$ est compacte.
	\end{prop}
	
	\begin{proof}
		Il suffit de suivre la preuve du lemme 7.3 dans \cite{hol}.
	\end{proof}
		
	
	\begin{ex}
		L'ensemble $\N \times [0, \infty[$ muni de la métrique
		\[
			\left\{
			\begin{array}{ccl}
				d((i,x), (j,y)) &=& i+j+x+y \quad \text{ si } i \neq j \\
				d((i,x), (i, y)) &=& \abs{x - y}
			\end{array}
			\right.
		\]
		est Gromov-hyperbolique, propre, et d'adhérence non compacte (voir figure \ref{fig_ex-noncompact}).
	\end{ex}
	
	\begin{figure}[H]
		\centering
		\caption{Exemple d'espace Gromov-hyperbolique propre dont l'adhérence n'est pas compacte.} \label{fig_ex-noncompact}
		\includegraphics{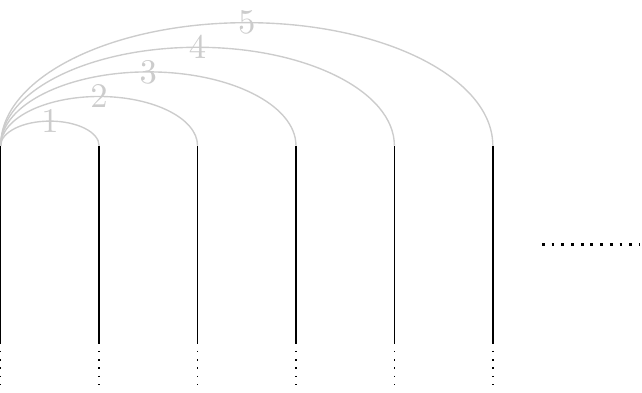}
	\end{figure}
	
	Sur la figure ci-dessus, les arêtes grises ne font pas parties de l'espace et indiquent des distances. 
		
	\subsection{Action d'un semi-groupe d'isométries sur l'espace hyperbolique}
	
	\'Etant donné un ensemble d'isométries d'un espace métrique $X$, on définit son exposant critique, respectivement son entropie, qui correspondent aux vitesses auxquelles croît l'orbite d'un point sous l'action de l'ensemble d'isométries, d'un point de vue de comptage, respectivement d'un point de vue de l'espace occupé. \\
	
	\begin{figure}[H]
		\centering
		\caption{Action du groupe $SL(2,\Z)$ sur le demi-plan de Poincaré $\hyp$}
		\includegraphics{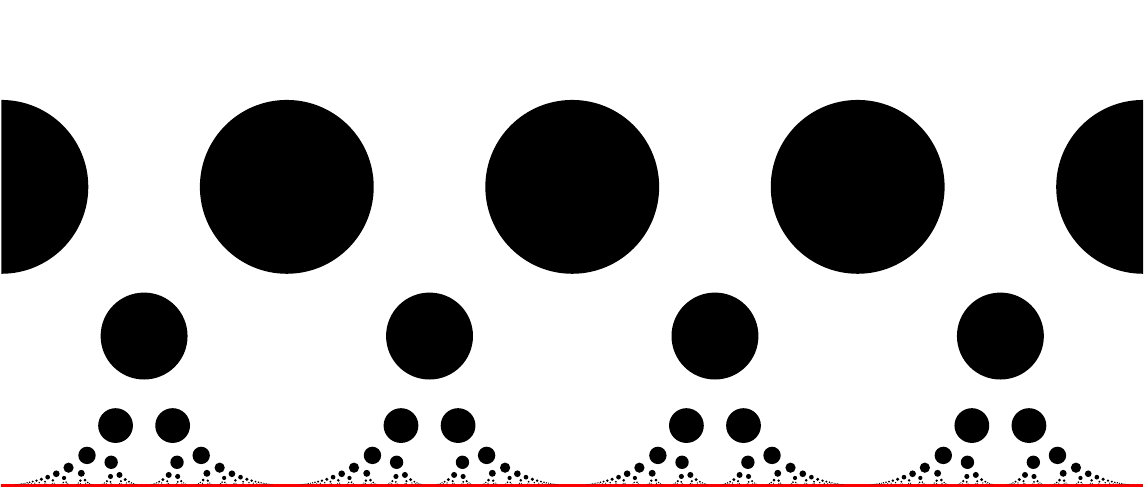}
	\end{figure}
	
	Soit $X$ un espace métrique de point base $o$.
	
	\begin{define}
		On appelle \defi{exposant critique} d'une partie $A \subset \Isom(X)$, l'exposant critique de l'orbite $Ao \subseteq X$ :
		\[ \delta_A := \delta_{Ao}. \]
	\end{define}
	
	\begin{define}
		On dit qu'une partie $A \subseteq \Isom(X)$ est \defi{séparée} (respectivement \defi{$\epsilon$-séparée}) si l'orbite $Ao \subseteq X$ est séparée (respectivement $\epsilon$-séparée), et si pour toutes les isométries $\gamma \neq \gamma' \in A$, on a $\gamma o \neq \gamma' o$.
	\end{define}
	
	\begin{rem}
		En général, l'exposant critique dépend du point $o$ choisi (par exemple pour un groupe elliptique dense), mais l'exposant critique d'une partie séparée de $\Isom(X)$ n'en dépend pas. 
	\end{rem}
	
	\begin{rem}
		Si une partie $A \subseteq \Isom(X)$ est séparée et que l'orbite $A.o$ est infinie, alors 
		l'exposant critique de la partie $A$ est aussi l'exposant critique de la série de Poincaré $P_s$ de $A$ :
			\[ P_s := \sum_{\gamma \in A} e^{-s d(o, \gamma o)}. \]
			C'est-à-dire que la série diverge pour $s < \delta_A$ et converge pour $s > \delta_A$.
		Voir \cite{mercat} pour plus de détails.
	\end{rem}
	
	\begin{define}
		On dit qu'une partie $A \subseteq \Isom(X)$ est une partie \defi{couvrante} (respectivement \defi{$\epsilon$-couvrante}) de $B \subseteq \Isom(X)$ si l'orbite $Ao \subseteq X$ est une partie couvrante (respectivement $\epsilon$-couvrante) de $Bo$.
	\end{define}
	
	\begin{define}
		On appelle \defi{entropie} d'une partie $A \subseteq \Isom(X)$ l'entropie de l'orbite $Ao \subseteq X$ :
		\[ h_A := h_{Ao}. \]
	\end{define}
	
	\begin{rem}
		Quand $X = \hypn$, l'entropie d'une partie $A \subseteq \Isom(X)$ est égale à l'entropie volumique de l'orbite $A B(o,1)$ d'une boule $B(o,1)$.
	\end{rem}
	
	Voici quelques propriétés de l'entropie.
	\begin{props} \textbf{Propriétés de l'entropie} \label{ppteg} \\ 
		Soit $X$ un espace métrique de point base $o$,
		et soit une partie $A \subseteq \Isom(X)$.
		On a les propriétés :
		\begin{enumerate}
			\item L'entropie ne dépend pas du point base $o$ choisi.
			\item On a l'inégalité \label{hmaj}
				$h_A \leq \delta_A$.
			\item Si la partie $A o$ est séparée, alors on a l'égalité $h_A = \delta_A$. \\ \label{sep}
				\noindent Ceci est en particulier le cas si $A$ est un groupe discret. 
			\item L'entropie est croissante : si $A \subseteq B \subseteq \Isom(X)$, alors on a \label{hc}
				\[ h_A \leq h_B \leq h_{\Isom(X)}. \]
			\item Si l'espace $X$ est propre, alors pour toute partie $S$ couvrante de $A$, on a l'inégalité \label{couv}
				$h_S \geq h_A$.
			\item Si $A$, $B$ et $C$ sont des parties de $\Isom(X)$ telles que l'on ait $A \subseteq B \cup C$, \label{hu}
				alors on a $h_A \leq \max\{ h_B, h_C \}$.
			\item Pour toute isométrie $\gamma \in \Isom(X)$, on a \label{hm}
				$h_{\gamma A} = h_A = h_{A \gamma}$.
			\item Si l'espace $X$ est propre, alors en posant $A_{> n} := \{ \gamma \in A | d(o, \gamma o) > n \}$, on a \label{hf}
				\[ h_{A_{> n}} = h_A. \]
		\end{enumerate}
	\end{props}
	
	\begin{rem} \label{epsc}
		Pour toute partie $Y \subseteq X$ d'un espace métrique $X$, il existe une partie $S$ $\epsilon$-séparée et $2\epsilon$-couvrante de $Y$.
		Ainsi, quand l'espace $X$ est propre, pour calculer l'entropie d'une partie $A$ de $\Isom(X)$, on est ramené à calculer l'exposant critique d'une partie séparée de $A$.
	\end{rem}
	
	\begin{proof}
		Les points 2, 3 et 4 sont évidents à partir des propriétés \ref{pptex}.
		Le point 1 découle du fait que l'exposant critique d'une partie séparée ne dépende pas du point base $o$ choisi.
		
		Montrons le point 5.
		Soit $S$ une partie couvrante de $A$
		et soit $S'$ une partie séparée de $A$.
		Montrons que l'on a $\delta_{S'} \leq \delta_{S}$.

		Cela va découler du lemme suivant.
		\begin{lemme} \label{isc}
			Soit $X$ un espace métrique propre de point base $o$, et $Y \subseteq \Isom(X)$ une partie du groupe d'isométries.
			Alors, pour tous réels $r > 0$ et $r' > 0$, il existe une constante $C_{r, r'}$ telle que pour toute partie $r$-couvrante $S$ de $Y$, pour toute partie $r'$-séparée $S'$ de $Y$, et pour toute partie $Z$ de $Y$, on ait l'inégalité 
			\[ \# S'o \cap Z \leq C_{r, r'} \# So \cap Z^r, \]
			où $Z^r := \{ x \in X | \exists y \in Z : d(x, y) \leq r \}$ est le $r$-voisinage fermé de $Z$.
			
		\end{lemme}
		
		\begin{proof}
			
			Soit $S''$ une partie $r'$-couvrante et séparée de la boule $B(o, r+r')$. Posons
			$C_{r, r'} := \# S''$ son cardinal, qui est fini par propreté.
			
			Pour toute partie $r'$-séparée $S'$ de $X$ on a alors
			\[ \# B(o, r) \cap S' \leq \# S'' = C_{r, r'}. \]
			Et comme pour toute isométrie $\gamma \in Y$, la partie $\gamma^{-1} S'$ est encore $r'$-séparée, on a aussi
			\[ \# B(\gamma o, r) \cap S' =  \# B(o, r) \cap \gamma^{-1} S' \leq C_{r, r'}. \]
			On a donc, pour toute partie $S$ $r$-couvrante de $Y$, et $S'$ partie $r'$-séparée de $Y$,
			\[ \# Z \cap S' o = \# Z \cap \bigcup_{\gamma \in S} B(\gamma o, r) \cap S' o \leq C_{r, r'} \# Z^r \cap S o. \]
		\end{proof}
		
		D'après le lemme ci-dessus, il existe une constante $C$ telle que
		\[ \#(B(o, n) \cap S' o) \leq C \cdot \#(B(o,n+C) \cap S o). \]
		Ainsi, on a $\delta_{S'} \leq \delta_{S}$ en passant à la limite.
		On obtient alors l'inégalité souhaitée $h_A \leq \delta_S$ en passant à la borne supérieure sur les parties séparées $S'$ de $A$.

		Montrons maintenant le point \ref{hu}.
		Si $S$ est une partie séparée de $A$, alors on a
		\[ \sum_{\gamma \in A \cap S} e^{-s d(o, \gamma o)} \leq \sum_{\gamma \in B \cap S} e^{-s d(o, \gamma o)} + \sum_{\gamma \in C \cap S} e^{-s d(o, \gamma o)}, \]
		pour tout réel $s$.
		Donc
		\[ \max\{ \sum_{\gamma \in B \cap S} e^{-s d(o, \gamma o)}, \sum_{\gamma \in C \cap S} e^{-s d(o, \gamma o)} \} = \infty,\]
		dès que $\sum_{\gamma \in A \cap S} e^{-s d(o, \gamma o)} = \infty$,
		et donc $\max\{ h_B, h_C \} \geq h_A$.
		
		Montrons le point \ref{hm}.
		Soit $\gamma_0 \in \Isom(X)$, et soit $S$ une partie $r$-séparée et couvrante de $A$, pour un réel $r > d(o, \gamma o)$.
		Pour tout $\gamma \in A$, l'inégalité triangulaire donne $d(o, \gamma \gamma_0 o) \leq d(o, \gamma o) + d(o, \gamma_0 o)$.
		Ainsi la partie $S \gamma_0$ est encore une partie séparée et couvrante de $A \gamma_0$, et on a
		\[ \# \{ \gamma \in A \cap S | d(o, \gamma o) \leq n \} \leq \# \{ \gamma' \in (A \cap S) \gamma_0 |\ d(o, \gamma' o) \leq n + d(o, \gamma_0 o) \}, \]
		d'où
		\begin{align*}
			h_A			&= \delta_{A \cap S} \\
						&= \limsup_{n \to \infty} \frac{1}{n} \log\left( \# \{ \gamma \in A \cap S |\ d(o, \gamma o) \leq n \} \right) \\
						&\leq \limsup_{n \to \infty} \frac{1}{n} \log\left( \# \{ \gamma \in (A \cap S) \gamma_0 |\ d(o, \gamma o) \leq n + d(o, \gamma_0 o) \} \right) \\
						&= \limsup_{n \to \infty} \frac{1}{n + d(o, \gamma_0 o)} \log\left( \# \{ \gamma \in (A \cap S) \gamma_0 |\ d(o, \gamma o) \leq n + d(o, \gamma_0 o) \} \right) \\
						&= \delta_{A \gamma_0 \cap S \gamma_0} \\
						&= h_{A \gamma_0}.
		\end{align*}
		L'autre inégalité $h_{A \gamma_0} \leq h_{A}$ s'obtient par symétrie, en remplaçant l'élément $\gamma_0$ par $\gamma_0^{-1}$ et la partie $A$ par $A \gamma_0$. L'égalité $h_{\gamma A} = h_A$ s'obtient de la même façon.
		
		Le point \ref{hf} s'obtient en remarquant que la partie $\Gamma \backslash \Gamma_{> n} o$ est bornée et que donc son intersection avec toute partie séparée est finie par propreté.
	\end{proof}
	
	Voici un exemple de semi-groupe d'exposant critique strictement supérieur à son entropie à cause d'un phénomène de chevauchements qui n'existe pas pour les groupes.
	
	\begin{figure}[H]
		\centering
		\caption{Orbite d'un point sous l'action du semi-groupe de l'exemple \ref{exee} dans le disque de Poincaré.}
		\includegraphics{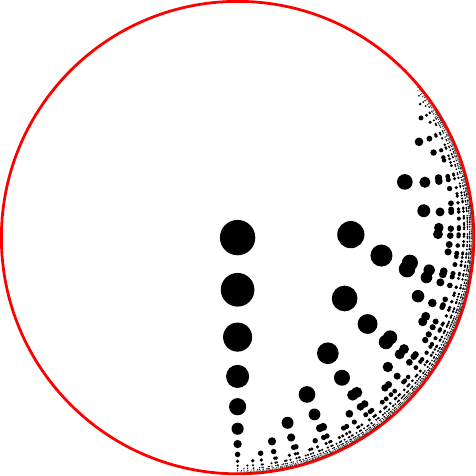}
	\end{figure}
	
	\begin{ex} \label{exee} 
		Le sous-semi-groupe de $SL(2,\R)$ engendré par les matrices
		$
			\begin{pmatrix}
				
				\sqrt{\frac{2}{\pi}} & 0 \\
				0 & \sqrt{\frac{\pi}{2}}
			\end{pmatrix}
		$ et
		$
			\begin{pmatrix}
				\sqrt{\frac{2}{\pi}} & 1 \\
				0 & \sqrt{\frac{\pi}{2}}
			\end{pmatrix}
		$
		et agissant sur le disque de Poincaré $\D$, a pour exposant critique $\delta = \frac{\log(2)}{\log(\pi/2)} > 1$ et a pour entropie $1$.
		En particulier il n'est pas séparé.
	\end{ex}
	
	L'exposant critique du semi-groupe l'exemple ci-dessus s'obtient facilement puisque le semi-groupe est libre, et parce-que c'est un semi-groupe de développement $\beta$-adique, ce qui permet d'avoir que la norme d'un élément est comparable à sa longueur en les générateurs :
	\begin{align*}
		\delta_\Gamma &= \limsup_{n \to \infty} \frac{1}{n} \log(\# \{ \gamma \in \Gamma | d(o, \gamma o) \leq n \}) \\
					&= \limsup_{n \to \infty} \frac{1}{n \log(\pi/2)} \log(\# \{ \gamma \in \Gamma \text{ de longueur } n \} ) \\
					&= \frac{\log(2)}{\log(\pi/2)}.
	\end{align*}
	Le fait que l'entropie vaille $1$ (le maximum pour un semi-groupe d'isométries de $\D$) découle du théorème~\ref{cpaulin}, parce-qu'il est facile de montrer que l'ensemble limite de ce semi-groupe est un segment de longueur non nulle (vu dans $\R$), et parce-que l'ensemble limite radial est égal à l'ensemble limite tout entier par la proposition~\ref{pradial}, puisque le semi-groupe est contractant et de type fini.
	
	
	\subsection{Action sur le bord}
	Les isométries d'un espace $X$ Gromov-hyperbolique agissent naturellement sur le bord $\bX$. Pour étudier cette action sur le bord, commençons par définir l'ensemble limite du semi-groupe.
	
	\begin{define}
		Soit $X$ un espace Gromov-hyperbolique de point base $o$, et soit $\Gamma$ un semi-groupe d'isométries de $X$.
		On appelle \defi{ensemble limite} du semi-groupe $\Gamma$, et on note $\Lambda_\Gamma$ le bord de l'orbite $\Gamma o$ : $\Lambda_\Gamma := \bord (\Gamma o)$.
	\end{define}
	
	\begin{rem}
		L'ensemble limite ne dépend pas du point base $o$ choisi.
	\end{rem}
	
	On va maintenant définir une partie de l'ensemble limite appelée ensemble limite radial, dont on saura mieux contrôler la dimension visuelle.
	
	\begin{define}
		On dit qu'une partie $A \subseteq X$ d'un espace métrique $X$ est une \defi{sous-quasi-géodésique} s'il existe une constante $C$ telle que l'on ait
		\[ d(x, y) + d(y, z) \leq d(x, z) + C \]
		pour tous $x$, $y$ et $z$ dans $A$ tels que $\max\{d(x, y), d(y, z)\} \leq d(x, z)$.
	\end{define}
	
	\begin{rem}
		Si l'espace $X$ est géodésique, alors dire qu'une partie $A$ est une sous-quasi-géodésique revient à dire qu'il existe une géodésique dont tout point de $A$ est à distance bornée.
	\end{rem}
	
%
%
		
	\begin{define}
		Soit $X$ un espace Gromov-hyperbolique de point base $o$ et $\Gamma$ un semi-groupe d'isométries de $X$.
		On appelle \defi{ensemble limite radial} (ou \defi{ensemble limite conique}) du semi-groupe $\Gamma$, et on note $\Lambda_\Gamma^c$ l'ensemble :
		\[  \Lambda^c_\Gamma := \EnsembleQuotient{\{(x_i) \in (\Gamma o)^\N | \lim_{i,j \rightarrow \infty}(x_i | x_j) = \infty \text{ et } \{x_i\}_{i \in \N} \text{ est une sous-quasi-géodésique} \}}{\sim} \subseteq \Lambda_\Gamma, \]
		où $\sim$ est la relation d'équivalence vue dans la définition du bord d'un espace Gromov-hyperbolique.
	\end{define}
	
	\begin{rem}
		L'ensemble limite radial est une partie de l'ensemble limite qui ne dépend pas non plus du point base $o$ choisi.
	\end{rem}
	
	\begin{rem} \label{auto-sim}
		Si $\Gamma$ est un semi-groupe de type fini d'isométries d'un espace Gromov-hyperbolique, alors son ensemble limite est auto-similaire :
		\[ \Lambda_\Gamma = \bigcup_{g \text{ générateur}} g \Lambda_\Gamma. \]
	\end{rem}
	
	\begin{proof}
		On a clairement l'inclusion $\bigcup_{g \text{ générateur}} g \Lambda_\Gamma \subseteq \Lambda_\Gamma$. Montrons l'autre inclusion.
		Soit $\xi \in \Lambda_\Gamma$ et soit $(\gamma_n)_{n \in \N} \in \Gamma^\N$ une suite d'éléments de $\Gamma$ telle que l'on ait
		\[ \lim_{n \to \infty} \gamma_n o = \xi. \]
		(Ceci est une notation qui signifie $\lim_{n \to \infty} (\gamma_n o | \xi ) = \infty.$) \\
		Comme le semi-groupe $\Gamma$ est de type fini, il existe un générateur $g$ tel que l'on ait une sous-suite $(\gamma_{\phi(n)}) \in (g \Gamma)^\N$.
		On a alors $\xi = \lim_{n \rightarrow \infty} \gamma_{\phi(n)} o \in \Lambda_{g \Gamma} = g \Lambda_{\Gamma}.$
	\end{proof}
	
	\subsection{Les ensembles $X_\gamma$}
	
	On va associer à chaque isométrie $\gamma$ d'un espace métrique $X$, une partie $X_\gamma$ de l'espace $X$ qui correspond à un domaine dans lequel l'élément $\gamma$ contracte.
	Ceci nous servira dans toute la suite.
	
	\begin{define}
		Soit $X$ un espace métrique.
		Pour $\gamma \in \Isom(X)$, on définit un domaine $X_\gamma \subseteq X$ par
		\[ X_\gamma := \{ x \in X | (x | \gamma o) \geq \frac{1}{2} d(o, \gamma o) \}. \]
	\end{define}
	
	\begin{figure}[H]
		\centering
		\caption{$X_\gamma$}
		\includegraphics{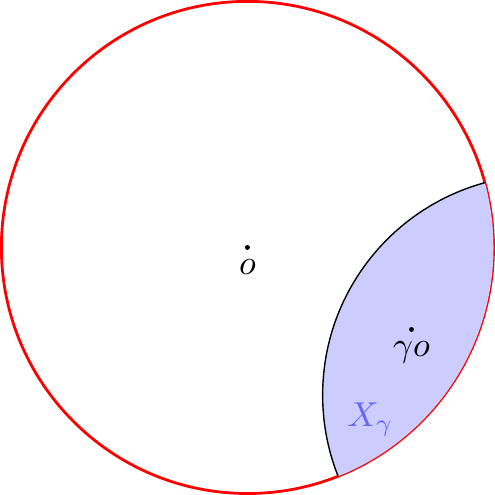}
	\end{figure}
	
	\begin{rem} \label{remXg}
		L'inégalité $(x | \gamma o) \geq \frac{1}{2} d(o, \gamma o)$ est équivalente à
		\[ d(x, \gamma o) \leq d(x, o).\]
		Les éléments de $X_\gamma$ sont donc les points de $X$ qui sont plus près de $\gamma o$ que de $o$.
	\end{rem}
	
	Le lemme suivant indique que l'ensemble $X_\gamma$ est un domaine dans lequel $\gamma$ contracte, et qu'il est de taille petite vue de $o$ quand l'élément $\gamma$ est de grande norme.
	
	\begin{lemme} \label{lXg}
		Soit $X$ un espace métrique.
		Pour tout $\gamma \in \Isom(X)$, on a
		\begin{enumerate}
			\item $\gamma(X \backslash X_{\gamma^{-1}}) \subseteq X_\gamma$,
			\item Si l'espace $X$ est $\delta$-hyperbolique de point base $o$, alors pour tout $(x,y) \in (\overline{X_\gamma})^2$, on a
			\[ (x|y) \geq \frac{1}{2} d(o, \gamma o) - \delta. \]
		\end{enumerate}
	\end{lemme}
	
	\begin{proof}
		\begin{enumerate}
			\item On a
				\begin{align*}
					x \in X \backslash X_{\gamma^{-1}} &\eq d(x, \gamma^{-1} o) > d(x, o) \quad  \text{ par la remarque \ref{remXg},}  \\
												&\eq d(\gamma x, o) > d(\gamma x, \gamma o) \quad \text{ parce-que $\gamma$ est une isométrie}  \\
												&\imp \gamma x \in X_\gamma.												
				\end{align*}
			\item Soient $x$ et $y$ dans $\overline{X_\gamma}$.
				Par $\delta$-hyperbolicité, on a
				\[ (x|y) \geq \min\{ (x | \gamma o) , (y | \gamma o) \} - \delta. \]
				Or, par définition de $X_\gamma$ on a pour tout $x$ dans $\overline{X_\gamma}$
				\[ (x | \gamma o) \geq \frac{1}{2} d(o, \gamma o), \]
				d'où le résultat.
		\end{enumerate}
	\end{proof}
	
	\begin{define}
		On dit qu'une isométrie $\gamma \in \Isom(X)$ d'un espace métrique $X$ est \defi{contractante} si les domaines $X_\gamma$ et $X_{\gamma}^{-1}$ sont Gromov-disjoints.
	\end{define}
	
	\begin{figure}[H]
		\centering
		\caption{Une isométrie contractante.}
		\includegraphics{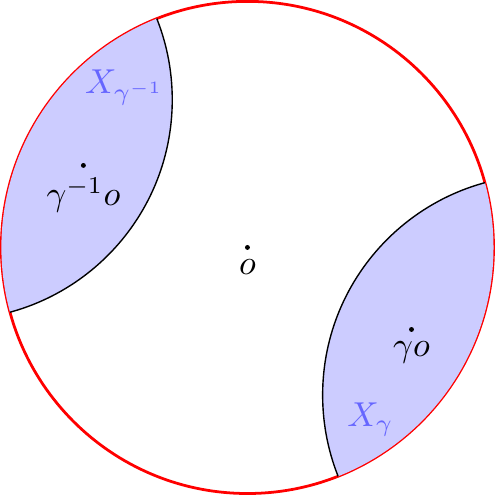}
	\end{figure}
	
	\begin{rem}
		Si une isométrie $\gamma$ est contractante, alors la partie
		$\gamma^{-1}X_\gamma \backslash X_{\gamma}$ est un domaine fondamental pour l'élément $\gamma$.
	\end{rem}

	Voici un critère de contraction.
	\begin{lemme}[Critère de contraction] \label{non_cont} \label{seph}
		Soit $X$ un espace $\delta$-hyperbolique de point base $o$, et $\gamma$ une isométrie de $X$.
		S'il existe deux points $x$ et $x'$ dans $\overline{X_\gamma} \cup \overline{X_{\gamma^{-1}}}$, tels que
		\[ (x | x') < \frac{1}{2} d(o, \gamma o) - 3\delta, \]
		alors l'isométrie $\gamma$ est contractante.
	\end{lemme}
	
	\begin{proof}
		Par contraposée, supposons que l'isométrie $\gamma$ n'est pas contractante. Les ensembles $X_\gamma$ et $X_{\gamma^{-1}}$ ne sont alors pas Gromov-disjoints.
		Soit alors $y \in X_{\gamma^{-1}}$ et $y' \in X_\gamma$ tels que $(y | y') \geq \frac{1}{2} d(o, \gamma o) - \delta$.
		
		Soient aussi $x$ et $x'$ deux points de l'union $\overline{X_\gamma} \cup \overline{X_{\gamma^{-1}}}$.
		Si les points $x$ et $x'$ étaient tous les deux dans $\overline{X_\gamma}$ ou tous les deux dans $\overline{X_{\gamma^{-1}}}$, on aurait l'inégalité $(x | x') \geq \frac{1}{2} d(o, \gamma o) - \delta $ par le lemme \ref{lXg}.
		On peut donc supposer que l'on a par exemple $x \in \overline{X_{\gamma^{-1}}}$ et $x' \in \overline{X_\gamma}$.
		Par $\delta$-hyperbolicité et par le lemme \ref{lXg},
		on a alors
		\[ (x | x') \geq \min\{ (x | y), (y | y') , (y' | x') \} - 2\delta \geq \frac{1}{2} d(o, \gamma o) - 3\delta. \]
	\end{proof}
	
	Voici une propriété des isométries contractantes.
	\begin{lemme} \label{hn}
		Soit $X$ un espace métrique propre, et soit $h \in \Isom(X)$ une isométrie contractante.
		Alors on a
		\[ \lim_{n \to \infty} d(o, h^n o) = \infty. \]
		De plus, si l'espace $X$ est $\delta$-hyperbolique, alors il existe un entier $n_0$ tel que pour tout $n \geq n_0$, les parties $X_{h^n}$ et $X_{h^{-n}}$ soient $(M + 2\delta)$-disjointes,
		où $M := sup_{(x, x') \in X_h \times X_{h^{-1}}} (x | x')$.
	\end{lemme}
	
	\begin{proof}
		Montrons que l'on a $\lim_{n \to \infty} d(o, h^n o) = \infty$.
		Pour cela, choisissons un réel $\epsilon > 0$ tel que l'on ait
		\[ B(o, \epsilon) \subseteq X \backslash (X_h \cup X_{h^{-1}}). \]
		Cela est possible puisque la partie $X \backslash (X_h \cup X_{h^{-1}})$ est ouverte et contient $o$.
		
		Pour tous entiers $n \neq m \in \Z$, les boules $h^n B(o, \epsilon)$ et $h^m B(o, \epsilon)$ sont disjointes.
		En effet, quitte à tout composer à gauche par $h^{-m}$, on se ramène à $m = 0$. On a alors $B(h^n o, \epsilon) = h^n B(o, \epsilon) \subseteq X_h \cup X_{h^{-1}}$.
		Ainsi, l'orbite $(h^n o)_{n \in \Z}$ est une partie $\epsilon$-séparée de $X$.
		Par propreté, on a donc bien $\lim_{n \to \infty} d(o, h^n o) = \infty$.
		
		Montrons que les parties $X_{h^n}$ et $X_{h^{-n}}$ sont $(M + 2\delta)$-disjointes.
		Par $\delta$-hyperbolicité, pour $x \in X_{h^n}$ et $x' \in X_{h^{-n}}$, on a
		\[ (h^n o | h^{-n} o) \geq \min\{ (h^n o | x), (x | x'), (x', h^{-n} o) \} - 2\delta. \]
		Or, on a $h^n o \in X_h$ et $h^{-n} o \in X_{h^{-1}}$, donc on a $(h^n o | h^{-n} o) \leq M$. 
		D'autre part, par définition de $X_{h^n}$, on a $(x | h^n o) \geq \frac{1}{2} d(o, h^n o)$ et de même $(x' | h^{-n} o) \geq \frac{1}{2} d(o, h^n o)$.
		Comme on a $\lim_{n \to \infty} d(o, h^n o) = \infty$, quitte à choisir $n$ assez grand on a
		\[ d(o, h^n o) > 2M + 6\delta. \]
		On obtient donc l'inégalité
		$(x | x' ) \leq M + 2\delta$.
		Les ensembles $X_{h^n}$ et $X_{h^{-n}}$ sont alors disjoints, puisque sinon on aurait pour $y \in X_{h^n} \cap X_{h^{-n}}$, par le lemme \ref{lXg}, l'absurdité
		\[ M + 2\delta < \frac{1}{2} d(o, h^n o) - \delta \leq (y | y) \leq M + 2\delta. \]
		
		
%
		
%
	\end{proof}
	
	\section{Parties contractantes de Isom($X$)} \label{contract}
	
	Nous voyons ici la définition et des propriétés des parties contractantes.
	Il s'agit d'ensemble d'isométries qui contractent toutes dans la même direction, de façon contrôlée.
	Cette propriété est très pratique, puisqu'elle est stable par produit, et permet de contrôler la vitesse à laquelle les produits d'isométries tendent vers l'infini.
	Cela permet par exemple de garantir que l'ensemble limite d'un semi-groupe de type fini qui a cette propriété est radial (proposition~\ref{pradial}).
	
	\begin{define}
		Soit $X$ un espace métrique de point base $o$.
		On dit qu'une partie $A \subseteq \Isom(X)$ est \defi{contractante} s'il existe deux domaines Gromov-disjoints $X_-$ et $X_+$ de $X$ tels que l'on ait $ A (X \backslash X_-) \subseteq X_+$, et que l'on ait $o \in X \backslash \left( \overline{X_+} \cup \overline{X_-} \right)$.
	\end{define}
	
	Cette définition dépend à priori du point base $o$ choisi.
	
	\begin{figure}[H]
		\centering
		\caption{Un semi-groupe contractant.}
		\includegraphics{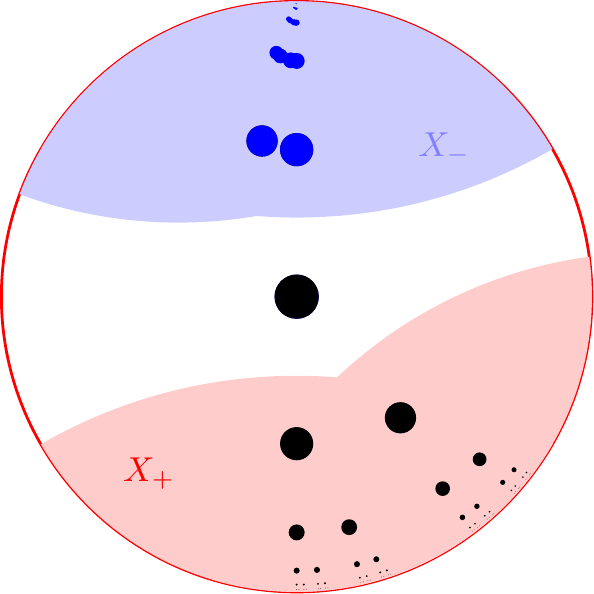}
	\end{figure}
	
	Dans cette partie, nous donnons quelques résultats sur les parties contractantes qui nous seront utiles par la suite.
	
	Voici quelques propriétés simples des parties contractantes.
	\begin{props} \textbf{Propriétés des parties contractantes} \label{pptc} \\
		Soit $X$ un espace métrique.
		\begin{enumerate}
			\item Si $\gamma$ est une isométrie contractante, alors le singleton $\{\gamma\}$ est contractant. \label{cont_imp_cont} 
			\item Si une partie $A \subseteq \Isom(X)$ est contractante, alors le semi-groupe engendré (c'est-à-dire l'ensemble des produits \emph{non vides} d'éléments de $A$) l'est aussi.
			
			\item Si une partie $A \subseteq \Isom(X)$ est contractante, alors l'inverse $A^{-1} := \{\gamma^{-1} | \gamma \in A\}$ l'est aussi.
		\end{enumerate}
	\end{props}
	
	\begin{proof}
		\begin{enumerate}
			\item Il suffit de prendre $X_- := X_{\gamma^{-1}}$ et $X_+ := X_\gamma$. 
			\item Si $\gamma$ et $\gamma'$ sont deux éléments d'une partie contractante $A$, alors on a
			\[ \gamma \gamma' (X \backslash X_-) \subseteq \gamma (X_+) \subseteq \gamma(X \backslash X_-) \subseteq X_+. \]
			
			\item Si $A$ est une partie contractante pour des domaines $X_+$ et $X_-$, alors la partie $A^{-1}$ est contractante pour les domaines $X_-$ et $X_+$ respectivement.
			En effet, on a pour tout $\gamma \in A$, l'inclusion $\gamma^{-1}(X \backslash X_+) \subseteq X_-$, qui se déduit de l'inclusion $\gamma(X \backslash X_-) \subseteq X_+$ par contraposée.
		\end{enumerate}
	\end{proof}
	
	\begin{rem}
		La réciproque du point~\ref{cont_imp_cont} est fausse : une isométrie $\gamma$ telle que l'ensemble $\{ \gamma \}$ est contractant n'est pas nécessairement contractante.
	\end{rem}
	
	Voici un critère de contraction.
	\begin{prop}[Critère de contraction] \label{crit_cont}
		Soit $X$ un espace $\delta$-hyperbolique de point base $o$, et soit une partie $A \subseteq \Isom(X)$ telle que l'on ait
		\[ \sup_{\gamma, \gamma' \in A} (\gamma^{-1} o | \gamma' o) < \frac{1}{2} \inf_{\gamma \in A} d(o, \gamma o) - 3\delta. \]
		Alors la partie $A$ est contractante.
	\end{prop}
	
	\begin{proof}
		Montrons que les domaines
		\[ X_+ := \bigcup_{\gamma \in A} X_\gamma \quad \text{et} \quad X_- := \bigcup_{\gamma \in A} X_{\gamma^{-1}} \]
		conviennent.
		
		On a déjà bien pour toute isométrie $\gamma \in A$, l'inclusion
		\[ \gamma(X \backslash X_-) \subseteq \gamma(X \backslash X_{\gamma^{-1}}) \subseteq X_\gamma \subseteq X_+ .\]
		
		Montrons maintenant que les domaines $X_+$ et $X_-$ sont Gromov-disjoints.
		Soient $x \in X_-$ et $x' \in X_+$. Il existe alors deux isométries $\gamma$ et $\gamma'$ de $X$ tels que l'on ait $x \in X_{\gamma^{-1}}$ et $x' \in X_{\gamma'}$.
		On a ensuite, par $\delta$-hyperbolicité,
		\[ M \geq (\gamma^{-1} o | \gamma' o) \geq \min\{ (\gamma^{-1} o | x), (x | x') , (x' | \gamma' o) \} - 2\delta, \]
		où $M := \sup_{\gamma, \gamma' \in A} (\gamma^{-1} o | \gamma' o)$.
		
		Or, on a $(\gamma^{-1} o | x) \geq \frac{1}{2} d(o, \gamma o)$ et $(\gamma' o | x') \geq \frac{1}{2} d(o, \gamma' o)$ par définition de $X_{\gamma^{-1}}$ et $X_{\gamma'}$.
		De plus, on a $M < \frac{1}{2} d(o, \gamma o) - 3\delta$ et $M < \frac{1}{2} d(o, \gamma' o) - 3\delta$ par hypothèse.
		On en déduit que l'on a
		\[ M \geq (x | x') - 2\delta. \]
		Pour vérifier que les domaines $X_+$ et $X_-$ sont bien Gromov-disjoints, il ne reste donc plus qu'à montrer qu'ils sont d'adhérences dans $X$ disjointes.
		On a
		
		\begin{align*}
			\inf_{(x, x') \in X_- \times X_+} d(x, x') &= \inf_{(x, x') \in X_- \times X_+} d(o, x) + d(o, x') - 2(x | x') \\
										&\geq \left( \frac{1}{2} \inf_{\gamma \in A} d(o, \gamma o) - \delta \right) + \left( \frac{1}{2} \inf_{\gamma \in A} d(o, \gamma o) - \delta \right) - 2(M + 2\delta) \\
										& > \left( M + 2\delta \right) + \left( M + 2\delta \right) - 2(M + 2 \delta) \\
										& = 0,
		\end{align*}
		puisque l'on a $d(o, x) = (x | x) \geq \frac{1}{2} d(o, \gamma o) - \delta$ par le lemme \ref{lXg} pour $x \in X_{\gamma^{-1}}$, et de même $d(o, x') \geq \frac{1}{2} d(o, \gamma' o) - \delta$ pour $x' \in X_{\gamma'}$.
		
		Les domaines $X_-$ et $X_+$ sont donc bien Gromov-disjoints.

		Pour finir, l'ensemble $X \backslash \left( \overline{X_+} \cup \overline{X_-} \right)$ contient bien le point base $o$, puisqu'il contient la boule ouverte $B(o, M + 2\delta + \epsilon)$ pour un réel $\epsilon > 0$ assez petit.
	\end{proof}
	
	Voici un lemme qui permet de contrôler la taille de l'image du domaine $X_+$ par les éléments d'un semi-groupe contractant.
	Cela sera utilisé dans la preuve du théorème \ref{thm2}.
	
	\begin{lemme} \label{cont1}
		Soit $X$ un espace métrique de point base $o$, soit une partie $X_+ \subseteq \overline{X}$ et soit une isométrie $\gamma \in \Isom(X)$ telles que l'on ait
		\[ C := \sup_{x \in X_+} (\gamma^{-1}.o | x) < \infty. \]
		Alors pour tous $x$ et $x' \in X_+$, on a l'inégalité
		\[ (\gamma x | \gamma x') \geq d(o, \gamma o) - 2C. \]
	\end{lemme}
	
	\begin{proof}
		Soient $x$ et $x'$ des éléments de $X_+$. On a
		\begin{align*}
			(\gamma x | \gamma x')	&= \frac{1}{2} \left[ d(\gamma x, o) + d(o, \gamma x') - d(\gamma x, \gamma x') \right] \\
								&= \frac{1}{2} [ d(x, o) + d(\gamma^{-1} o, o) - 2(\gamma^{-1} o | x) \\
								& \qquad + d(\gamma^{-1} o , o) + d(o, x') - 2(\gamma^{-1} o | x') - d(x, x') ] \\
								& = (x | x') - (\gamma^{-1} o | x) - (\gamma^{-1} o | x') + d(o, \gamma o) \\
								&\geq (x | x') + d(o, \gamma o) - 2C \\
								&\geq d(o, \gamma o) - 2C.
		\end{align*}
	\end{proof}
	
	Le lemme qui suit dit que l'on a l'égalité triangulaire à une constante près dans un semi-groupe contractant.
	
	\begin{lemme} \label{cet}
		Soit $X$ un espace métrique de point base $o$.
		Si $A$ est une partie contractante de $\Isom(X)$, alors pour tous éléments $\gamma$ et $\gamma'$ de $A$ on a
		\[ d(o, \gamma \gamma' o) \geq d(o, \gamma o) + d(o, \gamma' o) - 2M, \]
		où $M = \sup_{\gamma, \gamma' \in A} (\gamma^{-1} o | \gamma' o)$.
	\end{lemme}
	
	\begin{proof}
		On a l'égalité $d(o, \gamma \gamma' o) = d(o, \gamma o) + d(o, \gamma' o) - 2(\gamma^{-1} o | \gamma' o)$.
	\end{proof}
	
	Voici une propriété intéressante des parties contractantes qui permettra de montrer que l'ensemble limite d'un semi-groupe contractant de type fini est radial.
	
	\begin{lemme} \label{cqg}
		Si $A$ est une partie contractante du groupe d'isométries $\Isom(X)$ d'un espace métrique $X$ de point base $o$ et si $(\gamma_n)_{n \in \N}$ est une suite de $A^\N$, alors
		l'ensemble $\{ \gamma_0 \gamma_1 ... \gamma_n o \}_{n \in \N}$ est une sous-quasi-géodésique.
	\end{lemme}
	
	\begin{proof}[Preuve du lemme \ref{cqg}]
		Montrons que pour $a \leq b \leq c \in \N$, on a
		\[ d(\gamma_0 \gamma_1... \gamma_a o, \gamma_0 \gamma_1... \gamma_b o)
			+ d( \gamma_0 \gamma_1... \gamma_b o, \gamma_0 \gamma_1... \gamma_c o)
			\leq d(\gamma_0 \gamma_1... \gamma_a o, \gamma_0 \gamma_1... \gamma_c o) + M, \]
		où $M = \sup_{\gamma, \gamma' \in A} (\gamma^{-1} o | \gamma' o)$.
		On se ramène à $a = 0$, et l'on conclut avec le lemme \ref{cet} avec $\gamma = \gamma_0 \gamma_1 ... \gamma_b$ et $\gamma' = \gamma_{b+1} \gamma_{b+2} ... \gamma_c$.
	\end{proof}
	
	Le lemme qui suit dit que les isométries d'un semi-groupe contractant tendent vers l'infini avec leur longueur quand l'espace est propre.
	
	\begin{lemme} \label{cdv}
		Si $A$ est une partie contractante du groupe d'isométries $\Isom(X)$ d'un espace métrique propre $X$ de point base $o$ et si $(\gamma_n)_{n \in \N}$ est une suite de $A^\N$, alors on a
		\[ \lim_{n \to \infty} d(o, \gamma_0 \gamma_1 ... \gamma_n o) = \infty. \]
	\end{lemme}
	
	\begin{proof}
		L'ensemble $\{ \gamma_0 ... \gamma_n o \}_{n \in \N}$ est une partie séparée de $X$.
		En effet, il existe un réel $\epsilon > 0$ tel que la boule $B(o, \epsilon)$ soit incluse dans $X \backslash (X_+ \cup X_-)$, et ses images par les éléments $\gamma_0 ... \gamma_n$ sont disjointes, puisque pour toute isométrie $\gamma$ de $A$, on a l'inclusion $\gamma B \subseteq X_+$.
		Le résultat découle alors de la propreté de l'espace $X$.
	\end{proof}
	
	\begin{prop} \label{pradial}
		Soit $X$ un espace métrique propre et $\Gamma$ un semi-groupe d'isométries de $X$, contractant et de type fini.
		Alors l'ensemble limite de $\Gamma$ est radial :
		\[ \Lambda_\Gamma = \Lambda^c_\Gamma. \]
	\end{prop}
	
	\begin{proof}
		Soit $\xi$ un point de l'ensemble limite $\Lambda_\Gamma$.
		D'après la remarque \ref{auto-sim}, il existe un générateur $g_0$ du semi-groupe $\Gamma$ tel que $\xi \in g_0 \Lambda_\Gamma$.
		Par récurrence, il existe une suite $(g_i)_{i \in \N}$ telle que pour tout $n \in \N$, on ait
		\[ \xi \in g_0 ... g_n \Lambda_\Gamma. \]
		Puis par le lemme \ref{cdv}, on a $\lim_{n \to \infty} d(o, g_0 ... g_n o) = \infty$.
		On a donc $\xi = \lim_{n \to \infty} g_0 ... g_n o$ par le lemme \ref{cont1} appliqué à l'isométrie $g_0 ... g_n$ et à l'ensemble $X_+ = \Lambda_\Gamma \cup \{ o \}$.
		Or, le lemme \ref{cqg} affirme que l'ensemble $\{ g_0 ... g_n o \}_{n \in \N}$ est une sous-quasi-géodésique, donc le point $\xi$ est dans dans l'ensemble limite radial : $\xi \in \Lambda^c_{\Gamma}$. On a montré l'inclusion $\Lambda_\Gamma \subseteq \Lambda_\Gamma^c$. L'autre inclusion est claire.
	\end{proof}
	
	Le fait que le semi-groupe $\Gamma$ soit de type fini est une hypothèse nécessaire, comme le montre le contre-exemple suivant :
\begin{contr-ex}
	Soit $\Gamma$ le sous-semi-groupe de $SL(2,\Z)$ engendré par les matrices $\begin{pmatrix} 1 & n \\ n & n^2+1 \end{pmatrix}, n \geq 1$.
	Alors le semi-groupe $\Gamma$ est contractant, mais les ensembles limite et limite radial diffèrent : $\Lambda_\Gamma \neq \Lambda_\Gamma^c$.
\end{contr-ex}
En effet, l'enveloppe convexe $X_+ \subset \hyp$ de l'intervalle $[-\frac{1}{2}, 1]$ dans $\hyp$ (en identifiant le bord de $\hyp$ à $\R$ de façon usuelle) est envoyée dans celle de l'intervalle $[0, \frac{2}{3}]$,
et donc le semi-groupe est contractant.
Mais le point $0$ est dans l'ensemble limite, et n'est pas radial.
		
	\section{Construction d'une grosse partie contractante} \label{pthm1} 
	
	Le but de cette section est de démontrer le théorème :
	
	\begin{thm} \label{thm1}
		Soit $X$ un espace Gromov-hyperbolique propre à bord compact, et soit $\Gamma$ un semi-groupe d'isométries de $X$ dont l'ensemble limite $\Lambda_\Gamma$ contient au moins deux points.
		Alors il existe un sous-semi-groupe $\Gamma'$ de $\Gamma$, contractant et de même entropie :
		\[ h_{\Gamma'} = h_\Gamma. \]
	\end{thm}
	
	Ce théorème est l'étape principale de la preuve du théorème \ref{thm0}.
	
	La preuve du théorème \ref{thm1} repose sur l'étude du support du semi-groupe $\Gamma$, qui est une partie fermée et $\Gamma \times \Gamma^{-1}$-invariante de $\bX \times \bX$ décrivant les directions dans lesquelles il y a beaucoup d'isométries qui contractent et dilatent.
	On montre que ce support, s'il n'est pas réduit à deux points, fournit un sous-semi-groupe contractant d'entropie $h_\Gamma$.
	Et pour cela, on commence par montrer qu'il existe une isométrie contractante dans le semi-groupe $\Gamma$,
	et on effectue un ping-pong entre cet élément et une \og grosse \fg\ partie du semi-groupe $\Gamma$.
	Puis l'on montre que dans le cas où le support est réduit à deux points, le semi-groupe fixe un doublet de points au bord, ce qui permet de conclure rapidement si le semi-groupe ne fixe pas de point au bord.
	Enfin, on termine par le cas où le semi-groupe fixe un point au bord. Dans ce cas, on ne peut pas effectuer de ping-pong, mais l'existence d'une isométrie contractante nous permet de démontrer qu'il existe une proportion suffisante des éléments du semi-groupe qui contractent loin du point fixe. On fera cela en utilisant le lemme des tiroirs sur une partition en copies d'un domaine fondamental pour l'isométrie contractante. \\
	
	
	\subsection{Construction d'une isométrie contractante}
	Nous allons montrer qu'un semi-groupe d'isométries d'un espace Gromov-hyperbolique dont l'ensemble limite contient au moins deux points possède toujours une isométrie contractante.
	On a la proposition :
	
	\begin{prop} \label{econt}
		Soit $X$ un espace Gromov-hyperbolique, et $\Gamma$ un semi-groupe d'isométries de $X$ dont l'ensemble limite n'est pas réduit à un point.
		Alors $\Gamma$ contient une isométrie $h$ contractante.
		De plus, si $\eta$ est un point du bord de $X$, 
		alors on peut choisir $h$ tel que
		$\eta \not\in X_h$.
	\end{prop}
	
	L'idée de la preuve est de considérer deux grands éléments de $\Gamma$ qui contractent en deux endroits distincts (près de points distincts de l'ensemble limite).
	Le produit de ces deux élément est alors contractant parce-que ces deux éléments n'étant pas contractants (sinon il n'y a rien à démontrer) et ayant une grande norme,
	leur inverses contractent à nouveau près des mêmes points, et ainsi, le produit contracte depuis l'un des endroits vers l'autre.
	(voir figure \ref{fig_contract}.)
	
	\begin{figure}[H]
		\centering
		\caption{Construction d'une isométrie contractante à partir de deux isométries $\gamma$ et $\gamma'$ non contractantes.} \label{fig_contract}
		\includegraphics{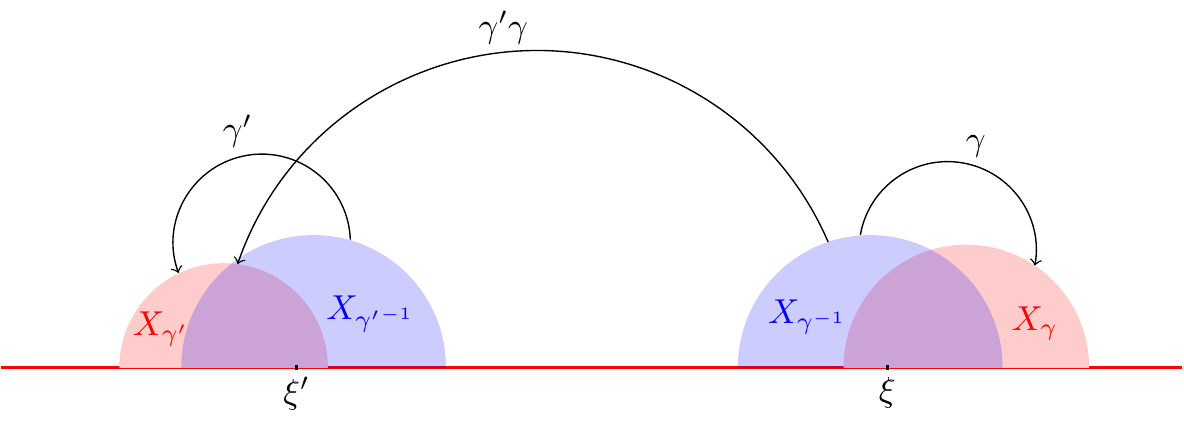}
%
%
%
%
%
%
%
	\end{figure}
	
	\begin{proof} 

		Supposons pour commencer que l'ensemble limite contienne au moins trois points.
		Soit $\eta \in \Lambda_\Gamma$, et soient $\xi \neq \xi'$ deux points distincts de l'ensemble limite $\Lambda_\Gamma$, et qui sont distincts du point $\eta$, et soient $\gamma$ et $\gamma'$ deux isométries de $\Gamma$ telles que l'on ait
		\[ (\xi | \gamma o) > 2 (\xi | \xi') + 10\delta + (\xi | \eta) \quad \text{et} \quad (\xi' | \gamma' o) > 2 (\xi | \xi') + 10\delta + (\xi' | \eta). \]
		On a alors également $d(o, \gamma o) \geq (\xi | \gamma o) > 2 (\xi | \xi') + 10 \delta + (\xi | \eta)$ et $d(o, \gamma' o) > 2 (\xi | \xi') + 10\delta + (\xi' | \eta)$. \\
		
		Si une des isométries $\gamma$ ou $\gamma'$ était contractante, alors elle conviendrait, et la preuve serait finie.
		En effet, on a $\eta \not\in X_\gamma$ et $\eta \not \in X_{\gamma'}$, puisque par $\delta$-hyperbolicité on a
		\[
			(\eta | \xi) \geq \min\{(\eta | \gamma o), (\gamma o | \xi)\} - \delta,
		\]
		et donc $(\gamma o | \eta) \leq (\eta | \xi) + \delta < \frac{1}{2} d(o, \gamma o)$, et de même avec $\gamma'$.
		On peut donc supposer que les éléments $\gamma$ et $\gamma'$ ne sont pas contractants.
		
		On a alors le lemme suivant.
		\begin{lemme}
			Les ensembles $X_{\gamma^{-1}}$ et $X_{\gamma'}$ sont $((\xi | \xi') + 2\delta)$-disjoints.
		\end{lemme}
		
		\begin{proof}
			Soient $x \in X_{\gamma^{-1}}$ et $x' \in X_{\gamma'}$.
			Par $\delta$-hyperbolicité, on a
			\[ (\xi | \xi') \geq \min\{ (\xi | x), (x | x'), (x' | \xi') \} - 2\delta. \]
			Or, par le lemme \ref{non_cont}, on a $(\xi | x) \geq \frac{1}{2} d(o, \gamma o) - 3 \delta > (\xi | \xi') + 2\delta$, et de même pour $(\xi' | x')$.
			On conclut donc que l'on a
			\[ (x | x') \leq (\xi | \xi') + 2\delta. \]
			Pour finir la preuve du lemme, il reste à montrer que les parties $X_{\gamma^{-1}}$ et $X_{\gamma'}$ sont d'adhérences dans $X$ disjointes.
			Par l'absurde, supposons qu'il existe $y \in \Adh{X_{\gamma^{-1}}} \cap \Adh{X_{\gamma'}} = X_{\gamma^{-1}} \cap X_{\gamma'}$.
			On a alors d'une part $(y | y) \geq \frac{1}{2} d(o, \gamma o) - \delta > (\xi | \xi') + 4\delta$ par le lemme \ref{lXg}, et d'autre part $(y | y) \leq (\xi | \xi') + 2\delta$ par l'inégalité ci-dessus.
			Contradiction.
		\end{proof}
		
		Montrons alors que l'isométrie $\gamma' \gamma$ est contractante.
		
%

		On a
		\[ \gamma' \gamma o \in \gamma'(X_{\gamma}) \subseteq \gamma'(X \backslash X_{{\gamma'}^{-1}}) \subseteq X_{\gamma'}, \]
		et de même ${\gamma}^{-1} {\gamma'}^{-1} o \in X_{\gamma^{-1}}$, donc par le lemme ci-dessus
		\[ (\gamma' \gamma o | (\gamma' \gamma)^{-1} o) \leq (\xi | \xi') + 2\delta < \frac{1}{2} d(o, \gamma o) - 3\delta. \]
		
		Or, on a $d(o, \gamma o) \leq d(o, \gamma' \gamma o)$. En effet, on a
		\[ d(o, \gamma o ) = d(o, \gamma' \gamma o) - d(o, \gamma' o) + 2({\gamma'}^{-1} o | \gamma o). \] 
		Or, en utilisant le lemme ci-dessus avec $\gamma$ et $\gamma'$ permutés, on obtient $({\gamma'}^{-1} o | \gamma o) \leq (\xi | \xi') + 2\delta$,
		et d'autre part, on a $d(o, \gamma' o) \geq 2(\xi | \xi') + 10\delta$.
		On obtient donc bien
		\[ d(o, \gamma o) \leq d(o, \gamma' \gamma o) - 2(\xi | \xi') - 10\delta + 2(\xi | \xi') + 4\delta \leq d(o, \gamma' \gamma o), \]
		et donc
		\[ (\gamma' \gamma o | (\gamma' \gamma)^{-1} o) < \frac{1}{2} d(o, \gamma' \gamma o) - 3\delta. \]
		Donc l'élément $\gamma \gamma'$ est contractant par le critère de contraction (lemme \ref{seph}).
		Pour finir, on a $\eta \not\in X_{\gamma'} \supseteq X_{\gamma' \gamma}$ d'après ce que l'on a fait ci-avant.
		
		
		Supposons maintenant que l'ensemble limite $\Lambda_\Gamma$ soit réduit à deux points :
		\[
			\Lambda_\Gamma = \{ \xi, \eta \}.
		\]
		Par $\Gamma$-invariance de l'ensemble limite, les isométries de $\Gamma$ fixent le doublet $\{\xi, \eta\}$.
		Distinguons alors deux cas :
		\begin{enumerate}
			\item	Si toutes les isométries de $\Gamma$ fixent chacun des points $\xi$ et $\eta$, considérons une isométrie $\gamma$ telle que
				\[
					(\gamma o | \xi) > 2(\xi | \eta) + 6\delta.
				\]
				Celle-ci existe bien puisque l'on a $\xi \in \Lambda_\Gamma$.
				On a alors le lemme :
				
				\begin{lemme}
					On a $\eta \not \in X_\gamma$ et $\eta \in X_{\gamma^{-1}}$.
				\end{lemme}
				
				\begin{proof}	
					Commençons par montrer que l'on a $\eta \not\in X_{\gamma}$.
					Par $\delta$-hyperbolicité, on a
					\[
						(\xi | \eta) \geq \min\{ (\xi | \gamma o), (\gamma o | \eta) \} - \delta,
					\]
					et donc $(\gamma o | \eta) \leq (\xi | \eta) + \delta < \frac{1}{2} d(o, \gamma o)$, ce qui prouve la première partie du lemme.
					
					On a ensuite $\eta = \gamma^{-1} \eta \in X_{\gamma^{-1}}$, puisque l'isométrie $\gamma$ fixe le point $\eta$.
				\end{proof}
				
				Pour terminer la preuve de la proposition dans le cas où toutes les isométries fixent chacun des points $\xi$ et $\eta$, il ne reste plus qu'à démontrer le lemme :
				
				\begin{lemme}
					L'isométrie $\gamma$ est contractante.
				\end{lemme}
				
				\begin{proof}
%
					Le point fixe $\xi$ est dans l'union $\overline{X_\gamma} \cup \overline{X_{\gamma^{-1}}}$,
					puisque si l'on avait $\xi \not \in \overline{X_\gamma} \cup \overline{X_{\gamma^{-1}}}$, alors on aurait $\xi = \gamma \xi \in \overline{X_\gamma}$, ce qui est absurde.
					
					Le fait que l'isométrie $\gamma$ soit contractante découle alors du critère de contraction \ref{seph}, puisque l'on a
					\[
						(\xi | \eta) < \frac{1}{2}d(o, \gamma o) - 3 \delta.
					\]
					
				\end{proof}
				
			\item	S'il existe une isométrie $\gamma_0 \in \Gamma$ qui échange $\xi$ et $\eta$, c'est-à-dire $\gamma_0 \xi = \eta$ et $\gamma_0 \eta = \xi$.
			
				La première partie de la preuve permet d'obtenir une isométrie contractante $h \in \Gamma$, mais avec $\eta \in X_h$.
				En outre, on peut choisir l'isométrie $h$ aussi grande que l'on veut, et donc demander à avoir l'inégalité
				\[
					(h o | \eta) > 2 (\xi | \eta) + 4d(o, \gamma_0 o) + 6 \delta.
				\]
				
				Montrons qu'alors l'isométrie $\gamma := \gamma_0 h \gamma_0$ convient.
				Pour cela, nous allons utiliser le critère de contraction \ref{seph} avec les points $\xi$ et $\eta$.
				On a le lemme suivant.
				
				\begin{lemme}
					On a $\eta \not \in X_{\gamma}$ et $\xi \not \in X_{\gamma^{-1}}$.
				\end{lemme}
				
				\begin{proof}
					Montrons que $\eta \not \in X_{\gamma}$.
					On a
					\begin{align*}
						(\eta | \gamma o)	&= (\gamma_0 \xi | \gamma_0 h \gamma_0 o) \\
										&= (\xi | h \gamma_0 o) - (\xi | \gamma_0^{-1} o) - (h \gamma_0 o | \gamma_0^{-1} o) + d(o, \gamma_0 o) \\
										&\leq (\xi | h \gamma_0 o) + d(o, \gamma_0 o).
					\end{align*}
					Et par $\delta$-hyperbolicité, on a
					\[
						(\xi | \eta) \geq \min\{ (\xi | h \gamma_0 o), (h \gamma_0 o | \eta) \} - \delta.
					\]
					Or, le point $\gamma_0 o$ n'est pas dans l'ensemble $X_{h^{-1}}$ d'après le lemme~\ref{lXg}, puisque l'on a
					$(\gamma_0 o | \gamma_0 o) = d(o, \gamma_0 o) < \frac{1}{2} d(o, ho) - \delta$.
					On a donc $h \gamma_0 o \in X_h$.
					Comme on a aussi $\eta \in X_h$, le lemme~\ref{lXg} donne
					\[
						(\eta | h \gamma_0 o) \geq \frac{1}{2} d(o, ho) - \delta > (\xi | \eta) + \delta.
					\]
					Ainsi, on obtient
					\begin{align*}
						(\eta | \gamma o)	&\leq (\xi | h \gamma_0 o) + d(o, \gamma_0 o) \\
										&\leq (\xi | \eta) + d(o, \gamma_0 o) + \delta \\
										&< \frac{1}{2} d(o, ho) - d(o, \gamma_0 o) \\
										&\leq \frac{1}{2} d(o, \gamma_0 h \gamma_0 o),
					\end{align*}
					ce qui prouve que $\eta$ n'est pas dans $X_\gamma$.
					
					On montre de manière semblable que $\xi$ n'est pas dans $X_{\gamma^{-1}}$.
				\end{proof}
				
				Ainsi, on a $\xi = \gamma_0 \eta = \gamma_0 h \eta = \gamma \xi \in X_\gamma$ et de même $\eta \in X_{\gamma^{-1}}$.
				Le critère de contraction \ref{seph} s'applique donc, puisque l'on a bien l'inégalité
				\[
					(\xi | \eta) < \frac{1}{2} d(o, h o) - d(o, \gamma_0 o) - 3\delta \leq \frac{1}{2}d(o, \gamma o) - 3 \delta.
				\]
				
		\end{enumerate}
		
	\end{proof}
	
	\subsection{Support d'un semi-groupe} \label{supp}
	
	Définissons le support d'un ensemble d'isométries, qui correspond aux couples de points du bord pour lesquels il y a beaucoup d'isométries qui contractent dans un voisinage du premier point du couple et dilatent depuis un voisinage du deuxième point du couple.
	Le support d'un semi-groupe est à relier au support de la mesure de Patterson-Sullivan, mais il est plus simple à définir et à manipuler.
	
	\begin{define}
		Soit $X$ un espace Gromov-hyperbolique de point base $o$.
		On appelle \defi{support} d'une partie $A \subseteq \Isom(X)$ l'ensemble
		\[ \supp(A) := \{ (\xi, \mu) \in \bX \times \bX | \text{ pour tout } \epsilon > 0,\ h_{A^{\beta(\xi, \epsilon) \times \beta(\mu, \epsilon)}} = h_A \}, \]
		où l'on a posé $\beta(\xi, \epsilon) := \{ x \in \overline{X} |\ (\xi | x) > - \log(\epsilon) \}$ et
		\[ A^{\beta(\xi, \epsilon) \times \beta(\mu, \epsilon)} := \{ \gamma \in A |\ \gamma o \in \beta(\xi, \epsilon) \text{ et } \gamma^{-1} o \in \beta(\mu, \epsilon) \}. \]
	\end{define}
	
	On peut montrer que le support ne dépend pas du point base $o$ choisi.
	
	\begin{prop} \label{nonvide}
		Soit $X$ un espace Gromov-hyperbolique propre à bord compact. 
		Alors le support de toute partie $\Gamma \subseteq \Isom(X)$ est non vide.
	\end{prop}
	
	\begin{proof}
		Soit $\delta > 0$ tel que l'espace $X$ soit $\delta$-hyperbolique.
		Construisons par récurrence une suite $(\xi_n, \eta_n)_{n \in \N}$ de couples de points du bord telle que l'on ait l'égalité $h_{\Gamma^{\beta(\xi_n, e^{-2\delta n}) \times \beta(\eta_n, e^{-2\delta n})} } = h_\Gamma$, et que l'on ait
		\[ \beta(\xi_n, e^{-2\delta n}) \times \beta(\eta_n, e^{-2\delta n}) \cap \beta(\xi_{n+1}, e^{-2\delta (n+1)}) \times \beta(\eta_{n+1}, e^{-2\delta (n+1)}) \neq \emptyset. \]
		
		Pour $n=0$, tous les points $\xi_0$ et $\eta_0 \in \bX$ conviennent puisque l'on a $\beta(\xi_0, 1) = \overline{X} = \beta(\eta_0, 1)$.
		
		Supposons $\xi_n$ et $\eta_n$ construits tels que $h_{\Gamma^{\beta(\xi_n, e^{-2\delta n}) \times \beta(\eta_n, e^{-2\delta n})} } = h_\Gamma$.
		Le carré $\bX \times \bX$ du bord étant compact, il en va de même du pavé fermé $(\overline{ \beta(\xi_n, e^{-2\delta n})} \cap \bX) \times (\overline{ \beta(\eta_n, e^{-2\delta n})} \cap \bX)$, et l'on peut extraire un recouvrement fini du recouvrement par les pavés ouverts $\beta(\xi, e^{-2\delta (n+1)}) \times \beta(\eta, e^{-2\delta (n+1)})$ quand $(\xi, \eta)$ décrit $\beta(\xi_n, e^{-2\delta n}) \times \beta(\eta_n, e^{-2\delta n})$.
		Le complémentaire dans $\beta(\xi_n, e^{-2\delta n}) \times \beta(\eta_n, e^{-2\delta n})$ de ce recouvrement est alors borné. Or, l'entropie d'une partie bornée de $X$ est nulle : on a 
		\[ h_{\Gamma^{B(o, R) \times B(o, R)}} = 0, \]
		pour tout $R > 0$, par propreté.
		Par le point \ref{hu} des propriétés \ref{ppteg} de l'entropie, il existe alors un couple $(\xi_{n+1}, \eta_{n+1})$ de $\beta(\xi_n, e^{-2\delta n}) \times \beta(\eta_n, e^{-2\delta n})$ vérifiant
		\[ h_{\Gamma^{\beta(\xi_{n+1}, e^{-2\delta (n+1)}) \times \beta(\eta_{n+1}, e^{-2\delta (n+1)})} } \geq h_{\Gamma^{\beta(\xi_n, e^{-2\delta n}) \times \beta(\eta_n, e^{-2\delta n})} } = h_\Gamma. \]		
		
		La suite $\left( \beta(\xi_{n}, e^{-2\delta n+2\delta}) \times \beta(\eta_{n}, e^{-2\delta n + 2\delta}) \right)_{n \in \N}$ ainsi obtenue des pavés grossis de $e^{2 \delta}$ est alors décroissante.
		En effet, soit $x \in \beta(\xi_{n+1}, e^{-2\delta(n+1)+2\delta})$ et soit $y \in \beta(\xi_{n+1}, e^{-2\delta(n+1)}) \cap \beta(\xi_{n}, e^{-2\delta n})$. Par $\delta$-hyperbolicité on a alors
		\[ (x | \xi_n) \geq \min\{ (x | \xi_{n+1}), (\xi_{n+1} | y), (y | \xi_n) \} - 2\delta > 2\delta n - 2\delta, \]
		ce qui donne bien l'inclusion $\beta(\xi_{n+1}, e^{-2\delta (n+1) + 2\delta}) \subseteq \beta(\xi_{n}, e^{-2\delta n + 2\delta})$.
		En faisant de même avec $\eta$, on obtient bien la décroissance souhaitée.
		
		La suite converge donc vers un point $(\xi, \eta)$. Ce point est bien dans le support $\supp(\Gamma)$, puisque pour tout $\epsilon > 0$, on peut trouver par Gromov-hyperbolicité un entier $n$ assez grand tel que l'on ait l'inclusion
		\[ \beta(\xi_{n}, e^{-2\delta n + 2\delta}) \times \beta(\eta_{n}, e^{-2\delta n + 2\delta}) \subseteq \beta(\xi, \epsilon) \times \beta(\eta, \epsilon), \]
		et donc tel que l'on ait l'inégalité
		$h_\Gamma = h_{\Gamma^{\beta(\xi_{n}, e^{-2\delta n + 2\delta}) \times \beta(\eta_{n}, e^{-2\delta n + 2\delta})}} \leq h_{\Gamma^{\beta(\xi, \epsilon) \times \beta(\eta, \epsilon)}} $.
		
	\end{proof}

	\begin{rem}
		C'est un des seuls endroits de la preuve où l'on utilise la compacité de l'adhérence de l'espace $X$ (l'autre endroit où l'on utilise cette compacité est pour définir la mesure de Patterson-Sullivan, voir~\ref{paulin} \ref{patt-sull}).
		Cette hypothèse n'est pas beaucoup plus forte que de demander seulement la propreté de l'espace $X$ (voir \ref{sbord}). 
	\end{rem}

	Voici une propriété de $\Gamma$-invariance du support :
	
	\begin{prop} \label{ginv}
		Soit $X$ un espace Gromov-hyperbolique et soit $\Gamma$ un semi-groupe d'isométries de $X$.
		Le support $\supp(\Gamma)$ est $\Gamma \times \Gamma^{-1}$-invariant pour l'action de $\Isom(X) \times \Isom(X)$ sur $\bX \times \bX$ donnée par 
		\[ (\gamma, {\gamma'}^{-1}) (\xi, \eta) := (\gamma \xi, {\gamma'}^{-1} \eta), \]
		pour toutes isométries $\gamma$ et $\gamma' \in \Isom(X)$ et $(\xi, \eta) \in \bX \times \bX$.
		
	C'est-à-dire que l'on a $(\gamma \xi, {\gamma'}^{-1} \eta) \in \supp(\Gamma)$ pour tout $(\xi, \eta) \in \supp(\Gamma)$ et tout $(\gamma, \gamma') \in \Gamma \times \Gamma$.
	\end{prop}
	
	
	\begin{proof}
		Soient $(\xi, \eta) \in \supp(\Gamma)$ et $(\gamma, \gamma') \in \Gamma \times \Gamma$.
		Montrons que l'on a $(\gamma \xi, {\gamma'}^{-1} \eta) \in \supp(\Gamma)$.
		Soit $\epsilon > 0$.
		Par le point \ref{hm} des propriétés \ref{ppteg}, on a
		\[ h_\Gamma = h_{\Gamma^{\beta(\xi, \epsilon) \times \beta(\eta, \epsilon)}} = h_{\gamma \Gamma^{\beta(\xi, \epsilon) \times \beta(\eta, \epsilon)} \gamma'}. \]
		Montrons que l'on a l'inclusion
		\[ \gamma \Gamma^{\beta(\xi, \epsilon) \times \beta(\eta, \epsilon)} \gamma'
			\subseteq \Gamma^{\beta(\gamma \xi, \epsilon e^{d(o, \gamma o) + d(o, \gamma' o)}) \times \beta({\gamma'}^{-1} \eta, \epsilon e^{ d(o, \gamma o) + d(o, \gamma' o)} )}. \]
		Soit $\gamma'' \in \gamma \Gamma^{\beta(\xi, \epsilon) \times \beta(\eta, \epsilon)} \gamma'$.
		On a alors d'une part
		\[ (\gamma^{-1} \gamma'' o | \xi) \geq (\gamma^{-1} \gamma'' {\gamma'}^{-1} o | \xi) - d(o, \gamma' o) > -\log(\epsilon) - d(o, \gamma' o). \]
		et d'autre part 
		\begin{align*}
			(\gamma'' o | \gamma \xi)	&= (\gamma^{-1} \gamma'' o | \xi) + (\gamma o | \gamma'' o) + (\gamma o | \gamma \xi) - d(o, \gamma o) \\
								&\geq (\gamma^{-1} \gamma'' o | \xi) - d(o, \gamma o) \\
								&\geq -\log(\epsilon) - d(o, \gamma o) - d(o, \gamma' o).
		\end{align*}
		De la même façon on obtient les inégalités
		\[ ({\gamma''}^{-1} o | {\gamma'}^{-1} \eta) \geq (\gamma' {\gamma''}^{-1} o | \eta) - d(o, \gamma' o) \geq (\gamma' {\gamma''}^{-1} \gamma o | \eta) - d(o, \gamma o) - d(o, \gamma' o), \]
		et donc $\gamma'' \in \Gamma^{\beta(\gamma \xi, \epsilon e^{d(o, \gamma o) + d(o, \gamma' o)}) \times \beta({\gamma'}^{-1} \eta, \epsilon e^{ d(o, \gamma o) + d(o, \gamma' o)})}$.
		
		On a donc $h_{\Gamma^{\beta(\gamma \xi, \epsilon e^{d(o, \gamma o) + d(o, \gamma' o)}) \times \beta({\gamma'}^{-1} \eta, \epsilon e^{ d(o, \gamma o) + d(o, \gamma' o) })}} \geq h_\Gamma$, et on a l'autre inégalité
		\[
			h_{\Gamma^{\beta(\gamma \xi, \epsilon e^{d(o, \gamma o) + d(o, \gamma' o)}) \times \beta({\gamma'}^{-1} \eta, \epsilon e^{ d(o, \gamma o) + d(o, \gamma' o) })}} \leq h_\Gamma
		\]
		par l'inclusion $\Gamma^{\beta(\gamma \xi, \epsilon e^{d(o, \gamma o) + d(o, \gamma' o)}) \times \beta({\gamma'}^{-1} \eta, \epsilon e^{ d(o, \gamma o) + d(o, \gamma' o) })} \subseteq \Gamma$.
		Ceci prouve bien que le point $(\gamma, {\gamma'}^{-1}) (\xi, \eta) = (\gamma \xi, {\gamma'}^{-1} \eta)$ est dans le support $\supp(\Gamma)$ puisque le réel $\epsilon e^{d(o, \gamma o) + d(o, \gamma' o)}$ parcours $\R_+^*$ quand $\epsilon$ parcours $\R_+^*$.
	\end{proof}
	
	\begin{quest}
		On a l'inclusion $\supp(\Gamma) \subseteq \Lambda_\Gamma \times \Lambda_{\Gamma^{-1}}$, mais a-t'on l'égalité ?
	\end{quest}

	\subsection{Le cas générique}
	
	On va maintenant montrer que si le support $\supp(\Gamma)$ n'est pas réduit à certains sous-espaces de $\bX \times \bX$, alors on a la conclusion du théorème \ref{thm1}, c'est-à-dire l'existence d'un sous-semi-groupe contractant de $\Gamma$ qui est d'entropie totale $h_\Gamma$.
	
	\begin{prop} \label{diag}
		Soit $X$ un espace Gromov-hyperbolique, et soit $\Gamma$ un semi-groupe d'isométries de $X$.
		Si le support de $\Gamma$ n'est pas inclus dans la diagonale de $\bX \times \bX$, alors il existe un sous-semi-groupe contractant $\Gamma'$ de $\Gamma$ tel que $h_{\Gamma'} = h_\Gamma$.
	\end{prop}
	
	\begin{proof}
		Soit $(\xi, \mu) \in \supp(\Gamma)$ avec $\xi \neq \mu$.
		On peut alors trouver un réel $\epsilon > 0$ assez petit pour que le produit $\beta(\xi, \epsilon) \times \beta(\eta, \epsilon)$ soit Gromov-disjoint de la diagonale.
		D'après le critère de contraction (proposition \ref{crit_cont}), l'ensemble $\Gamma^{\beta(\xi, \epsilon) \times \beta(\eta, \epsilon)}_{\geq n}$ est alors contractant pour $n$ assez grand, où l'on a posé
		\[ \Gamma^{\beta(\xi, \epsilon) \times \beta(\eta, \epsilon)}_{\geq n} := \{ \gamma \in \Gamma |\ \gamma o \in \beta(\xi, \epsilon), \gamma^{-1} o \in \beta(\eta, \epsilon) \text{ et } d(o, \gamma o) \geq n \}. \]
		Or, par définition du support et par le point \ref{hf} des propriétés \ref{ppteg} de l'entropie, on a
		\[ h_{\Gamma^{\beta(\xi, \epsilon) \times \beta(\eta, \epsilon)}_{\geq n}} = h_\Gamma. \]
		Le semi-groupe engendré par $\Gamma^{\beta(\xi, \epsilon) \times \beta(\eta, \epsilon)}_{\geq n}$ convient donc.

	\end{proof}
	
	\begin{prop} \label{hyp_cont}
		Soit $X$ un espace Gromov-hyperbolique propre, et $\Gamma$ un semi-groupe d'isométries de $X$.
		Si le semi-groupe $\Gamma$ contient une isométrie contractante $h$, et que le support $\supp(\Gamma)$ contient au moins trois points, alors il existe un sous-semi-groupe contractant $\Gamma'$ de $\Gamma$ tel que $h_{\Gamma'} = h_\Gamma$. 
	\end{prop}
	
	L'idée de la preuve est de montrer qu'il y a beaucoup d'éléments du semi-groupe avec lesquels l'élément $h$ joue un ping-pong.
	Cela permet alors de faire contracter ces éléments d'un endroit précis vers un endroit précis en les composant à gauche et à droite avec h.
	On obtient alors un semi-groupe contractant en ne gardant que les éléments assez grands.
	
	\begin{proof}
		Pour une isométrie contractante $h$ d'un espace $\delta$-hyperbolique $X$, définissons une partie $I_h$ de $X \cup \bX$ par
		\[ I_h := \{ \xi \in X \cup \bX | (\xi | h o) \geq \frac{1}{2} d(o, h o) - \delta \}. \]
		Montrons que si le support $\supp(\Gamma)$ n'est pas inclus dans $I_{h^{-1}} \times \bX \cup \bX \times I_h$, alors il existe un sous-semi-groupe contractant $\Gamma'$ de $\Gamma$, avec $h_{\Gamma'} = h_{\Gamma}$. 
		
		\begin{rem}
			L'ensemble $I_h$ contient $X_h$. Il est un peu plus gros que $X_h$ afin de garantir que les éléments $\gamma \in \Gamma$ assez grands et tels que l'on ait $\gamma^{-1} o \not \in I_h$ soient tels que les parties $X_{\gamma^{-1}}$ et $X_h$ sont disjointes (et idem dans l'autre sens).
			Ceci permettra alors de faire le ping-pong. 
		\end{rem}
		
		\begin{lemme} \label{ih}
			Soit $X$ un espace Gromov-hyperbolique, soit $h$ une isométrie contractante de $X$ et soient $I_h$ et $I_{h^{-1}} \subseteq X$ les parties définies ci-dessus. 
			Pour toute isométrie $\gamma \in \Isom(X)$ telle que
			$d(o, \gamma o) \geq d(o, h o)$ et
			$\gamma o \not \in I_{h^{-1}}$, on a
			\[ X_\gamma \cap X_{h^{-1}} = \emptyset, \]
			et pour toute isométrie $\gamma \in \Isom(X)$ telle que
			$d(o, \gamma o) \geq d(o, h o)$ et
			$\gamma^{-1} o \not \in I_{h}$, on a
			\[ X_{\gamma^{-1}} \cap X_{h} = \emptyset. \]
		\end{lemme}
		
		\begin{proof}
			Soit $\gamma \in \Isom(X)$ tel que $d(o, \gamma o) \geq d(o, h o)$ et $\gamma o \not \in I_{h^{-1}}$.
			Montrons l'inclusion $X_\gamma \subseteq X \backslash X_{h^{-1}}$.
			
			Soit $x \in X_\gamma$.
			Par $\delta$-hyperbolicité, on a
			\[ (\gamma o | h^{-1} o) \geq \min\{ (x | \gamma o), (x | h^{-1} o) \} - \delta. \]
			Or, par définition de $I_{h^{-1}}$ et de $X_\gamma$, on a les inégalités
			\[ \frac{1}{2} d(o, h o) - \delta > (\gamma o | h^{-1} o) \quad \text{ et } \quad (x | \gamma o) \geq \frac{1}{2} d(o, \gamma o) \geq \frac{1}{2} d(o, h o). \]
			D'où l'inégalité
			\[ \frac{1}{2} d(o, h o) - \delta > (x | h^{-1} o) - \delta \]
			qui prouve que le point $x$ n'est pas dans l'ensemble $X_{h^{-1}}$.
			
			La deuxième partie du lemme découle de la première en remplaçant $h$ par $h^{-1}$ et $\gamma$ par $\gamma^{-1}$.
		\end{proof}
		
		Posons $H$ l'ensemble des isométries satisfaisant les conditions du lemme ci-dessus :
		\[ H := \{ \gamma \in \Isom(X) | d(o, \gamma o) \geq d(o, h o), \gamma o \not \in I_{h^{-1}} \text{ et } \gamma^{-1} o \not \in I_h \} \] 

		Faisons alors un ping-pong entre les isométries de $H$ et l'isométrie $h$ pour obtenir une partie contractante.
		
		\begin{lemme}
			Soit $X$ un espace Gromov-hyperbolique, soit $h$ une isométrie contractante de $X$ et soit $H$ la partie de $X$ définie ci-dessus.
			Alors l'ensemble 
			\[ hHh := \{ h \gamma h | \gamma \in H \} \]
			est une partie contractante de $\Isom(X)$.
		\end{lemme}
		
		\begin{proof}
			Si $\gamma$ est un élément de $H$, alors par le lemme \ref{ih} on a
			\[ h \gamma h (X \backslash X_{h^{-1}}) \subseteq h \gamma (X_h) \subseteq h \gamma (X \backslash X_{\gamma^{-1}}) \subseteq h (X_\gamma) \subseteq h (X \backslash X_{h^{-1}}) \subseteq X_h, \]
			et l'ensemble $X \backslash (\overline{X_h} \cup \overline{X_{h^{-1}}} ) = X \backslash (X_h \cup X_{h^{-1}} )$ contient bien le point base $o$.
			
			La partie $hHh$ est donc contractante pour les domaines $X_- := X_{h^{-1}}$ et $X_+ := X_{h}$.
			
		\end{proof}
		
		\begin{figure}[H]
			\centering
			\caption{Ping-pong avec l'isométrie contractante $h$.}
			\includegraphics{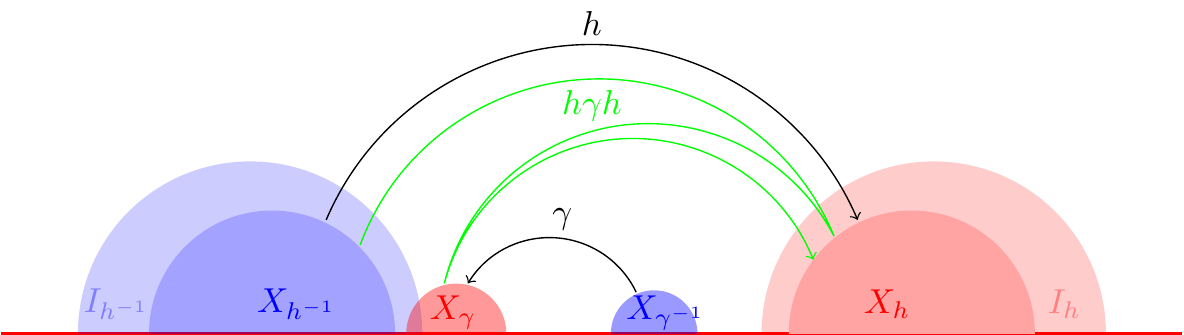}
%
%
%
%
%
%
%
%
%
%
		\end{figure}
		
		Continuons la preuve de la proposition \ref{hyp_cont}.
		Soit $h \in \Gamma$ un élément contractant.
		Le semi-groupe engendré par $h(H \cap \Gamma) h$ est alors un sous-semi-groupe contractant de $\Gamma$ d'après le lemme ci-dessus.
		
		Si le support $\supp(\Gamma)$ n'est pas inclus dans $I_{h^{-1}} \times \bX \cup \bX \times I_h$, alors l'entropie de $H \cap \Gamma$ est égale à $h_\Gamma$ par le point \ref{hf} des propriétés \ref{ppteg}.
		Or, par le point \ref{hm} des propriétés \ref{ppteg}, l'entropie de $h(H \cap \Gamma)h$ est égale à celle de $H \cap \Gamma$, et donc l'entropie du semi-groupe engendré est aussi égale à $h_\Gamma$.
		
		
		On est donc ramené à ce que le support $\supp(\Gamma)$ soit inclus dans $I_{h^{-1}} \times \bX \cup \bX \times I_h$. 
		Par les lemmes \ref{hn} et \ref{seph}, comme l'espace $X$ est propre, il existe un entier $n_0$ tel que pour tout $n \geq n_0$, l'isométrie $h^n$ soit aussi une isométrie contractante du semi-groupe $\Gamma$. On est alors même ramené à ce que le support $\supp(\Gamma)$ soit inclus dans $I_{h^{-n}} \times \bX \cup \bX \times I_{h^n}$, pour tout $n \geq n_0$. 
		
		Or, le diamètre des ensembles $I_{h^n}$ et $I_{h^{-n}}$ tend vers $0$ : on a
		\[ \lim_{n \to \infty} \inf_{x, x' \in I_{h^n}} (x | x') = \infty. \]
		On est donc ramené à ce que le support $\supp(\Gamma)$ soit inclus dans un ensemble 
		\[ \left( \{ \xi_- \} \times \bX \right) \cup \left(\bX \times \{ \xi_+ \} \right), \]
		où $\xi_-$ et $\xi_+$ sont des points du bord $\bX$.
		
		De plus, en utilisant la proposition \ref{diag}, on est ramené au cas où le support $\supp(\Gamma)$ est inclus dans la diagonale de $\bX \times \bX$, donc on est ramené à ce que le support soit inclus dans le doublet $\{ (\xi_-, \xi_-), (\xi_+, \xi_+) \}$.
		Ainsi, on a bien démontré l'existence d'un sous-semi-groupe contractant d'entropie $h_\Gamma$ dès que le support $\supp(\Gamma)$ contient au moins 3 points. 
		Ceci termine la preuve de la proposition \ref{hyp_cont}.
	\end{proof}
	
	\begin{figure}[H]
		\centering
		\caption{$\supp(\Gamma)$}
		\includegraphics{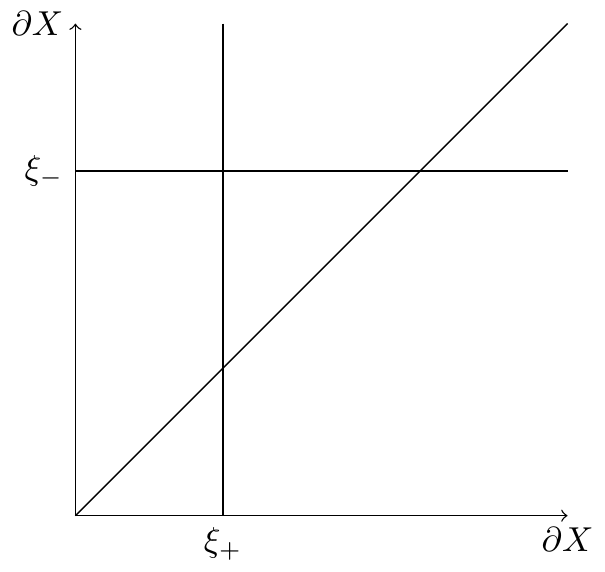}
	\end{figure}
	
	\subsection{Le cas où le support est un singleton}
	
	On a le lemme :
	
	\begin{lemme}
		Soit $X$ un espace Gromov-hyperbolique et soit $\Gamma$ un semi-groupe d'isométries de $X$.
		Si le support $\supp(\Gamma)$ est un singleton $\{(\xi, \xi)\}$, alors le semi-groupe $\Gamma$ fixe le point $\xi$.
	\end{lemme}
	
	\begin{proof}
		Cela découle de la propriété de $\Gamma \times \Gamma^{-1}$-invariance du support (proposition~\ref{ginv}).
	\end{proof}
	
	Ainsi, on est ramené à ce que le semi-groupe fixe un point au bord\\(voir sous-section~\ref{ptfixe}).
	
	\subsection{Le cas où le support est un doublet de points} 
	Supposons maintenant que le support $\supp(\Gamma)$ 
	soit un doublet de points du bord $\{ (\xi_+, \xi_+), (\xi_-, \xi_-) \}$. 
	On a alors le lemme suivant.
	
	\begin{lemme}
		Soit $X$ un espace Gromov-hyperbolique et soit $\Gamma$ un semi-groupe d'isométries de $X$.
		Si le support $\supp(\Gamma)$ est un doublet de points $\{ (\xi_+, \xi_+), (\xi_-, \xi_-) \}$, alors le semi-groupe $\Gamma$ fixe le doublet $\{ \xi_+, \xi_- \}$.
	\end{lemme}
	
	\begin{proof}
		Cela découle de la propriété de $\Gamma \times \Gamma^{-1}$-invariance du support (proposition~\ref{ginv}).
	\end{proof}
	 
	Montrons maintenant que quitte à multiplier par un élément qui échange les deux points du doublet, si l'on se restreint aux éléments de norme assez grande et qui contractent dans un des deux sens (d'un des points du doublet vers l'autre),
	alors on obtient un ensemble contractant. Et l'un de ces deux sous-semi-groupes engendrés sera forcément d'entropie totale (c'est-à-dire d'entropie $h_\Gamma$). \\

	Le lemme qui suit dit qu'un point fixe d'une isométrie est toujours dans son domaine de contraction ou dans son domaine de dilatation.
	
	\begin{lemme} \label{fixe}
		Soit $X$ un espace Gromov-hyperbolique et $\gamma$ une isométrie de $X$.
		Si $\xi \in \overline{X}$ est un point fixe pour l'isométrie $\gamma$ (i.e. $\gamma \xi = \xi$), alors on a
		\[ \xi \in X_\gamma \cup X_{\gamma^{-1}}. \]
	\end{lemme}
	
	\begin{proof}
		Supposons que l'on ait $\xi \not\in X_\gamma$. Alors on a $\xi = \gamma^{-1} \xi \in X_{\gamma^{-1}}$.
		En faisant de même avec l'inverse $\gamma^{-1}$, on conclut.
	\end{proof}
	
	Le lemme suivant dit qu'une isométrie assez grande qui fixe deux points contracte de l'un des points vers l'autre ou bien échange les deux points.
	
	\begin{lemme} \label{disj}
		Si une isométrie $\gamma$ d'un espace $\delta$-hyperbolique $X$ de point base $o$ fixe un doublet de points du bord $\{ \xi_+, \xi_- \} \subseteq X$ sans les échanger (i.e. $\gamma(\xi_-) \neq \xi_+$) et vérifie l'inégalité $d(o, \gamma o) > 2(\xi_+ | \xi_-) + 2\delta$, alors on a 
		\[ \left( \xi_+ \in X_\gamma \text{ et } \xi_- \in X_{\gamma^{-1}} \right) \quad \text{ou} \quad \left( \xi_- \in X_\gamma \text{ et } \xi_+ \in X_{\gamma^{-1}} \right) . \]
	\end{lemme}
	
	\begin{proof}
		Comme l'élément $\gamma$ fixe le doublet $\{ \xi_-, \xi_+ \}$, et n'échange pas $\xi_-$ et $\xi_+$, les points $\xi_-$ et $\xi_+$ sont fixes par $\gamma$. Le lemme \ref{fixe} nous donne donc l'inclusion
		\[ \{ \xi_+, \xi_- \} \subseteq X_\gamma \cup X_{\gamma^{-1}}. \]
		Or, on ne peut pas avoir $\{ \xi_+, \xi_- \} \subseteq X_{\gamma^{-1}}$, puisque par le lemme \ref{lXg} on a
		\[ \inf_{x, x' \in X_{\gamma^{-1}}} (x | x') \geq \frac{1}{2} d(o, \gamma o) - \delta > (\xi_+ | \xi_-). \]
		En faisant de même avec $X_\gamma$, on obtient donc bien ce qui était annoncé.
	\end{proof}
	
	Démontrons donc le théorème \ref{thm1} dans le cas où le semi-groupe $\Gamma$ fixe un doublet $\{ \xi_+, \xi_-\} \subseteq \bX$, mais ne fixe pas de point au bord.
	Considérons les parties suivantes du semi-groupe $\Gamma$ :
	\[ \Gamma^+ := \{ \gamma \in \Gamma | \xi_+ \in X_\gamma \text{ et } \xi_- \in X_{\gamma^{-1}} \} \]
	\[ \Gamma^- := \{ \gamma \in \Gamma | \xi_- \in X_\gamma \text{ et } \xi_+ \in X_{\gamma^{-1}}\} \]
	
	On a alors le lemme suivant.
	
	\begin{lemme} \label{cont+-}
		Les ensembles $\Gamma^+_{> n}$ et $\Gamma^-_{> n}$ sont contractants, \\
		pour tout entier $n > 2(\xi_+ | \xi_-) + 10\delta$, où l'on a posé
		\[ A_{> n} := \{ \gamma \in A | d(o, \gamma o) > n \}, \]
		pour une partie $A \subseteq \Gamma$.
	\end{lemme}
	
	\begin{proof}
		Montrons que $\Gamma^+_{> n}$ est contractant.
		Soient $\gamma$ et $\gamma'$ deux isométries de  $\Gamma^+_{> n}$.
		Par $\delta$-hyperbolicité, on a
		\[ (\xi_+ | \xi_-) \geq \min \{ (\xi_+ | \gamma' o), (\gamma' o | \gamma^{-1} o), (\gamma^{-1} o | \xi_-) \} - 2\delta. \]
		Or, par définition de $X_{\gamma'}$, on a $(\xi_+ | \gamma' o) \geq \frac{1}{2} d(o, \gamma' o) > (\xi_+| \xi_-) + 2\delta$
		et on a de même $(\xi_- | \gamma^{-1} o) > (\xi_+, \xi_-) + 2\delta$, puisque $\xi_+ \in X_{\gamma'}$ et $\xi_- \in X_{\gamma^{-1}}$.
		On en déduit l'inégalité
		\[ (\xi_+ | \xi_-) \geq (\gamma' o | \gamma^{-1} o) - 2\delta. \]
		Ainsi, on obtient
		\[ \sup_{\gamma, \gamma' \in \Gamma_+} (\gamma^{-1} o | \gamma' o) \leq (\xi_+ | \xi_-) + 2\delta < \infty, \]
		et on a bien pour tout $\gamma \in \Gamma_+$,
		\[ \frac{1}{2} d(o, \gamma o)  - 3 \delta > \frac{1}{2} n - 3 \delta > (\xi_+ | \xi_-) + 2\delta, \]
		donc par le critère \ref{crit_cont}, l'ensemble $\Gamma^+_{> n}$ est contractant. De la même façon, l'ensemble $\Gamma^-_{> n}$ est contractant.
	\end{proof}
	
	Pour terminer la preuve du théorème \ref{thm1} dans ce cas, il ne reste donc plus qu'à démontrer que l'entropie d'une de ces deux parties contractantes est $h_\Gamma$ :
	
	\begin{lemme} \label{h+-}
		On a $\max(h_{\Gamma^+_{> n}}, h_{\Gamma^-_{> n}}) = h_\Gamma$, pour $n \geq 2(\xi_+ | \xi_-) + 2\delta$.
	\end{lemme}
	
	\begin{proof}
		Il y a deux cas :
		\begin{enumerate}
			\item Il n'existe pas d'élément qui échange $\xi_-$ et $\xi+$.
				Dans ce cas, par le lemme \ref{disj}, on a
				\[ \Gamma_{> n} = {(\Gamma_+)}_{>n} \cup {(\Gamma_-)}_{>n}. \]
			\item Il existe un élément $\gamma_0 \in \Gamma$ qui échange $\xi_-$ et $\xi+$.
				Dans ce cas, on peut écrire
				\[ \left( \gamma_0 \left( \Gamma_{> n} \backslash \left( \Gamma_+ \cup \Gamma_- \right) \right) \right)_{> n}
					\subseteq (\Gamma_+)_{> n} \cup (\Gamma_-)_{> n}. \]
			
		\end{enumerate}
		Le résultat découle alors des point \ref{hu}, \ref{hm} et \ref{hf} des propriétés \ref{ppteg} de l'entropie.
	\end{proof}
	
	Les lemmes \ref{cont+-} et \ref{h+-} donnent un semi-groupe contractant d'entropie $h_\Gamma$ parmi l'un des deux semi-groupes suivants : l'un engendré par $\Gamma^+_{> n}$ et l'autre engendré par $\Gamma^-_{> n}$, pour $n$ assez grand.
		
	Ceci termine la preuve du théorème \ref{thm1} dans le cas où le semi-groupe $\Gamma$ ne fixe pas de point au bord.

	\subsection{Le cas où le semi-groupe $\Gamma$ fixe un point au bord} \label{ptfixe}
	
	Supposons que le semi-groupe $\Gamma$ fixe un point $\xi \in \bX$ mais ait un ensemble limite $\Lambda_\Gamma$ contenant au moins deux points.
	
	Par la proposition~\ref{econt}, il existe alors un élément contractant $h$ tel que le point fixe $\xi$ ne soit pas dans $\overline{X_h}$.
	
	
	On a alors la proposition suivante.
	
	\begin{prop} \label{ptf}
		Soit $X$ un espace métrique propre, soit $\Gamma$ un semi-groupe d'isométries de $X$, et soit $h$ une isométrie contractante.
		Si l'on pose
		\[ \Gamma' := \{ \gamma \in \Gamma | \gamma o \in X_h \}, \]
		alors on a l'égalité $h_{\Gamma'} = h_\Gamma$.
	\end{prop}
	
	Ceci permettra de conclure grâce au lemme suivant.
	
	\begin{lemme} \label{rcontf}
		Soit $X$ un espace Gromov-hyperbolique, soit $\Gamma$ un semi-groupe d'isométries fixant un point $\xi \in \bX$, et soit $h$ une isométrie contractante telle que $\xi \not \in \overline{X_h}$.
		Alors il existe un réel $r$ tel que l'ensemble
		\[ \Gamma'_{> r} := \{ \gamma \in \Gamma | \gamma o \in X_h \text{ et } d(o, \gamma o) > r \} \]
		soit contractant.
	\end{lemme}
	
	Le semi-groupe engendré sera alors encore contractant et aura encore pour entropie $h_\Gamma$ par les points \ref{hc} et \ref{hf} des propriétés \ref{ppteg} de l'entropie, donc on aura bien obtenu la conclusion du théorème \ref{thm1}. 
	
	\begin{proof}[Preuve du lemme \ref{rcontf}]
		Soit
		\[ C := \sup_{x \in X_h} (x | \xi) < \infty. \]
		Montrons que les réels $r > 2C + 8 \delta $ conviennent.
		
		Pour toute isométrie $\gamma \in \Gamma'_{> r}$, le point $\xi$ est dans $X_{\gamma^{-1}}$.
		En effet, on a l'inégalité $(\xi | \gamma o) \leq C < \frac{1}{2} d(o, \gamma o)$ qui donne que $\xi$ n'est pas dans $X_{\gamma}$, et donc on a
		\[ \xi = \gamma^{-1} \xi \in X_{\gamma^{-1}}. \]
		
		Par $\delta$-hyperbolicité, pour toutes isométries $\gamma$ et $\gamma' \in \Gamma'_{> r}$ on a l'inégalité
		\[ C \geq (\xi | \gamma' o) \geq \min \{ (\xi | \gamma^{-1} o), (\gamma^{-1} o | \gamma' o) \} - \delta. \]
		Or, on a $(\xi | \gamma^{-1} o) \geq \frac{1}{2} d(o, \gamma o) > \frac{1}{2} r > C + \delta$, donc on obtient l'inégalité
		\[ (\gamma^{-1} o | \gamma' o) \leq C + \delta < \frac{1}{2} r - 3 \delta. \]
		Par le critère de contraction \ref{crit_cont}, on obtient donc que la partie $\Gamma'_{> r}$ est contractante.
	\end{proof}
	
	Pour démontrer que l'ensemble $\Gamma'$ est d'entropie $h_\Gamma$ (i.e. la proposition \ref{ptf}), nous allons découper les boules $B(o, R)$ en morceaux, selon les copies d'un domaine fondamental pour l'élément $h$, et montrer que quitte à appliquer à chaque morceau une puissance de l'élément $h$, on peut ramener chaque morceau dans une partie proche de $X_h$, tout en restant dans la boule.
	Le lemme suivant permettra de ramener chaque morceaux.
	
	\begin{lemme}
		Soit $X$ un espace métrique, et soit $h$ une isométrie contractante de $X$.
		Alors il existe une constante $C < \infty$ telle que pour tout entier $n \in \N$ et pour tout réel $R > 0$ on ait l'inclusion
		\[ h^n \left( h^{-n} X_{h^{-1}} \cap B(o, R) \right) \subseteq B(o, R+C). \]
	\end{lemme}
	
	\begin{proof}
		Comme l'élément $h$ est contractant, il existe une constante $C$ telle que l'on ait
		\[ \sup_{(x,x') \in X _{h^{-1}} \times X_h}(x | x') \leq \frac{1}{2} C < \infty. \]
		En particulier, pour tout entier $n \geq 1$ et tout point $x \in h^{-n} X_{h^{-1}}$, on a 
		\[ \frac{1}{2} C \geq (h^n x | h^n o) = \frac{1}{2} \left( d(o, h^n x) + d(o, h^n o) - d(o, x) \right). \] 
		On a alors,
		\[ d(o, h^n x) \leq d(o, x) + C, \]
		d'où l'inclusion souhaitée.
	\end{proof}
	
	\begin{lemme} \label{minXh}
		Soit $X$ un espace métrique propre, 
		soit $\Gamma$ un semi-groupe d'isométries de $X$, soit $h$ une isométrie contractante de $\Gamma$, 
		et soit $S$ est une partie $\epsilon$-séparée et couvrante de $\Gamma$.
		Alors il existe une constante $C > 0$, telle que pour tout rayon $R > 0$ il existe une partie $\epsilon$-séparée $S_R \subseteq \Gamma$ telle que l'on ait l'inégalité
		\[
			\#S_R o \cap B(o, R + C) \cap X_h \geq \frac{1}{CR} \# S o \cap B(o, R).
		\]
	\end{lemme}
	
	Ce lemme dit que l'on a une proportion non négligeable des éléments de la boule $B(o, R+C)$ qui sont dans le domaine $X_h$. 
	Pour le démontrer, nous allons utiliser le lemme des tiroirs pour trouver un morceau de la boule qui contient beaucoup d'éléments, et ramener ce morceau par le lemme précédent dans le domaine $X_h$.
	
	\begin{proof}
		Posons
		\[ D_h := h^{-2} \left( X_h \backslash h X_h \right) . \]
		Alors $D_h$ est un domaine fondamental pour l'élément $h$ qui est inclus dans $X_{h^{-1}}$. 
		
		Majorons le nombre de morceaux du découpage de la boule $B(o,R)$ par ce domaine fondamental.
	
		\begin{souslemme}
			Il existe une constante $C_0 > 0$ telle que pour tout $R > 1$ et tout $n \geq C_0 R$, on ait
			\[ B(o, R) \cap h^{-n} D_h = \emptyset. \]
		\end{souslemme}
		\begin{proof}
			Si $x \in h^{-n} D_h \subseteq h^{-n} X_{h^{-1}}$, par le lemme \ref{cont1} (pour $X_+ = X_{h^{-1}}$ et $\gamma = h^{-n}$), on a
			\[ d(o, x) = (x | x) \geq d(o, h^n o) - 2\sup_{y \in X_{h^{-1}}}(h o | y). \]
			Or, par les lemmes \ref{hn} et \ref{cet}, il existe des constantes $C_1 > 0$ et $C_2 > 0$ telles que
			\[ d(o, h^n o) \geq C_1 n - C_2. \]
			Il existe donc une constante $C_0 > 0$ telle que pour tout $R \geq 1$ et $n \geq RC_0$ on ait
			\[ d(o, x) \geq C_1 n - C_2 - 2\sup_{y \in X_{h^{-1}}}(h o | y) > n/C_0 \geq R. \]
			D'où $x \not\in B(o, R)$ pour $x \in h^{-n} D_h$ avec $n \geq RC_0$.
		\end{proof}

	\begin{figure}[H]
		\centering
		\caption{Partition de la boule $B(o, R)$ à l'aide d'un domaine fondamental pour une isométrie contractante $h$.} 
		\includegraphics{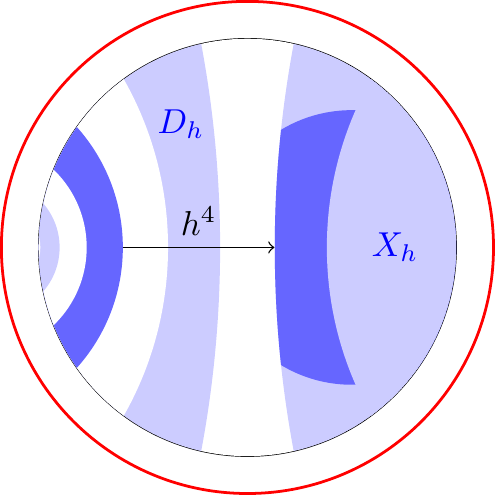}
%
%
%
%
%
%
%
%
%
%
%
	\end{figure}
	
		D'après ce sous-lemme, on peut donc partitionner la boule $B(o,R)$ en $n = \ceil{C_0R}+2$ morceaux :
		\[ B(o,R) = \left( B(o, R) \cap X_h \right) \sqcup \bigsqcup_{k = -1}^{n-3} \left( B(o, R) \cap h^{-k} D_h \right) . \]
		Le lemme des tiroirs nous donne alors que l'un des morceaux du découpage que l'on obtient pour $B(o, R) \cap S o$ est de cardinal au moins $\frac{1}{n} \# B(o,R) \cap S o$.
		Si le morceau en question est $B(o, R) \cap X_h$ ou $B(o, R) \cap h^{-(-1)} D_h$, alors le résultat est clair : $S_R := hS$ convient.
		Sinon, le morceau est $B(o, R) \cap h^{-k} D_h$ pour un entier $k \geq 0$. Par le lemme précédent, on a alors 
		\[ h^k \left( B(o,R) \cap S o \cap h^{-k} D_h \right) \subseteq B(o,R+C') \cap D_h \cap h^k S o, \]
		pour une constante $C'$ assez grande. Et par inégalité triangulaire, on a
		\[ h^2 \left( B(o,R+C') \cap D_h \cap h^k S o \right) \subseteq B(o,R+C' + d(o, h^2 o)) \cap h^2 D_h \cap h^{k+2} S o. \]
		Comme on a l'inclusion $h^2 D_h \subseteq X_h$, les inégalités précédentes donne l'inégalité annoncée
		\[ \# S_R o \cap B(o, R + C) \cap X_h
			\geq \frac{1}{C R} \# S o \cap B(o, R), \]
		avec $S_R = h^{k+2} S$, pour tout $R \geq 1$, pour une constante $C$ assez grande. 
	\end{proof}
	
	Voici maintenant un lemme d'inversion des quantificateurs.
	
	\begin{lemme} \label{ich}
		Soit $X$ un espace métrique propre, soit $\Gamma' \subseteq \Isom(X)$, et soient $h \geq 0$, $C > 0$ et $\epsilon > 0$ des réels. Si pour tout réel $R \geq 1$, il existe une partie $\epsilon$-séparée $S_R \subseteq \Gamma'$ telle que l'on ait l'inégalité
		\[
			\# S_R \cap B(o,R) \geq C e^{h R},
		\]
		alors on a $h_{\Gamma'} \geq h$.
	\end{lemme}
	
	\begin{proof}
		Soit $S \subseteq \Gamma'$ une partie séparée et $\frac{\epsilon}{2}$-couvrante. On a les inégalités
		\begin{align*}
			h_{\Gamma'}	&= \limsup_{R \to \infty} \frac{1}{R} \log( \# S o \cap B(o, R)) ) \\
						&\geq \limsup_{R \to \infty} \frac{1}{R} \log( \#S_R o \cap B(o, R))) ) \\
						& \geq h.
		\end{align*}
	\end{proof}
	
	\begin{proof}[Preuve de la proposition \ref{ptf}]
		Soit $S$ une partie $\epsilon$-séparée et couvrante du semi-groupe $\Gamma$.
		En utilisant le lemme \ref{minXh}, on obtient l'inégalité
		\[ \# S_R o \cap B(o, R + C) \geq \frac{1}{C R} \# S o \cap B(o, R), \]
		pour une partie $S_R$ $\epsilon$-séparée de $\Gamma'$, pour tout $R > 1$ et pour une constante $C$.
		Comme $S$ est une partie couvrante de $\Gamma$, son exposant critique est minoré par $h_{\Gamma}$, ce qui donne
		\[
			\frac{1}{C R} \# S o \cap B(o, R) \geq e^{(h_\Gamma-\epsilon_R) R},
		\]
		avec $\lim_{R \to \infty} \epsilon_R = 0$.
		Le lemme~\ref{ich} nous donne alors l'inégalité $h_{\Gamma'} \geq h_\Gamma - \epsilon_R$, puis on obtient l'inégalité $h_{\Gamma'} \geq h_\Gamma$ en faisant tendre $R$ vers l'infini.
		
		L'autre inégalité $h_{\Gamma'} \leq h_\Gamma$ se déduit de l'inclusion $\Gamma' \subseteq \Gamma$.
	\end{proof}

	Ceci termine la preuve du théorème \ref{thm1} : tous les cas ont été traités, puisque le support est non vide par la proposition \ref{nonvide}.
	
	\subsection{Contre-exemple quand l'ensemble limite du semi-groupe $\Gamma$ est réduit à un point}
	
	Sans l'hypothèse $\# \Lambda_\Gamma \geq 2$, les théorèmes \ref{thm1} et \ref{thm0} sont faux en général.
	
	\begin{ex} \label{contre-ex}
		Pour $X = \hyp$ muni de sa métrique usuelle, le sous-semi-groupe de $SL(2,\R)$ engendré par les matrices
		$\begin{pmatrix}
			1 & 1†\\
			0 & 1
		\end{pmatrix}$
		et
		$\begin{pmatrix}
			2 & 0†\\
			0 & \frac{1}{2}
		\end{pmatrix}$
		a pour entropie $\frac{1}{2}$, mais ne contient que des sous-semi-groupes contractants d'entropie nulle, donc en particulier ne contient que des sous-semi-groupes de Schottky d'exposant critique nul.
		
		Même chose avec le groupe parabolique engendré par la matrice
		$\begin{pmatrix}
			1 & 1†\\
			0 & 1
		\end{pmatrix}$.
	\end{ex}
	
	\begin{figure}[H]
		\centering
		\caption{Orbite d'un point sous l'action du semi-groupe de l'exemple \ref{contre-ex}.}
		\includegraphics{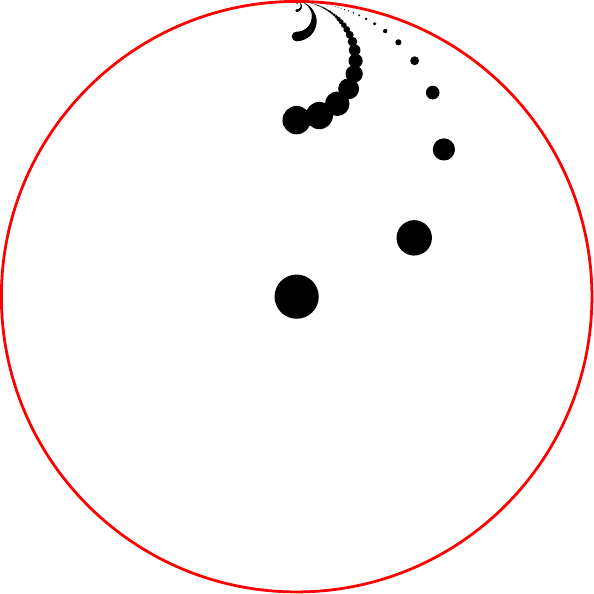}
%
	\end{figure}
		
	\section{Semi-groupes de Schottky} \label{schottky}
	
	Les semi-groupes de Schottky sont les semi-groupes ayant la dynamique la plus simple, puisque par définition leurs générateurs jouent au \og ping-pong \fg.
	En particuliers ils sont libres et séparés.
	Le théorème principal de cette section, que nous démontrons ici (théorème~\ref{thm2}), affirme que l'entropie d'un semi-groupe est approchée aussi près que l'on veut par celle de ses sous-semi-groupes de Schottky. 
	
	\begin{define}
		Soient $X$ un espace métrique, et $\Gamma$ un semi-groupe d'isométries de $X$.
		On dit que le semi-groupe $\Gamma$ est \defi{de Schottky} pour une partie $X_+ \subseteq X$, 
		s'il admet une partie génératrice finie
		$\{ g_1, g_2, ... , g_n \}$, telle que les parties $g_1 X_+$, $g_2 X_+$, ..., $g_n X_+$ et $X \backslash X_+$ soient deux à deux Gromov-disjointes, et que l'ensemble $X_+ \backslash (g_1 X_+ \cup ... \cup g_n X_+)$ soit d'intérieur non vide.
	\end{define}
	
	\begin{figure}[H]
		\centering
		\caption{Un semi-groupe de Schottky engendré par deux isométries $g_1$ et $g_2$}
		\includegraphics{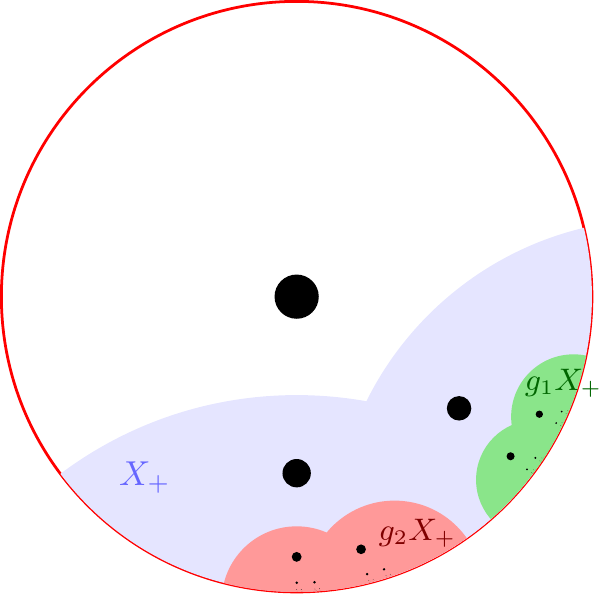}
%
%
%
%
%
%
	\end{figure}
	
	\begin{props}
		Soit $X$ un espace métrique et soit $\Gamma$ un semi-groupe de Schottky d'isométries de $X$. On a alors :
		\begin{enumerate}
			\item Le semi-groupe $\Gamma$ est libre.
			\item Le semi-groupe $\Gamma$ est séparé (donc en particulier d'orbite discrète).
			\item Le semi-groupe $\Gamma$ est contractant.
			\item Si l'espace $X$ est propre, on a $\delta_\Gamma < \infty$.
		\end{enumerate}
	\end{props}
	
	\begin{proof}
		Soit $B$ une boule incluse dans l'ensemble $X_+ \backslash (g_1 X_+ \cup ... \cup g_n X_+)$, alors ses images par les isométries de $\Gamma$ sont toutes deux à deux disjointes. Donc le semi-groupe est libre et séparé. Le semi-groupe $\Gamma$ est contractant pour les domaines $X \backslash X_+$ et $g_1 X_+ \cup ... \cup g_n X_+$, en choisissant $o \in X_+ \backslash (\overline{g_1 X_+ \cup ... \cup g_n X_+})$.
		Pour obtenir la finitude de l'exposant critique, il suffit d'utiliser le lemme \ref{cdv} pour obtenir un entier $k$ tel que pour tous générateurs $\gamma_1$, ..., $\gamma_k$, on ait
		\[ d(o, \gamma_1 ... \gamma_k o) \geq 2M + 1, \]
		où $M$ est la constante de Gromov-disjonction du semi-groupe contractant $\Gamma$.
		Puis on utilise le lemme \ref{cet} pour obtenir la minoration
		\[ d(o, \gamma_1 ... \gamma_n o) \geq \floor{\frac{n}{k}}, \]
		pour tout $n$, et pour des générateurs $\gamma_1$, ..., $\gamma_n$.
		Et on obtient alors la majoration $\delta_\Gamma \leq k \log(N)$, où $N$ est le nombre de générateurs du semi-groupe $\Gamma$.
	\end{proof}
	
	\begin{thm} \label{thm2}
		Soit $X$ un espace Gromov-hyperbolique propre, et soit $\Gamma$ un semi-groupe d'isométries de $X$.
		Si le semi-groupe $\Gamma$ est contractant, alors pour tout $\epsilon > 0$, il existe un sous-semi-groupe Schottky $\Gamma' \subseteq \Gamma$ tel que $\delta_{\Gamma'} \geq h_\Gamma - \epsilon$.
	\end{thm}
	
	Pour construire un \og gros \fg\ sous-semi-groupe de Schottky, nous 
	considèrerons une partie suffisamment séparée de ce semi-groupe contractant, et montrerons que les éléments de norme donnée assez grande, contractent à des endroits suffisamment écartés les uns des autres pour jouer au \og ping-pong \fg\  et donc engendrer un semi-groupe de Schottky. \\
	
	
	
	\begin{lemme} \label{sc}
		Soit $X$ un espace Gromov-hyperbolique, et soient $X_-$ et $X_+$ des parties de $X$. 
		Alors il existe un réel $r$ et un entier $n_0$ tels que pour tout semi-groupe contractant $\Gamma$ d'isométries de $X$ pour les parties $X_+$ et $X_-$, pour toute partie $S$ $r$-séparée de $\Gamma$, et pour tout entier $n \geq n_0$,
		l'ensemble $S \cap A_n$ engendre un semi-groupe de Schottky pour le domaine $X_+$,
		où l'on a posé
		\[ A_n := \{ \gamma \in \Isom(X) | d(o, \gamma o) \in [n, n+1[ \}. \]
	\end{lemme}
	
	\begin{proof}
		Supposons qu'il existe un semi-groupe contractant pour les parties $X_-$ et $X_+$ (sinon il n'y a rien à démontrer).
		Montrons que le réel $r := 4C + 4\delta +2$ convient, où
		\[ C := \sup_{(x, x') \in X_+ \times X_-} (x | x'), \]
		et $\delta$ est un réel tel que l'espace $X$ soit $\delta$-hyperbolique.
		
		Soit $\Gamma$ un semi-groupe contractant pour les parties $X_-$ et $X_+$, et soient $\gamma$ et $\gamma'$ deux isométries de $A_n$ qui vérifient l'inégalité
		\[ d(\gamma o, \gamma' o) \geq r. \]
		Montrons qu'alors les domaines $\gamma X_+$ et $\gamma' X_+$ sont Gromov-disjoints.
		
		Soient $x \in \gamma X_+$ et $x' \in \gamma' X_+$.
		Par Gromov-hyperbolicité, on a alors
		\[ (\gamma o | \gamma' o) \geq \min\{ (\gamma o, x), (x | x'), (x', \gamma' o) \} - 2\delta. \]
		Or, on a l'inégalité
		\begin{align*}
			(\gamma o | \gamma' o)	&= \frac{1}{2} \left( d(\gamma o, o) + d(\gamma' o, o) - d(\gamma o, \gamma' o) \right) \\
								&< (n + 1) - \frac{r}{2} = n - 2C - 2\delta,
		\end{align*}
		et d'après le lemme \ref{cont1}, on a les inégalités
		\[ (\gamma o | x) \geq d(o, \gamma o) - 2C \geq n - 2C \]
		et de même $(x' | \gamma' o) \geq n - 2C$.
		On obtient donc l'inégalité
		\[ (x | x') < n - 2C. \]
		
		Les ensembles $\gamma X_+$ et $\gamma' X_+$ sont alors disjoints, puisque si l'on avait $y \in \gamma X_+ \cap \gamma' X_+$, on aurait l'absurdité
		\[ n - 2C - 2\delta > (\gamma o |  \gamma' o) \geq \min\{ (\gamma o | y), (y, \gamma' o) \} - \delta \geq n - 2C - \delta. \]
		
		
		On a donc montré que les ensembles $\gamma X_+$ et $\gamma' X_+$ sont $(n-2C)$-disjoints.
		Pour finir, l'intersection $B(o, n-1) \cap X_+$ est d'interieur non vide si l'entier $n$ est assez grand, et elle est incluse dans $X_+ \backslash \left( \cup_{\gamma \in S \cap A_n} \gamma X_+ \right)$.
		
	\end{proof}
	
	Ainsi, pour $n$ assez grand et pour le réel $r$ donné par le lemme, si l'on considère une partie $S$ $r$-séparée et couvrante de $\Gamma$, alors
	l'ensemble $A_n \cap S$ engendre un semi-groupe de Schottky.
	Et par les points \ref{sep} et \ref{couv} des propriétés \ref{ppteg} et la remarque \ref{delta-ann} on a
	\[ h_\Gamma = \limsup_{n \to \infty} \frac{1}{n} \ln(\#(A_n \cap S)). \]
	Pour tout $\epsilon > 0$, on peut donc trouver un entier $n$ tel que
	\[ \ln( \# (A_n \cap S) ) \geq n(h_\Gamma - \epsilon). \]
	
	Montrons alors que le semi-groupe de Schottky $\Gamma'$ engendré par $A_n \cap S$ a un exposant critique supérieur ou égal à $\frac{n}{n+1} (h_\Gamma - \epsilon)$. Pour cela, on va utiliser le lemme suivant.
	
	\begin{lemme} \label{mnec}
		Soit $X$ un espace métrique. Si une partie $S$ du groupe d'isométries $\Isom(X)$ engendre un semi-groupe libre $\Gamma$, alors on a la minoration de l'exposant critique :
		\[ \delta_{\Gamma} \geq \frac{\log(\# S)}{r}, \]
		où $r = \sup_{\gamma \in S} d(o, \gamma o)$.
	\end{lemme}
	
	\begin{proof}
		Cela découle de l'inégalité triangulaire. Pour des générateurs $\gamma_1, ..., \gamma_n \in S$, on a
		\[ d(\gamma_1 ... \gamma_n o, o) \leq d(\gamma_1 o, o) + ... + d(\gamma_n o, o) \leq n r, \]
		donc, par liberté du semi-groupe $\Gamma$ pour la partie $S$, on obtient
		\[ \# \{ \gamma \in \Gamma | d(o, \gamma o) \leq nr \} \geq (\#S)^n. \]
		On a donc
		\begin{align*}
			\delta_{\Gamma}	&= \limsup_{n \to \infty} \frac{1}{n} \ln(\#\{ \gamma \in \Gamma | d(o, \gamma o) \leq n \}) \\
						&\geq \limsup_{n \to \infty} \frac{1}{nr} \ln((\#S)^n) \\
						&= \frac{\ln(\#S)}{r}.
		\end{align*}
	\end{proof}
	
	En appliquant le lemme \ref{mnec} à la partie $A_n \cap S$, qui engendre un semi-groupe qui est de Schottky et qui est donc libre, on obtient l'inégalité
	\[ \delta_{\Gamma'} \geq \frac{\log(\#(A_n \cap S))}{n+1} \geq \frac{n}{n+1} (h_\Gamma - \epsilon). \]	
	Or, l'entier $n$ pouvait être choisi arbitrairement grand, et le réel $\epsilon > 0$ est arbitraire.
	Ceci achève la preuve du théorème \ref{thm2}. \\
	
	On peut maintenant facilement démontrer le théorème \ref{thm0}.
	
	\begin{proof}[Preuve du théorème \ref{thm0}]
		On déduit aisément des théorèmes \ref{thm1} et \ref{thm2} l'inégalité
		\[ \sup_{\Gamma' < \Gamma \atop \Gamma' \text{ sous-semi-groupe de Schottky }} \delta_{\Gamma'} \geq h_\Gamma. \]
		L'autre inégalité s'obtient en remarquant qu'un semi-groupe de Schottky est séparé.
	\end{proof}
	
	\section{Dimension visuelle} \label{dim-vis} 
	
	Dans cette section, nous voyons une application du théorème \ref{thm0} à \og l'étude au bord \fg\ d'un semi-groupe.
	Nous obtenons le corollaire~\ref{cpaulin_} ci-après qui est une généralisation d'un résultat de F. Paulin (voir \cite{paulin}) qui généralise lui-même un résultat de Coornaert (voir \cite{coornaert}). \\
	
	Soit $X$ un espace Gromov-hyperbolique.
	Pour définir ce qu'est la dimension visuelle d'une partie $\Lambda$ du bord $\bX$, introduisons quelques notations.
	
	On définit la \defi{boule $\beta(\xi, r)$} de centre $\xi$ et de rayon $r$ sur le bord $\bX$ par
	\[ \beta(\xi, r) := \{ \eta \in \bX | (\xi | \eta) > -\log(r) \}. \]
	
	\begin{define}
		On appelle \defi{mesure visuelle} de dimension $s$ d'une partie $\Lambda \subseteq \bX$ du bord d'un espace $X$ Gromov-hyperbolique, le réel
		\[ H^s(\Lambda) := \lim_{\epsilon \to 0} H^s_\epsilon(\Lambda), \]
		où $H^s_\epsilon(\Lambda)$ est la borne inférieure des sommes
		\[ \sum_{i \in \N} r_i^{s} \]
		sur tous les recouvrements $(\beta( \xi_i, r_i ))_{i \in \N}$ de l'ensemble $\Lambda$ par des boules de rayons $r_i \leq \epsilon$.
		
		On appelle \defi{dimension visuelle} d'un ensemble $\Lambda \subseteq \bX$ le réel
		\[ \dv (\Lambda) := \inf\{ s \in \R_+ | H^s(\Lambda) = 0 \}. \]
	\end{define}
	
	\begin{rem}
		On a aussi
		\[ \dv (\Lambda) = \sup\{ s \in \R_+ | H^s(\Lambda) = \infty \}. \]
	\end{rem}
	
	\begin{rem}
		La mesure visuelle est une mesure.
	\end{rem}
	
	La notion de dimension visuelle généralise celle de dimension de Hausdorff.
	
	\subsection{Lien entre dimension visuelle et entropie} \label{paulin} 
	
	On a le résultat suivant.
	
	\begin{cor} \label{cpaulin_}
		Soit $X$ un espace Gromov-hyperbolique propre à bord compact et soit $\Gamma$ un semi-groupe d'isométries de $X$ dont l'ensemble limite contient au moins deux points. Alors on a l'égalité
		\[ \dv(\Lambda_{\Gamma}^c) = h_\Gamma. \]
	\end{cor}
	
	F. Paulin a énoncé ce résultat pour les groupes discret, et sa preuve semble s'adapter aux semi-groupes.
	Cependant, il fait des hypothèses supplémentaires par rapport à notre preuve, qui sont le fait que l'espace $X$ soit géodésique, qu'il soit quasi-géodésique, que le semi-groupe soit séparé, et qu'il ne fixe pas de point au bord (voir \cite{paulin}). \\
	
	\noindent Voici l'inégalité facile entre entropie et dimension visuelle de l'ensemble limite radial.
	
	\begin{prop} \label{if}
		Soit $X$ un espace Gromov-hyperbolique propre, et soit $\Gamma$ un semi-groupe d'isométries de $X$.
		On a l'inégalité
		\[ \dv(\Lambda_\Gamma^c) \leq h_\Gamma. \]
	\end{prop}
	
	\begin{proof}
		Soit $S$ une partie séparée et couvrante de $\Gamma$.
		Montrons que l'on a l'inégalité $\dim(\Lambda_S^c) \leq \delta_S$.
		Comme on a les égalités $\Lambda_S^c = \Lambda_\Gamma^c$ et $\delta_S = h_\Gamma$, ceci donnera bien l'inégalité souhaitée.
		
		Définissons l'\defi{ombre} d'une boule $B(x, r)$ par
		\[ O B(x, r) := \{ \xi \in \bX | (o | \xi)_{x} \leq r \}. \]
		On a alors l'inclusion
		\[ \Lambda_S^c \subseteq \bigcup_{r > 0} \bigcap_{n \geq 0} \bigcup_{\gamma \in S_{\geq n}} O B(\gamma o, r), \]
		où $S_{\geq n} := \{ \gamma \in S | d(o, \gamma o) \geq n \}$.
		En effet, si un élément $\xi$ est dans l'ensemble limite radial $\Lambda_S^c$, alors il existe un réel $r > 0$ et une partie $A$ de $S o$ qui est une $r$-sous-quasi-géodésique telle que $\xi \in \bord{ A }$. On a alors pour tout $x \in A$, $\xi \in O B(x, r)$, et pour tout $n \in \N$, $A_{\geq n} \neq \emptyset$. 
		
		Posons alors
		\[ \Lambda_r := \bigcap_{n \geq 0} \bigcup_{\gamma \in S_{\geq n}} O B(\gamma o, r), \]
		pour un réel $r>0$ et montrons que pour $s >\delta_S$, on a $H^s (\Lambda_r) < \infty$.
		
		On peut recouvrir chaque ombre par une boule de rayon $e^{-d(o,x) + r + \delta}$.
		En effet, soient $\xi$ et $\xi'$ deux points de l'ombre $O B(x, r)$. On a alors
		\[ (\xi | \xi') \geq \min\{ (\xi | x), (x | \xi')\} - \delta \geq d(o, x) - r - \delta, \]
		par $\delta$-hyperbolicité, et par l'inégalité
		\[ (\xi | x) = - (o | \xi)_x + d(o, x) \geq d(o, x) - r \]
		et de même avec $\xi'$.
		
		Ainsi, pour $\epsilon > 0$, en considérant un recouvrement de l'ensemble $\Lambda_r$ par des boules de rayon $\leq \epsilon$ qui recouvrent les ombres $O B(\gamma o, r)$ pour $\gamma \in S$ assez grand, on obtient
		\[ H^s(\Lambda_r) = \lim_{\epsilon \to 0} H^s_\epsilon(\Lambda_r) \leq \sum_{\gamma \in S} e^{-s(d(o, \gamma o) - r - \delta)} = e^{s (r + \delta)} P_s, \]
		où $P_s =  \sum_{\gamma \in S} e^{-s d(o, \gamma o)}$ est la série de Poincaré de $S$.
		On a $P_s < \infty$ dès que $s > \delta_S$, d'où $H^s(\Lambda_r) < \infty.$
		
		On a ensuite $H^s(\Lambda_r) = 0$ pour tout $s > \delta_S$, puis
		\[ H^s(\Lambda_S^c) = H^s(\bigcup_{r > 0} \Lambda_r) = 0. \]
		Ainsi, on a
		\[ \dv (\Lambda_S^c) \leq s \]
		pour tout $s > \delta_S$, d'où l'inégalité $\dv (\Lambda_S^c) \leq \delta_S$.
		
		
	\end{proof}
	
	Pour obtenir le corollaire \ref{cpaulin_} à partir du théorème \ref{thm0}, il suffit de démontrer le résultat dans le cas des semi-groupes de Schottky :
	
	\begin{prop} \label{pschottky}
		Soit $X$ un espace Gromov-hyperbolique propre à bord compact, 
		et soit $\Gamma$ un semi-groupe de Schottky d'isométries de $X$.
		Alors on a l'égalité
		\[ \dv(\Lambda_{\Gamma}) = \delta_\Gamma. \]
	\end{prop}
	
	\begin{proof}
		Par les propositions \ref{if} et \ref{pradial}, on a déjà l'inégalité
		\[ \dv (\Lambda_\Gamma) \leq h_\Gamma \leq \delta_\Gamma. \]
		
		Montrons l'inégalité $\dv (\Lambda_\Gamma) \geq \delta_\Gamma$.
		Pour cela, on va utiliser le lemme suivant, dû à Frostman, qui ramène le problème à construire une mesure convenable sur l'ensemble limite $\Lambda_\Gamma$ du semi-groupe $\Gamma$.
		
		\begin{lemme} \label{mnd}
			Soit $X$ un espace Gromov-hyperbolique et soit $\mu$ une probabilité portée par une partie $\Lambda$ du bord $\bX$.
			S'il existe un réel $s$ et une constante $C > 0$ tels que l'on ait
			\[ \mu(\beta(\xi, r)) \leq C r^s, \]
			pour toute boule $\beta(\xi, r)$ du bord $\bX$, alors on a l'inégalité
			\[ \dv \Lambda \geq s. \]
		\end{lemme}
		
		\begin{proof}
			Soit $\epsilon > 0$, et soit $R$ un recouvrement de l'ensemble $\Lambda$ par des boules de tailles inférieures à $\epsilon$.
			On a alors les inégalités
			\[ \sum_{\beta(\xi, r) \in R} r^s \geq \sum_{\beta(\xi, r) \in R} \frac{1}{C}\mu(\beta(\xi, r)) \geq \frac{1}{C} \mu(\bX) = \frac{1}{C}. \]
			
			On en déduit, en passant à la borne inférieure sur tous ces recouvrements que l'on a l'inégalité $H^s_\epsilon(\Lambda) \geq \frac{1}{C}$, puis en passant à la limite quand $\epsilon$ tend vers 0, que l'on a $H^s(\Lambda) \geq \frac{1}{C}$.
			On obtient donc bien l'inégalité souhaitée.
			
		\end{proof}
		
		La mesure à laquelle nous appliquerons ce lemme pour conclure est la mesure $\mu$ de Patterson-Sullivan, que nous allons définir maintenant.
		
		\subsubsection{Mesure de Patterson-Sullivan} \label{patt-sull}
		
		Soit $X$ un espace Gromov-hyperbolique propre à bord compact 
		et soit $\Gamma$ un semi-groupe discret d'isométries de $X$, avec $\delta_\Gamma < \infty$.
		Définissons des probabilités $\mu_s$ sur l'espace $X$, pour des réels $s > \delta_\Gamma$, par
			\[ \mu_s := \frac{1}{P_s} \sum_{\gamma \in \Gamma} e^{-s d(o, \gamma o)} D_{\gamma o}, \]
			où $ P_s := \sum_{\gamma \in \Gamma} e^{-s d(o, \gamma o)} $ est la série de Poincaré de $\Gamma$, et $D_x$ est le Dirac en $x$.
		
		\begin{rem}
			La série de Poincaré $P_s$ diverge pour $s < \delta_\Gamma$ et converge pour $s > \delta_\Gamma$.
		\end{rem}
		
		Pour définir la mesure $\mu$, nous aurons besoin que la série de Poincaré soit divergente en $\delta_\Gamma$ (i.e. $P_{\delta_\Gamma} = \infty$). On la rend divergente grâce au lemme suivant.
		
		\begin{lemme}[Astuce de Patterson] \label{lpat}
			Soit $s_0$ un réel et $(a_n)_{n \in \N}$ une suite de réels positifs.
			Si la série de Dirichlet $\sum_{n \in \N} a_n^{-s}$ est divergente pour $s < s_0$ et convergente pour $s > s_0$,
			alors il existe une fonction croissante $k : [0, \infty[ \to [0, \infty[$ telle que la série
			\[ \sum_{n \in \N}k(a_n) a_n^{-s} \]
			converge pour $s > s_0$ et diverge pour $s < s_0$ et pour $s = s_0$, et avec de plus la propriété :
			pour tout $\epsilon > 0$, il existe un réel $y_0$ tel que pour $y > y_0$ et $x > 1$, on ait
			\[ k(xy) \leq x^\epsilon k(y). \]
		\end{lemme}
		
		Voir \cite{pat} pour une preuve.
		
		Pour rendre la série de Poincaré divergente en $s = \delta_\Gamma$, il suffit de la remplacer par :
		\[ P_s := \sum_{\gamma \in \Gamma} k(e^{d(o, \gamma o)}) e^{-s d(o, \gamma o)}, \]
		où $k$ est la fonction fournie par ce lemme, et l'on fait de même pour la définition des mesures $\mu_s$.
		
		Les mesures $\mu_s$ sont des mesures de probabilités.
		Or, par hypothèse, l'adhérence $\overline{X}$ est compacte. Et l'ensemble de probabilités $\Part(\overline{X})$ sur le compact $\overline{X}$, muni de la convergence vague, est alors compact (voir par exemple \cite{rudin}).
		Il existe donc une suite de réels $(s_k)_{k \in \N}$, avec pour tout $k$, $s_k > \delta_\Gamma$,
		telle que la suite de mesures $(\mu_{s_k})_{k \in \N}$ converge vaguement vers une mesure de probabilité $\mu$ :
			\[ s_k \xrightarrow[k \to \infty]{} \delta_\Gamma \quad \text{et} \quad \mu_{s_k} \xrightharpoonup[k \to \infty]{} \mu. \]
		
		La mesure $\mu$ est alors portée par le bord $\bX$, puisque $\Gamma$ est une partie discrète, que l'espace $X$ propre et que l'on a $\lim_{s \to \delta_\Gamma} P_s = \infty$.
		
%
		
		Avant de majorer la mesure $\mu$ sur toutes les boules, majorons là sur les ensembles $\gamma X_+$.
		
		\begin{lemme} \label{mmpn}
			Soit $X$ un espace Gromov-hyperbolique propre à bord compact, 
			et soit $\Gamma$ un semi-groupe d'isométries de $X$,
			de Schottky pour un domaine $X_+$. Alors il existe un point $o$ tel que si $\mu$ est la mesure de Patterson-Sullivan définie ci-dessus (pour ce point $o$), alors il existe une constante $C > 0$ telle que pour toute isométrie $\gamma \in \Gamma$, on ait l'inégalité
			\[ \mu(\bord{ (\gamma X_+)}) \leq C e^{-\delta_\Gamma d(o, \gamma o)}. \]
		\end{lemme}
		
		\begin{proof}
			On a le lemme suivant.
			
			\begin{lemme} \label{pto}
				Soit $X$ un espace métrique, et $\Gamma$ un sous-semi-groupe de Schottky de $\Isom(X)$, pour une partie $X_+ \subset X$.
				Pour un point $o \in X_+ \backslash \left( \cup_{g \text{ générateur}} g X_+ \right)$, on a les équivalences
				\[ \gamma X_+ \cap \gamma' X_+ \neq \emptyset \longeq \left( \gamma \in \gamma' \Gamma \text{ ou } \gamma' \in \gamma \Gamma \right), \]
				\[ \gamma' o \in \gamma X_+ \longeq \gamma' \in \gamma \Gamma, \]
				pour toutes isométries $\gamma$ et $\gamma' \in \Gamma$.
			\end{lemme}
			
			\begin{proof}
				Montrons la première équivalence.
				Soient $\gamma$ et $\gamma'$ deux isométries de $\Gamma$ telles que l'on ait $\gamma X_+ \cap \gamma' X_+ \neq \emptyset$. Soient $g$ et $g'$ les générateurs tels que $\gamma \in g \Gamma$ et $\gamma' \in g' \Gamma$.
				\'Etant donné que les ensembles $g X_+$ et $g' X_+$ sont Gromov-disjoints si $g \neq g'$ et que l'on a les inclusions
				$\gamma X_+ \subseteq g X_+$ et $\gamma' X_+ \subseteq g' X_+$, on a nécessairement $g = g'$.
				Par récurrence, on a bien obtenu que $\gamma \in \gamma' \Gamma \text{ ou } \gamma' \in \gamma \Gamma$.
				La réciproque est claire.
				
				Montrons la deuxième équivalence.
				Soient $\gamma$ et $\gamma'$ deux isométries de $\Gamma$ telles que l'on ait $\gamma' o \in \gamma X_+$.
				On a $\gamma' o \in \gamma' X_+ \cap \gamma X_+$ puisque $o \in X_+$.
				Par l'équivalence précédente, on a donc $\gamma' \in \gamma \Gamma$ ou $\gamma \in \gamma' \Gamma$.
				Supposons que l'on ait $\gamma' \not \in \gamma \Gamma$. On peut alors trouver un générateur $g$ tel que l'on ait $\gamma \in \gamma' g \Gamma$.
				On a ensuite l'inclusion $\gamma X_+ \subseteq \gamma' g X_+$, donc $\gamma' o \in \gamma' g X_+$, puis $o \in g X_+$, ce qui contredit l'hypothèse. Donc on a bien $\gamma' \in \gamma \Gamma$.
				La réciproque est claire.
			\end{proof}
			
			Choisissons un point $o \in X_+ \backslash \left( \cup_{g \text{ générateur}} g X_+ \right)$ (i.e. comme dans le lemme \ref{pto}).
			
			
			Supposons que le semi-groupe $\Gamma$ soit divergent.
			Pour tout $s > \delta_\Gamma$ et pour toute isométrie $\gamma \in \Gamma$, on a alors
			\[ \mu_s( \gamma X_+) = \frac{1}{P_s} \sum_{\gamma' \in \gamma \Gamma} e^{-s d(o, \gamma' o)}. \]
			
			Or, le lemme \ref{cet} nous donne l'inégalité
			\[ d(o, \gamma' \gamma o) \geq d(o, \gamma' o) + d(o, \gamma o) - 2C', \]
			où $C'$ est la constante de contraction du semi-groupe $\Gamma$ pour le point $o$ :
			\[ C' := \sup_{\gamma, \gamma' \in \Gamma} (\gamma^{-1} o | \gamma' o) < \infty. \]
			On obtient alors
			\[ \mu_s( \gamma X_+) \leq \frac{1}{P_s} e^{2sC'} e^{-s d(o, \gamma o)} \sum_{\gamma' \in \Gamma} e^{-s d(o, \gamma' o)} = e^{2sC'} e^{-s d(o, \gamma o)}. \]
			D'où l'inégalité
			\[ \mu_s( \gamma X_+) \leq C_s e^{-s d(o, \gamma o)}. \]
			avec $C_s = e^{2s C'}$.
			En passant à la limite, on obtient l'inégalité voulue
			\[  \mu( \bord{ (\gamma X_+)}) \leq C e^{-\delta_\Gamma d(o, \gamma o)},  \]
			où $C = e^{2\delta_\Gamma C'}$, puisque l'ensemble $\bord{ (\gamma X_+) } \cap \Lambda_\Gamma$ est isolé (c'est-à-dire ouvert et fermé) dans l'ensemble limite $\Lambda_\Gamma$.
			
			Si le semi-groupe n'est pas divergent, on modifie le calcul précédent en conséquent en utilisant l'astuce de Patterson, et on conclut de la même façon.
		\end{proof}
		
		Montrons maintenant que la mesure $\mu$ est majorée pour toutes les boules.
		
		\begin{lemme} \label{mmcg}
			Soit $X$ un espace Gromov-hyperbolique propre à bord compact, 
			et soit $\Gamma$ un semi-groupe d'isométries de $X$,
			de Schottky pour un domaine $X_+$. Soit $o$ le point donné par le lemme \ref{pto}, et soit $\mu$ la mesure de Patterson-Sullivan correspondante.
			Alors il existe une constante $C > 0$ telle que pour toute boule $\beta(\xi, r)$ du bord $\bX$, on ait l'inégalité
			\[ \mu(\beta(\xi, r)) \leq C r^{\delta_\Gamma}. \]
		\end{lemme}
		
		\begin{proof}
			Soit $C'$ la constante donnée par le lemme \ref{mmpn}.
			Soit $\beta(\xi, r)$ une boule. Si la boule n'intersecte pas l'ensemble limite $\Lambda_\Gamma$, on a $\mu(\beta(\xi, r)) = 0$, et il n'y a rien à démontrer. Supposons donc que la boule rencontre l'ensemble limite. 
			On peut alors supposer que l'on a $\xi \in \Lambda_\Gamma$, quitte à recouvrir la boule $\beta(\xi, r)$ par une boule de rayon $2r$ centrée en un point de l'ensemble limite $\Lambda_\Gamma$, et à multiplier la constante $C$ par $2^{\delta_\Gamma}$.
			
			Notons
			\[ \Gamma_k := \{ \gamma \in \Gamma | \gamma \text{ de longueur $k$ en les générateurs} \}, \]
			et $\Gamma_0 := \{id \}$.
			Soit $n$ un entier tel que l'on ait
			\[ \# \{ \gamma \in \Gamma_n | \bord{ (\gamma X_+) } \cap \beta(\xi, r) \neq \emptyset \} = 1, \text{ et} \]
			\[ \# \{ \gamma \in \Gamma_{n+1} | \bord{ (\gamma X_+) } \cap \beta(\xi, r) \neq \emptyset \} \geq 2. \]
			Cet entier existe bien, puisque pour $n = 0$ on a $\bord X_+ \cap \beta(\xi, r) \neq \emptyset$, et puisque l'on a
			\[ \lim_{n \to \infty} \# \{ \gamma \in \Gamma_n | \bord{ (\gamma X_+) } \cap \beta(\xi, r) \neq \emptyset \} = \infty. \]
			En effet, si l'on a $\gamma_n o \xrightarrow[n \to \infty]{} \xi$ pour une suite $(\gamma_n)_{n \in \N}$ d'éléments de $\Gamma$, alors on a $\bord{ (\gamma_n X_+) } \subseteq \beta(\xi, r)$ à partir d'un certain rang par le lemme \ref{cont1} et par Gromov-hyperbolicité. 
			
			Soit $\gamma \in \Gamma_{n}$ tel que $\bord{ (\gamma X_+) } \cap \beta(\xi, r) \neq \emptyset$.
			Comme la mesure $\mu$ est portée l'ensemble limite $\Lambda_\Gamma$, on a
			\[ \mu(\beta(\xi, r)) \leq \mu(\bord{ (\gamma X_+)) } \leq C' e^{- \delta d(o, \gamma o)} \]
			par le lemme \ref{mmpn} et par le choix de $n$.
			Il reste donc à majorer la quantité $e^{ d(o, \gamma o)}$ en fonction de $r$.
			
			Pour tout $\gamma' \in \Gamma_{n+1}$ tel que $\bord{ (\gamma' X_+) } \cap \beta(\xi, r) \neq \emptyset$,
			on a $\gamma' \in \gamma \Gamma$. Soient alors $g \neq g'$ deux générateurs tels que l'on ait
			\[ \bord{ (\gamma g X_+) } \cap \beta(\xi, r) \neq \emptyset \quad \text{ et } \quad \bord{ (\gamma g' X_+) } \cap \beta(\xi, r) \neq \emptyset. \]
			Soit $C'' > 0$ une constante telle que les parties $g X_+$ pour $g$ parcourant les générateurs, et $X_+$, soient deux à deux $C''$-Gromov disjointes.
			L'image $\gamma^{-1} \beta(\xi, r)$ de la boule $\beta(\xi, r)$ par l'isométrie $\gamma^{-1}$ rencontre les ensembles $\bord{ (g X_+) }$ et $\bord{ (g' X_+) }$, donc il existe des points $\eta$ et $\eta'$ de $\gamma^{-1} \beta(\xi, r)$ tels que l'on ait l'inégalité
			\[ (\eta | \eta') \leq C''. \]
			On a alors
			\begin{align*}
				C''	&\geq (\eta | \eta') \\
					&= (\gamma \eta | \gamma \eta') + (\gamma^{-1} o | \eta) + (\gamma^{-1} o | \eta') - d(o, \gamma o) \\
					&\geq - \log(r) - \delta + 0 + 0 - d(o, \gamma o)
			\end{align*}
			d'où l'inégalité
			\[ e^{-d(o, \gamma o)} \leq e^{C'' + \delta} r. \]
			On obtient donc l'inégalité escomptée avec $C = C' e^{C'' + \delta}$.
		\end{proof}

		Les lemmes \ref{mmcg} et \ref{mnd} donnent l'inégalité
		\[ \dim_{vis} (\Lambda_\Gamma) \geq \delta_\Gamma, \]
		ce qui termine cette preuve de la proposition \ref{pschottky}.
	\end{proof}
	
	On peut maintenant facilement retrouver la généralisation du résultat de Paulin.
	
	\begin{proof}[Preuve du corollaire \ref{cpaulin_}]
		Soit $\epsilon > 0$.
		Par le théorème \ref{thm0}, il existe un sous-semi-groupe $\Gamma'$ de Schottky de $\Gamma$ tel que l'on ait
		$\delta_{\Gamma'} \geq h_\Gamma - \epsilon$.
		Par les propositions \ref{pschottky} et \ref{pradial}, on a donc les inégalités
		\[ \dv(\Lambda_\Gamma^c) \geq \dv (\Lambda_{\Gamma'}^c) = \dv (\Lambda_{\Gamma'}) = \delta_{\Gamma'} \geq h_\Gamma - \epsilon. \]
		Ceci étant vrai pour tout $\epsilon > 0$, on en déduit l'inégalité
		\[ \dv(\Lambda_\Gamma^c) \geq h_\Gamma. \]
		L'autre inégalité est donnée par la proposition \ref{if}.
	\end{proof}
	
	\subsection{Semi-groupes de développement $\beta$-adique} \label{dev-beta}
	
	Dans cette sous-section, nous obtenons une application du corollaire \ref{cpaulin_} aux semi-groupes de développement en base $\beta$.
	
	Le \defi{semi-groupe de développement en base $\beta \in \C$ avec ensemble de chiffres $A$} est le semi-groupe engendré par les applications affines :
	\[ x \mapsto x / \beta + t, \]
	où $t \in A$, pour une partie finie $A$ de $\C$.
	
	On peut voir ce semi-groupe comme un sous-semi-groupe de $SL_2(\C)$.
	En effet, à l'application $x \mapsto x / \beta + t$, on peut associer la matrice 
	$\begin{pmatrix}
		\frac{1}{\sqrt{\beta}}	& t \sqrt{\beta} \\
		0				& \sqrt{\beta}
	\end{pmatrix}$, où $\sqrt{\beta}$ est une racine carrée de $\beta$.
	On a donc une action par isométrie du semi-groupe sur l'espace $X = \hypt := \{ z + \tau j | z \in \C, \tau > 0\}$ (vu comme partie de l'ensemble des quaternions), dont le bord $\bX$ s'identifie à $\C \cup \{ \infty\}$.
	L'action de l'application $x \mapsto x/\beta + t$ sur $\hypt$ est donnée par
	\[ (x \mapsto x/\beta + t) . (z + \tau j) = (z/\beta + t) + (\tau/\abs{\beta})j. \]
	L'ensemble limite du semi-groupe est alors exactement l'ensemble des nombres complexes qui admettent un développement $\beta$-adique n'ayant qu'un seul chiffre avant la virgule et avec ensemble de chiffres $A$.
	
	Tout ceci fonctionne également en remplaçant le corps $\C$ par $\R$.
	
	\begin{define}
		On appelle \defi{nombre de Salem} généralisé un entier algébrique $\beta \in \C$ de module strictement supérieur à $1$, dont tous les conjugués sont de modules inférieurs ou égaux à $1$, sauf éventuellement son conjugué complexe.
		On appelle \defi{nombre de Pisot} généralisé un entier algébrique $\beta \in \C$ de module strictement supérieur à $1$, dont tous les conjugués sont de modules \emph{strictement} inférieurs à $1$, sauf éventuellement son conjugué complexe.
	\end{define}
	
	\begin{rem}
		Dans la définition classique de nombres de Pisot et de Salem, on demande à ce que le nombre soit un réel $\beta > 1$, mais tout ce que l'on verra est valable pour cette définition plus générale.
	\end{rem}
	
	\begin{prop} \label{psalem}
		Soit $\Gamma$ le semi-groupe engendré par les applications
		\[ x \mapsto x / \beta + t \]
		où $t \in A$ pour une partie finie $A \subset \Q(\beta)$.
		Si $\beta$ est un nombre de Salem généralisé, alors on a l'égalité
		\[ \dim_{H}(\Lambda_\Gamma) = \delta_\Gamma. \]
	\end{prop}
	
	\begin{figure}[H]
		\centering
		\caption{Développement en base $\beta = \varphi$ (le nombre d'or, qui est un nombre de Pisot), avec ensemble de chiffres $A = \{0, 1\}$.}
		\includegraphics{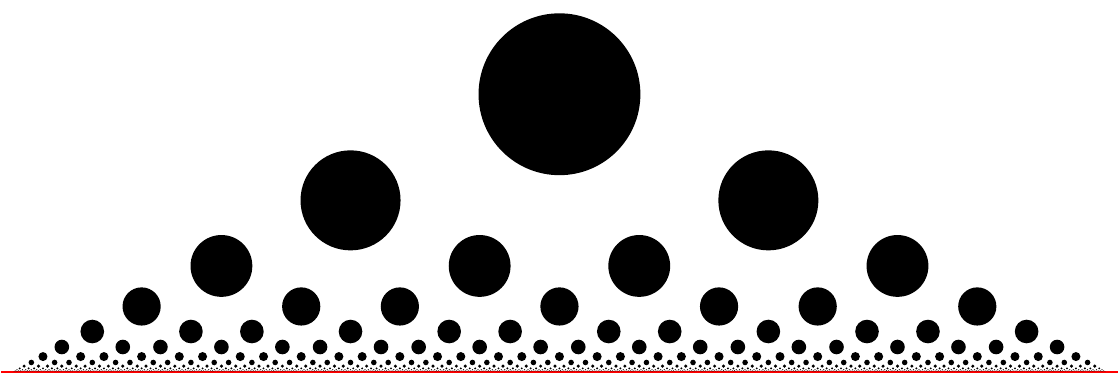}
	\end{figure}
		
	Avant de démontrer la proposition, on a le lemme suivant. \\ 
	
%
	

	Le lemme qui suit dit que si l'on regarde l'orbite d'une boule par le semi-groupe, alors le nombre de chevauchements en un point donné n'est pas trop grand par rapport à la distance au point base $o$.
	
	\begin{lemme} \label{poly}
		Sous les hypothèses de la proposition \ref{psalem},
		il existe un entier $r$ tel que l'on ait 
		\[ \# \{ \gamma \in \Gamma | d(\gamma j, x) \leq 1 \} = \underset{d(j, x) \to \infty}{O} (d(j, x)^r), \]
		pour $x \in \hypt$. 
	\end{lemme}
	
	\begin{rem} \label{non-sep}
		Si l'entier algébrique $\beta$ est de Pisot, alors le semi-groupe $\Gamma$ est même séparé (voir la condition de séparation de Lalley \cite{lalley}), et donc on peut prendre $r = 0$. 
	\end{rem}
	
		
	
	
	\begin{proof}[Preuve du lemme \ref{poly}]
		
		Le resultat suivant permet de majorer le paramètre de translation des isométries du semi-groupe $\Gamma$.
		
		\begin{souslemme} \label{tb}
			Si $\Gamma$ est un sous-semi-groupe de $\Aff(\C)$ engendré par des application $x \mapsto x/\beta + t$, pour $t \in A$, avec $A$ partie finie de $\C$ et $\beta \in \C$ tel que $\abs{\beta} > 1$, alors il existe une constante $C$ telle que pour toute application $x \mapsto \alpha x + t \in \Gamma$, on ait
			$\abs{t} \leq C$.
		\end{souslemme}
		
		\begin{proof}
			Un élément du semi-groupe $\Gamma$ s'écrit
			\[ x \mapsto x/\beta^n + \sum_{k = 0}^{n-1} \frac{t_k}{ \beta^k}, \]
			avec $t_k \in A$.
			On a alors la majoration
			\[ \abs{ \sum_{k = 0}^{n-1} \frac{t_k}{ \beta^k } } \leq  \max_{t \in A} \abs{t} \sum_{k = 0}^{n-1} \frac{1}{ \abs{\beta}^k } \leq \frac{\max_{t \in A} \abs{t}}{ 1 - \frac{1}{\abs{\beta}} }. \]
		\end{proof}
		
		Notons
		\[ \Gamma_x := \{ \gamma \in \Gamma | d(\gamma j, x) \leq 1 \}. \]
		On a alors les resultats suivants.
		
		\begin{souslemme} \label{lm}
			Il existe une constante $C$ telle que pour tout $x \in \hypt$, 
			toute isométrie de $\Gamma_x$ est de longueur au moins $\frac{d(j, x) - C}{\log(\beta)}$ et au plus $\frac{d(j, x)+C}{\log(\beta)}$ en les générateurs.
		\end{souslemme}
		
		\begin{proof}
			En effet, on a l'inégalité triangulaire
			\[ d(j, x) - d(\gamma j, x) \leq d(j, \gamma j) \leq d(\gamma j, x) + d(j, x). \]
			Par ailleurs, si l'on écrit $\gamma j = \abs{\beta}^{-n} j +t$, on a
			\[ d(j, \abs{\beta}^{-n} j) - d(j, j+t) \leq d(j, \gamma j) \leq d(j, \abs{\beta}^{-n} j) + d(j, j+t) \]
			où $n$ est la longueur de $\gamma$.
			Ensuite, d'après le sous-lemme \ref{tb}, la quantité $d(j, j + t)$ est bornée par une constante $C'$ indépendante de $\gamma$ et de $x$. Et on vérifie que l'on a $d(j, \abs{\beta}^{-n} j) = n\log(\beta)$ pour la métrique usuelle de $\hypt$.
			On obtient alors l'encadrement
			\[ d(j, x) - C' - 1 \leq n \log(\beta) \leq d(j, x) + C' + 1. \]
			D'où l'encadrement sur la longueur $n$ de $\gamma$ avec $C = C'+1$.
			
		\end{proof}
				
		Posons
		\[ n_x := \floor{\frac{d(j, x)+C}{\log(\beta)}}, \]
		la plus grande longueur possible des éléments de $\Gamma_x$.
		On a alors le resultat suivant.
		
		\begin{souslemme} \label{bplus}
			Il existe une constante $C$ telle que pour tout $x \in \hypt$ on ait 
			\[ \diam \left( \beta^{n_x} \Gamma_x 0 \right) \leq C, \]
			où $0$ est le point du bord $0 = 0 + 0j \in \bord{ \hypt }$.
		\end{souslemme}
		
		\begin{proof}
			L'application $y \mapsto \beta^{n_x} y$ étant une isométrie, on a 
			\[ d(\beta^{n_x} \gamma j , \beta^{n_x} x) = d(\gamma j , x) \leq 1, \]
			pour tout $\gamma \in \Gamma_x$.
			D'autre part, si l'on écrit $\gamma j = \gamma 0 + \abs{\beta}^{-n} j$, alors on a 
			\[ \beta^{n_x} \gamma j = \beta^{n_x} \gamma 0 + \abs{\beta}^{n_x - n} j, \]
			où $n$ est la longueur de $\gamma$.
			Or, par le sous-lemme \ref{lm}, il existe une constante $C$ telle que pour $\gamma \in \Gamma_x$,
			on ait $\abs{n - n_x} \leq C$, où $n$ est la longueur de $\gamma$.
			La distance
			\[ d( \beta^{n_x} \gamma 0, \beta^{n_x} x) \]
			est donc bornée indépendamment de $x$ et $\gamma \in \Gamma_x$, par inégalité triangulaire.
		\end{proof}
		
		Quitte à multiplier la partie $A$ par les dénominateurs (ce qui ne change pas la conclusion du sous-lemme), 
		on peut supposer que l'on a $A \subseteq \Z[\beta]$.
		La quantité
		\[ \beta^{n_x} \gamma 0 \in \C \]
		est alors un polynôme en $\beta$ à coefficients entiers, pour tout élément $\gamma \in \Gamma_x$.
		
		Construisons alors un espace $E$ (indépendant du point $x$), dans lequel l'anneau $\Z[\beta]$ sera discret.
		
		Soit $\Part$ l'ensemble des valeurs absolues archimédiennes du corps $k := \Q(\beta)$, à équivalence près.
		L'ensemble $\Part$ est fini (de cardinal majoré par le degré de $\beta$), et l'on peut poser
		\[ E := \prod_{v \in \Part} k_v, \]
		où $k_v$ est le complété du corps $k$ pour la valeur absolue $v$.
		
		On a alors $E = \R^r \times \C^s$, où $r$ est le nombre conjugués réels de $\beta$, et $2s$ est son nombre de conjugués complexes.
		
		On a maintenant le résultat suivant,
		
		\begin{prop} \label{discret}
			L'anneau $\Z[\beta]$ est discret dans l'espace $E$.
		\end{prop}
		
		qui découle de la formule du produit, qui est un résultat classique de théorie des nombre (voir par exemple \cite{lang}).
		
		\begin{rem}
			On peut même montrer que $\Z[\beta]$ est un réseau co-compact de l'espace $E$, mais nous n'en aurons pas besoin.
		\end{rem}
		
		\begin{prop}[Formule du produit]
			Soit $k$ un corps de nombres. Pour tout $x \in k \backslash \{ 0 \}$, on a
			\[ \prod_{v \in \Part_k} \abs{x}_v = 1, \]
			où $\Part_k$ est l'ensemble des valeurs absolues de $k$ à équivalence près, où l'on a choisis les valeurs absolues \og standards \fg\ dans chaque classe d'équivalence.
		\end{prop}
		
		
		\begin{proof}[Preuve de la proposition \ref{discret}]
			Il suffit de montrer que le point $0 \in \Z[\beta]$ est isolé. On aura alors bien la discrétude puisque l'ensemble $\Z[\beta]$ est un groupe additif.
			Soit $B$ la boule de $E$ de centre $0$ et de rayon $1/2$.
			Si un point $x$ est dans $B \cap \Z[\beta]$, alors on a
			\[ \prod_{v \in \Part_k} \abs{x}_v \leq \prod_{v \in \Part} \abs{x}_v \leq (\frac{1}{2})^{r+s} < 1, \]
			puisque pour toute valeur absolue ultramétrique $v$, on a $\abs{\beta}_v \leq 1$ et donc $\abs{x}_v \leq 1$, étant donné que $\beta$ est un entier algébrique.
			D'après la formule du produit, on a donc $x = 0$. D'où la discrétude de l'anneau $\Z[\beta]$ dans l'espace $E$.
		\end{proof}
		
		L'ensemble de valeurs absolues $\Part$ peut s'écrire
		\[ \Part := \Part_- \cup \Part_0 \cup \Part_+, \]
		où
		\begin{itemize}
			\item $\Part_+$ est l'ensemble des valeurs absolues $v \in \Part$ telles que $\abs{\beta}_v > 1$,
			\item $\Part_0$ est l'ensemble des valeurs absolues $v \in \Part$ telles que $\abs{\beta}_v = 1$,
			\item $\Part_-$ est l'ensemble des valeurs absolues $v \in \Part$ telles que $\abs{\beta}_v < 1$.
		\end{itemize}
		
		On peut alors décomposer cet espace $E$ dans lequel l'anneau $\Z[\beta]$ est discret en 3 morceaux :
		\[ E = E_+ \times E_0 \times E-, \]
		où $E_- := \prod_{v \in \Part_-} k_v$, $E_0 := \prod_{v \in \Part_0} k_v$, et $E_+ := \prod_{v \in \Part_+} k_v$.
		
		Le nombre $\beta$ étant de Salem généralisé, il existe une unique valeur absolue $v$ telle que $\abs{\beta}_v > 1$.
		On a donc $E_+ = \R$ ou $\C$ selon que le nombre $\beta$ est réel ou complexe.
		
		Montrons maintenant que la partie $\beta^n \Gamma_x 0$ est suffisamment bornée dans l'espace $E$.
		
		\begin{souslemme}
			Il existe une constante $C > 0$ telle que pour tout point $x \in \hypt$, il existe des compacts $K_+$, $K_0$ et $K_-$ respectivement de $E_+$, $E_0$ et $E_-$, de diamètres majorés par $C$, tels que l'on ait l'inclusion
			\[ \beta^{n_x} \Gamma_x 0 \subseteq \Z[\beta] \cap K_+ \times ((n_x + 1) K_0) \times K_-. \]
		\end{souslemme}
		
		\begin{proof}
			Soit un point $x \in \hypt$ et une isométrie $\gamma \in \Gamma_x$.		
			L'expression $\beta^{n_x} \gamma 0$ est un polynôme en $\beta$ à coefficients dans $\Z$, mais c'est aussi un polynôme en $\beta$, de degré au plus $n_x$, à coefficient dans $A$.
			
			Pour obtenir le compact $K_-$, il suffit alors de remarquer que si $\gamma$ est un nombre réel ou complexe avec $\abs{\gamma} < 1$, alors pour toute suite $(a_k)_{k \in \N} \in A^\N$, on a
			\[ \abs{\sum_{k=0}^{n_x} a_k \gamma^k} \leq \max_{a \in A} \abs{a} \sum_{k=0}^\infty \abs{\gamma}^k = \frac{\max_{a \in A} \abs{a}}{1 - \abs{\gamma}}. \]
			Si maintenant $\gamma$ est un nombre de module $1$, alors on a
			\[ \abs{\sum_{k=0}^{n_x} a_k \gamma^k}  \leq \max_{a \in A} \abs{a} \sum_{k=0}^{n_x} 1 = (n_x+1) \max_{a \in A} \abs{a}, \]
			pour toute suite $(a_k)_{k \in \N} \in A^\N$, ce qui nous donne le compact $K_0$.
			
			Pour finir, le lemme \ref{bplus} permet d'obtenir le compact $K_+$ dont le diamètre est indépendant de $x$, et les compacts $K_0$ et $K_-$ ne dépendent pas du point $x \in \hypt$.
		\end{proof}
		
		Finissons la preuve de la proposition \ref{poly}.
		Le groupe additif $\Z[\beta]$ étant discret dans l'espace $E$, il existe un réel $\epsilon > 0$ tel que les boules de $E$ centrées aux points de $\Z[\beta]$ et de rayons $\epsilon$ sont disjointes.
		La quantité
		\[ \# \Z[\beta] \cap \left( K_+ \times ((n_x + 1) K_0) \times K_- \right) \cdot \vol(B(j, \epsilon)) \]
		est donc majorée par le volume d'un $\epsilon$-voisinage du compact $K_+ \times ((1+n) K_0) \times K_-$.
		
		On obtient donc la majoration
		\begin{align*}
			 \# \Gamma_x 0 &= \# \beta^{n_x} \Gamma_x 0 \\
			 			&\leq \# \Z[\beta] \cap \left( K_+ \times ((n_x + 1) K_0) \times K_- \right) \\ 
						&\leq \frac{C (n_x + 1)^{p_0}}{\vol(B(j, \epsilon))},
		\end{align*}
		pour une constante $C$, et pour $p_0$ le nombre de conjugués de $\beta$ de module $1$ (en comptant bien les conjugués complexes).
		D'autre part, on a $n_x = O(d(j, x))$. Par le sous-lemme \ref{lm}, on a alors, pour une constante $C$,
		\[ \# \Gamma_x \leq \sum_{n = n_x - C}^{n_x} \# \Gamma_x 0 = (C+1) \# \Gamma_x 0 = O(d(j, x)^{p_0}). \]
		
		Ceci termine la preuve du lemme \ref{poly}.
	\end{proof}
		

	\begin{proof}[Preuve de la proposition \ref{psalem}]
		Le cas où $A$ est de cardinal inférieur ou égal à $1$ est clair : on a facilement $\delta_\Gamma = 0 = \dim_H(\Lambda_\Gamma)$.
		Supposons donc que l'ensemble $A$ est de cardinal au moins $2$. L'ensemble limite contient alors au moins deux points.
		D'après la proposition \ref{pradial}, l'ensemble limite du semi-groupe est radial : $\Lambda_\Gamma = \Lambda_\Gamma^c$.
		Or, d'après le corollaire \ref{cpaulin_}, l'ensemble limite radial a une dimension de Hausdorff égale à $h_\Gamma$ (puisque la dimension de Hausdorf coùncide avec la dimension visuelle, voir \cite{ghys} pour plus de détails).
		Il suffit donc de montrer l'égalité $h_\Gamma = \delta_\Gamma$ pour conclure.
		
		Soit $S$ une partie séparée et $1$-couvrante de $\Gamma$.
		D'après le lemme \ref{poly}, il existe alors un entier $r$ et une constante $C$ tels que pour tout réel $R$ assez grand on ait
		\[ \# \{ \gamma \in \Gamma | d(\gamma j, j) \leq R \} \leq C R^r \# \{ \gamma \in S | d(\gamma j, j) \leq R \} , \]
		On a alors
		\begin{align*}
			\delta_\Gamma	&= \limsup_{n \to \infty} \frac{1}{n} \log(\# \{ \gamma \in \Gamma | d(j, \gamma j) \leq n \} \\
						&\leq \limsup_{n \to \infty} \frac{1}{n} \log(\# \{ \gamma \in S | d(j, \gamma j) \leq n \} \\
						&= \delta_S,
		\end{align*}
		puisque $\limsup_{n \to \infty} \frac{1}{n} \log(Cn^r) = 0$.
		D'autre part, on a $\delta_S = h_\Gamma$ puisque $S$ est une partie séparée et couvrante de $\Gamma$.
		On a donc obtenu l'inégalité $h_\Gamma \geq \delta_\Gamma$, et l'autre inégalité $h_\Gamma \leq \delta_\Gamma$ est claire.
		Cela termine la preuve de la proposition \ref{psalem}.
	\end{proof}
	
	
	\begin{conj}[Conjecture de Furstenberg modifiée et généralisée]
		Soit $\Gamma$ un sous-semi-groupe de type fini de $SL(2, \R)$ dont l'ensemble limite n'est pas réduit à un seul point.
		Alors on a l'égalité
		\[ h_\Gamma  = \min( \delta_\Gamma, 1). \]
	\end{conj}
	
	La conjecture originale posait plutôt la question de l'égalité $\dim_H(\Lambda_\Gamma) = \min(\delta_\Gamma, 1)$. 
	L'avantage de cette formulation, qui est équivalente grâce au corollaire \ref{cpaulin_} quand le semi-groupe est contractant, est qu'elle ne fait plus intervenir le bord. 
	
	Kenyon attribue cette conjecture à Furstenberg, dans le cas particulier du semi-groupe engendré par les trois applications
	\[
		\left\{\begin{array}{cl}
			x &\mapsto x/3 \\
			x &\mapsto x/3 + t \\
			x &\mapsto x/3 + 1
		\end{array}\right.
	\]
	où $t$ est un réel. Cette question, sur cet exemple particulier, est toujours ouverte à ma connaissance, bien que l'on sache dire pas mal de choses (voir \cite{ken}).
	Sur cet exemple, la conjecture de Furstenberg se résume à déterminer si l'on a l'égalité $\dim_H(\Lambda_\Gamma) = 1$ quand $t$ est irrationnel, puisque dans ce cas le semi-groupe est libre, ce qui donne $\delta = 1$. Le cas où $t$ est rationnel a été résolu par Kenyon (et est aussi conséquence de mes résultats puisque dans ce cas le semi-groupe est séparé). Avec mes travaux, la conjecture de Furstenberg se ramène à déterminer si l'on a l'égalité $h_\Gamma = 1$ pour tout $t$ irrationnel.
	

	\begin{figure}[H]
		\centering
		\caption{Développement en base $\beta = 3$, avec ensemble de chiffres $A = \{0, \frac{2}{3}, 1\}$.} \label{2/3}
		\includegraphics{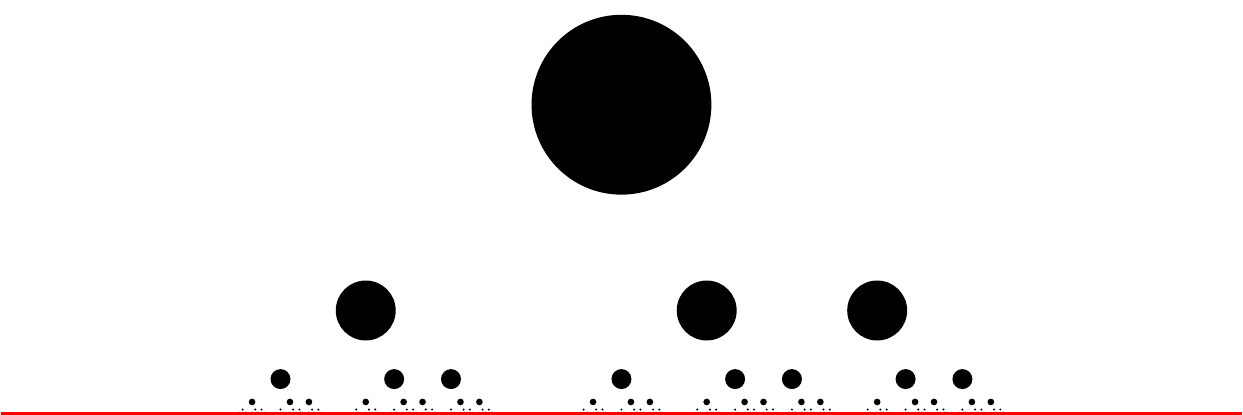}
	\end{figure}
	
%

	
	\section{Sous-groupes de Schottky} \label{groupes}
	
	Les résultats sur les semi-groupes permettent d'obtenir un résultat sur les groupes : voir corollaire~\ref{ss-groupes} ci-dessous.
	Plus précisément, nous parvenons à construire des sous-groupes de Schottky à partir de sous-semi-groupes de Schottky, et ceci nous donne des groupes de Schottky ayant un \og gros \fg exposant critique.
	
	Les groupes de Schottky sont définit de façon similaire aux semi-groupes de Schottky, il s'agit des groupes de type fini dont les générateurs jouent au \og ping-pong \fg :
	
	\begin{define}
		Soit $X$ un espace métrique.
		On dit qu'un ensemble $G$ d'isométries de $X$ engendre un \defi{groupe de Schottky} si c'est un ensemble fini et qu'il existe des parties
		$X_\gamma^+$ et $X_\gamma^-$ pour tout $\gamma \in G$ qui sont toutes deux-à-deux Gromov-disjointes, et telles que
		pour tout $\gamma \in G$ on ait $\gamma (X \backslash X_\gamma^-) \subseteq X_\gamma^+$, et telles que l'ensemble $X \backslash ( \bigcup_{\gamma \in G} X_\gamma^+ \cup X_\gamma^- )$ soit d'intérieur non vide.
	\end{define}
	
	\begin{figure}[H]
		\centering
		\caption{Un groupe de Schottky engendré par deux isométries $g$ et $h$.}
		\includegraphics{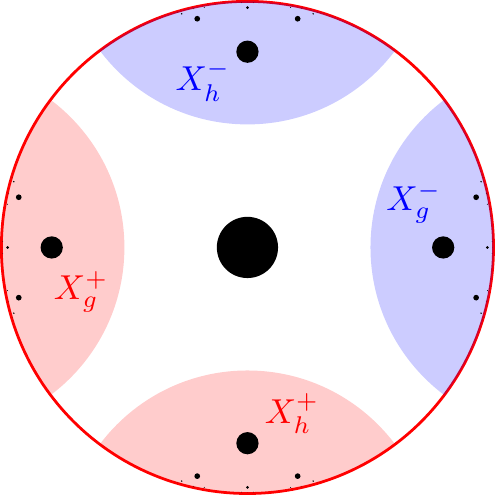}
%
%
%
%
%
%
	\end{figure}
	
	\begin{cor} \label{ss-groupes}
		Soit $X$ un espace Gromov-hyperbolique propre et soit $\Gamma$ un groupe discret et sans torsion d'isométries de $X$ ne fixant pas de point au bord, alors on a
		\[ \sup_{\Gamma' < \Gamma \atop \Gamma' \text{ groupe de Schottky}} \delta_{\Gamma'} \geq \frac{1}{2} \delta_\Gamma, \]
		où la borne supérieure est prise sur l'ensemble des sous-groupes de Schottky du groupe $\Gamma$.
	\end{cor}
	
	Ce résultat est à relier à une question dont parle M. Kapovich dans son article \cite{kap} (voir Problem 10.27, The gap problem).
	Avec ses notations, mon résultat donne l'inégalité $d_n \geq n/2$. La question est de savoir si l'on peut atteindre $d_n = n$ ou non.
	
	\begin{proof}[Preuve du corollaire \ref{ss-groupes}]
		Le théorème \ref{thm1} permet de trouver un sous-semi-groupe $\Gamma^c$ du groupe $\Gamma$ qui soit contractant pour des parties $X_+$ et $X_- \subset X$, et d'exposant critique $\delta_{\Gamma^c} = \delta_\Gamma$.
		On a alors le lemme suivant.
		
		\begin{lemme} \label{scic}
			Soit $X$ un espace métrique de point base $o$, et soit $\Gamma^c$ un sous-semi-groupe d'un groupe discret et sans torsion $\Gamma$ d'isométries de $X$, qui soit contractant pour des parties $X_+$ et $X_-$ de $X$.
			Pour tout $\epsilon > 0$, et pour un entier $n$ arbitrairement grand, il existe une partie $S_n$ de $\Gamma^c$ telle que
			\begin{itemize}
				\item $S_n$ engendre un semi-groupe de Schottky pour la partie $X_+$,
				\item $S_n^{-1}$ engendre un semi-groupe de Schottky pour la partie $X_-$,
				\item $\# S_n \geq e^{n (\delta_{\Gamma^c} - \epsilon)}$,
				\item $S_n \subseteq A_n$, où
			\end{itemize}
			\[ A_n := \{ \gamma \in \Isom(X) | d(o, \gamma o) \in [n, n+1[ \}. \]
		\end{lemme}
		
		\begin{proof}
			D'après le lemme \ref{sc}, il existe un réel $r$ et un entier $n_0$, tels que pour toute partie $r$-séparée $S$ de $\Gamma^c$ et pour tout $n \geq n_0$, le semi-groupe engendré par $S \cap A_n$ soit de Schottky pour la partie $X_+$.
			En faisant de même pour le semi-groupe inverse $\left( \Gamma^c \right)^{-1}$, on obtient un réel $r'$ et un entier $n_0'$.
			
			Soit $S$ une partie $r$-séparée et couvrante du semi-groupe contractant $\Gamma^c$.
			La partie $(S \cap A_n)^{-1}$ est séparée, puisque faisant partie du groupe discret et sans torsion $\Gamma$.
			Par le lemme des tiroirs, il existe donc une constante $C > 0$ (dépendant de la séparation du groupe $\Gamma$ et du réel $r'$), et une partie $r'$-séparée $S_n^{-1}$ de $(S \cap A_n)^{-1}$, telles que l'on ait
			\[ \#S_n \geq C \#(S \cap A_n). \]
			Pour $\epsilon > 0$ fixé, on peut alors trouver un entier $n$ arbitrairement grand, pour lequel on a l'inégalité
			\[ \# S_n \geq e^{n (\delta_{\Gamma^c} - \epsilon)}, \]
			puisque l'on a $h_{\Gamma^c} = \delta_{\Gamma^c}$, par séparation du sous-semi-groupe $\Gamma^c$ du groupe discret $\Gamma$. Et par le lemme \ref{sc} les parties $S_n$ et $S_n^{-1}$ engendrent chacune un semi-groupe de Schottky respectivement pour les parties $X_+$ et $X_-$, puisqu'elles sont respectivement $r$-séparées et $r'$-séparée.
		\end{proof}
		
		Construisons alors un groupe de Schottky de la façon suivante : 
		
		\begin{figure}[H]
			\centering
			\caption{Construction d'un groupe de Schottky à partir d'un semi-groupe de Schottky dont l'inverse est aussi un semi-groupe de Schottky.}
			\includegraphics{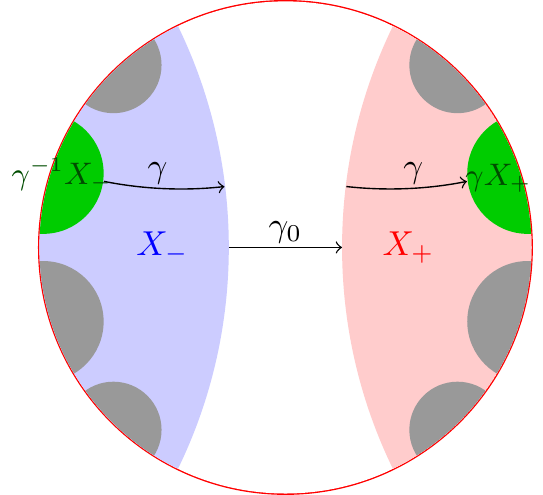}
%
%
%
%
%
%
%
%
%
%
		\end{figure}
		
		\begin{lemme} \label{scgc}
			Soit $X$ un espace métrique, et soit $\gamma_0$ une isométrie de $X$ telle que l'ensemble $\{ \gamma_0 \}$ soit contractant pour des domaines $X_-$ et $X_+$.
			Soit $S$ les générateurs d'un semi-groupe de Schottky pour la partie $X_+$, tel que l'inverse $S^{-1}$ engendre un semi-groupe de Schottky pour la partie $X_- $. 
			Posons
			\[ G := \{ \gamma \gamma_0 \gamma | \gamma \in S \}. \]
			Alors $G$ engendre un groupe de Schottky.
		\end{lemme}
		
		\begin{proof}
			Pour toute isométrie $\gamma \in S$, on a l'inclusion
			\[ \gamma \gamma_0 \gamma (X \backslash \gamma^{-1} X_-) = \gamma \gamma_0 (X \backslash X_-) \subseteq \gamma (X_+) = \gamma X_+, \]
			et les parties $\gamma^{-1} X_-$ pour $\gamma$ décrivant $S$, et $\gamma' X_+$ pour $\gamma'$ décrivant $S$, sont toutes deux à deux Gromov-disjointes.
		\end{proof}
		
		Ainsi, en appliquant le lemme \ref{scgc} avec la partie $S_n$ donnée par le lemme \ref{scic} et avec un élément $\gamma_0 \in \Gamma^c$ quelconque, on obtient une partie $G_n := \{ \gamma \gamma_0 \gamma | \gamma \in S_n \}$ qui engendre un groupe $\Gamma_n$ de Schottky. Il reste maintenant à minorer l'exposant critique du groupe obtenu.
		
		L'inégalité triangulaire donne
		$d(o, \gamma \gamma_0 \gamma o) \leq 2d(o, \gamma o) + d(o, \gamma_0 o) \leq 2(n+1) + d(o, \gamma_0 o)$.
		Par le lemme \ref{mnec}, on a donc la minoration
		\[ \delta_{\Gamma_n} \geq \frac{1}{2(n+1)+ d(o, \gamma_0 o)} \log (\# \{ G_n \}) \geq \frac{n(\delta_\Gamma - \epsilon)}{2(n+1)+ d(o, \gamma_0 o)}, \]
		puisque l'ensemble $G_n$ engendre un semi-groupe de Schottky (donc libre) qui est un sous-semi-groupe du groupe $\Gamma_n$.
		
		L'entier $n$ pouvait être choisi arbitrairement grand, et le réel $\epsilon > 0$ était arbitraire, donc on obtient bien l'inégalité annoncée, ce qui termine la preuve du corollaire \ref{ss-groupes}.
	\end{proof}
	
	\begin{rem}
		La minoration de la borne supérieure des exposants critiques des sous-groupes de Schottky par $\frac{1}{2} \delta_\Gamma$ n'est pas optimale. En pratique, les groupes de Schottky construits dans cette preuve ont des exposants critiques qui se rapprochent mieux que cela de l'exposant critique total $\delta_\Gamma$.
	\end{rem}
	
	\section{Caractérisation de l'entropie} \label{car-ent}
	
	Le corollaire suivant du théorème \ref{thm0} donne en particulier que l'exposant critique d'un semi-groupe séparé, qui était définit comme une limite supérieure d'une certaine quantité, est en fait une vraie limite. Ceci généralise un résultat que Roblin a établit pour un groupe discret d'isométries d'un espace CAT(-1) (mais avec une conclusion plus forte), voir \cite{robl}.
	
	\begin{cor} \label{lim_}
		Soit $X$ un espace Gromov-hyperbolique propre, et soit $\Gamma$ un semi-groupe d'isométries de $X$ dont l'ensemble limite contient au moins deux points.
		Alors on a
		\[ h_\Gamma = \lim_{n \to \infty} \frac{1}{n} \log( \#\{ \gamma \in S | d(o, \gamma o) \leq n \}), \]
		pour toute partie $S$ séparée et couvrante de $\Gamma$.
	\end{cor}
	
	\begin{proof}
		Soit $\epsilon > 0$.
		Par définition de la limite supérieure, il existe un entier $n_0$ tel que pour tout $n \geq n_0$, on ait
		\[ \frac{1}{n} \log(\#\{ \gamma \in S | d(o, \gamma o) \leq n \}) \leq \delta_S + \epsilon = h_\Gamma + \epsilon. \]
		
		Montrons l'autre sens.
		D'après le théorème \ref{thm1}, il existe un sous-semi-groupe $\Gamma^c$ de $\Gamma$ qui est contractant et d'exposant critique $\delta_{\Gamma^c} = \delta_\Gamma$, puisque les semi-groupes $\Gamma$ et $\Gamma^c$ sont séparés.
		
		\begin{souslemme}
			Soit $X$ un espace métrique et $A \subseteq X$ une partie.
			Soit $S$ une partie séparée et couvrante de $A$ et $r > 0$ un réel. Alors il existe une partie $S' \subseteq S$ telle que $S'$ est une partie $r$-séparée et couvrante de $A$.
		\end{souslemme}
		
		\begin{proof}
			Par récurrence ordinale, on construit une suite $(x_i)$ d'éléments de $S$ indexée par les ordinaux, en choisissant un élément
			\[ x_i \in S \backslash \bigcup_{j < i} B(x_j, r) \]
			tant que l'ensemble est non vide.
			Cela termine nécessairement puisque la suite ainsi construite ne peut pas avoir un cardinal strictement supérieur à celui de $S$.
			La partie $S'$ constituée des éléments de la suite est alors $r$-séparée, et elle est $(C+r)$-couvrante de $A$, où $C$ est telle que $S$ est $C$-couvrante de $A$. En effet, si $x \in A$, il existe un élément $y \in S$ tel que $d(x, y) \leq C$. Et par construction, il existe un élément $z \in S'$ tel que $d(y, z) \leq r$. On a alors bien trouvé $z \in S'$ tel que $d(x, z) \leq C + r$.
		\end{proof}
		
		D'après le lemme \ref{sc}, quitte à remplacer la partie $S$ par une sous-partie suffisamment séparée et encore couvrante de $\Gamma^c$, il existe un entier $k_0$ tel que pour tout $k \geq k_0$, l'ensemble $S \cap A_k$ engendre un semi-groupe de Schottky $\Gamma_k$, où
		\[ A_k := \{ \gamma \in \Isom(X) | d(o, \gamma o) \in [k, k+1[ \}. \]
		
		Pour tout $n \geq k+1 \geq k_0+1$, en utilisant l'inégalité triangulaire et le fait que le semi-groupe $\Gamma_k$ soit libre, on obtient l'inégalité
		\[ \#\{ \gamma \in \Gamma_k | d(o, \gamma o) \leq n \} \geq \# \{ \gamma \in \Gamma_k | \gamma \text{ de longueur } \floor{\frac{n}{k+1}} \} = \left( \# (S \cap A_k) \right)^{\floor{\frac{n}{k+1}}}. \]
		On a donc l'inégalité
		\[ \frac{1}{n} \log(\#\{ \gamma \in \Gamma | d(o, \gamma o) \leq n \}) \geq \frac{1}{n} \floor{\frac{n}{k+1}} \log(\#(S \cap A_k)), \]
		pour tout $n \geq k$.
		
		Or, il existe un entier $k$ tel que 
		\[ \frac{1}{k+1} \log(\#(S \cap A_k)) \geq \delta_S - \epsilon/2. \]
		Et comme $S$ est une partie couvrante de $\Gamma^c$, on a $\delta_S = h_{\Gamma^c} = h_\Gamma$.
		
		De plus, on a $\frac{1}{n} \floor{\frac{n}{k+1}} (k+1) \geq 1 - \frac{k+1}{n}$, ce qui nous donne l'inégalité
		\[ \frac{1}{n} \log(\#\{ \gamma \in \Gamma | d(o, \gamma o) \leq n \}) \geq \left(1 - \frac{k+1}{n} \right) \left(h_\Gamma - \epsilon/2 \right). \]
		On peut alors trouver un entier $n_k \geq n_0$ tel que pour tout $n \geq n_k$ on ait
		\[ h_\Gamma + \epsilon \geq \frac{1}{n} \log(\#\{ \gamma \in \Gamma | d(o, \gamma o) \leq n \}) \geq h_\Gamma - \epsilon. \]
		
		Ainsi, on a bien montré que l'on a
		\[ \lim_{n \to \infty} \frac{1}{n} \log(\#\{ \gamma \in \Gamma | d(o, \gamma o) \leq n \}) = h_\Gamma. \]
		
%

		
	\end{proof}

	\section{Semi-continuité inférieure de l'entropie} \label{semi-continuite}
	Nous allons voir que l'entropie est semi-continue inférieurement en un semi-groupe.
	Pour cela, commençons par donner une notion de convergence sur l'ensemble des semi-groupes d'isométries d'un espace métrique $X$.
	
	Rappelons la définition de la topologie usuelle compacte-ouverte sur $\Isom(X)$.
	\begin{define}
		Soit $X$ un espace métrique.
		On dit qu'une suite d'isométries $(\gamma_n)_{n \in \N} \in (\Isom(X))^\N$ \defi{converge}
		vers une isométrie $\gamma \in \Isom(X)$ si pour tout compact $K$ de $X$, la suite
		$(\restriction{\gamma_n}{K})_{n \in \N}$ des isométries restreintes à $K$ converge uniformément vers $\restriction{\gamma}{K}$, et si de même la suite
		$(\restriction{\gamma_n^{-1}}{K})_{n \in \N}$ converge uniformément vers $\restriction{\gamma^{-1}}{K}$.
	\end{define}
	
	\begin{rem}
		Ici, la convergence uniforme des inverses $(\restriction{\gamma_n^{-1}}{K})_{n \in \N}$ sur tout compact $K$ est automatique à partir de la convergence uniforme de la suite $(\restriction{\gamma_n}{K})_{n \in \N}$ pour tout compact $K$, puisque ce sont des isométries.
	\end{rem}
	
	\begin{define}
		Soit $X$ un espace métrique.
		On dit qu'une suite $(\Gamma_n)_{n \in \N}$ de semi-groupes d'isométries de $X$ \defi{converge géométriquement} vers un semi-groupe $\Gamma$,
		si l'on a les deux propriétés :
		\begin{itemize}
			\item pour toute isométrie $\gamma \in \Gamma$, il existe une suite d'isométries $(\gamma_n)_{n \in \N}$, avec pour tout $n$, $\gamma_n \in \Gamma_n$ et telle que $\gamma_n$ converge vers $\gamma$.
			\item pour toute partie infinie $P \subseteq \N$ et toute suite d'isométries $(\gamma_n)_{n \in P}$ qui converge vers une isométrie $\gamma \in \Isom(X)$, avec $\gamma_n \in \Gamma_n$ pour tout $n \in P$, on a $\gamma \in \Gamma$.
		\end{itemize}
	\end{define}
	
	Voir~\cite{haettel} pour plus de détails sur la convergence géométrique.
	
	\begin{rem}
		Dans notre résultat de semi-continuité, nous avons besoin seulement de la première de ces deux propriétés.
	\end{rem}
	
	Voici le résultat de semi-continuité :
	
	\begin{cor} \label{semi-cont}
		Soit $(\Gamma_n)_{n \in \N}$ une suite de semi-groupes d'isométries d'un espace Gromov-hyperbolique propre à bord compact $X$ 
		qui converge vers un semi-groupe $\Gamma$ dont l'ensemble limite contient au moins deux points. Alors on a l'inégalité
		\[ h_\Gamma \leq \liminf_{n \to \infty} h_{\Gamma_n}. \]
		Autrement dit, l'entropie est semi-continue inférieurement en les semi-groupes dont l'ensemble limite contient au moins deux points.
	\end{cor}
	
	\begin{rem}
		On pourrait aussi montrer que l'entropie est continue en les semi-groupes de Schottky.
	\end{rem}

	L'idée de la preuve du corollaire \ref{semi-cont}, est de montrer que si l'on a une suite de semi-groupes qui converge, alors on peut approcher un sous-semi-groupe de Schottky du semi-groupe limite par des sous-semi-groupes des semi-groupes de la suite, en trouvant des éléments qui s'approchent des générateurs. Ces semi-groupes seront alors des semi-groupes de Schottky dont les exposants critiques seront proches de celui du semi-groupe de Schottky de départ, et ainsi on obtiendra l'inégalité voulue.
	
	\begin{proof}[Preuve du corollaire \ref{semi-cont}]
		D'après le théorème \ref{thm1}, il existe un sous-semi-groupe con\-tra\-ctant $\Gamma^c$ de $\Gamma$ pour des parties $X_-$ et $X_+ \subset X$, avec $h_{\Gamma^c} = h_\Gamma$.
		De plus, on peut supposer que les parties $X_-$ et $X_+$ de $X$ sont ouvertes, quitte à les remplacer chacune par un $\epsilon$-voisinage ouvert, pour $\epsilon > 0$ assez petit. 
		
		D'après le critère de contraction (proposition \ref{crit_cont}), il existe un réel $n_0$ tel que l'ensemble d'isométries
		\[ \Isom(X)^{X_- \times X_+}_{>n_0} := \{ \gamma \in \Isom(X) | \gamma^{-1} o \in X_-, \gamma o \in X_+ \text{ et } d(o, \gamma o) > n_0 \}, \]
		soit contractant, pour des domaines $X_-'$ et $X_+'$.
		
		D'après le lemme \ref{sc}, il existe alors un réel $r$ et un entier $n_1$, tels que pour toute partie $r$-séparée $S$ de $\Isom(X)^{X_- \times X_+}_{>n_0}$ et pour tout $n \geq n_1$, le semi-groupe engendré par $S \cap \interieur{A_n}$ soit de Schottky pour la partie $X_+'$, où
		\[ \interieur{A_n} := \{ \gamma \in \Isom(X) | d(o, \gamma o) \in ]n, n+1[ \}. \]
		
		Soit $S$ une partie $r$-séparée et couvrante du semi-groupe contractant $\Gamma^c$. 
		
		L'ensemble $S \cap \interieur{A_n}$ étant fini, il existe une suite d'ensembles d'isométries $S_{k,n}$ de $\Gamma_k$, avec $\#S \cap \interieur{A_n} = \# S_{k, n}$, qui converge uniformément vers $S \cap \interieur{A_n}$.
		
		Les ensembles $\interieur{A_n} \cap X_-$ et $\interieur{A_n} \cap X_+$ étant ouverts, il existe un entier $k_0$ tel que pour tout $k \geq k_0$, on ait $S_{k,n} o \subseteq \interieur{A_n} \cap X_+$ et $S_{k,n}^{-1} o \subseteq \interieur{A_n} \cap X_-$. Pour $n > n_0$ et $k \geq k_0$, on a donc l'inclusion $S_{k,n} \subseteq \Isom(X)^{X_- \times X_+}_{>n_0}$.
		La condition pour une partie $S$ d'être $r$-séparée est également une condition ouverte, donc il existe un entier $k_1 \geq k_0$ tel que pour tout $k \geq k_1$ la partie $S_{k,n}$ soit r-séparée.
		
		Ainsi, pour $n$ assez grand et pour tout entier $k$ assez grand en fonction de $n$, la partie $S_{k,n}$ est incluse dans le semi-groupe contractant $\Isom(X)^{X_- \times X_+}_{>n_0}$, et est $r$-séparée. Elle engendre donc un semi-groupe de Schottky.
		
		Par le lemme \ref{mnec}, on a donc l'inégalité
		\[ h_{\Gamma_k} \geq \frac{\log(\# S_{k,n})}{n+1} = \frac{\log(\# S \cap \interieur{A_n})}{n+1}, \]
		pour tout $n$ assez grand, et pour tout $k$ assez grand en fonction de $n$.
		
		On obtient donc ce que l'on voulait
		\[ \liminf_{k \to \infty} h_{\Gamma_k} \geq \limsup_{n \to \infty} \frac{\log(\# S \cap \interieur{A_n})}{n+1} = h_\Gamma. \]

	\end{proof}

	\subsubsection{Exemple d'application : les semi-groupes de Kenyon} \label{semi-ken}
	
	Les semi-groupes de Kenyon (voir \cite{ken}) sont les semi-groupes $\Gamma_t$ engendrés par les $3$ transformations affines
	\[
		\left\{\begin{array}{cl}
			x &\mapsto x/3 \\
			x &\mapsto x/3 + t \\
			x &\mapsto x/3 + 1
		\end{array}\right.
	\]
	pour des réels $t$.
	
	D'après le corollaire \ref{semi-cont}, l'application $t \mapsto h_{\Gamma_t} = \dim(\Lambda_{\Gamma_t})$ est semi-continue inférieurement.
	En particulier, pour trouver un contre exemple à la conjecture de Furstenberg (i.e. un réel $t$ pour lequel on a $\dim_H(\Lambda_{\Gamma_t}) < \delta_{\Gamma_t}$), il suffit de trouver une suite de rationnels $(t_n)_{n \in \N}$ qui converge vers un irrationnel et avec pour tout $n$, $\delta_{\Gamma_{t_n}} \leq C < 1$.
	Mais bien que l'on sache calculer l'exposant critique $\delta_{\Gamma_t}$ du semi-groupe $\Gamma_t$ pour tout rationnel $t$, on ne sait pas s'il existe de telles suites. Voir \cite{ken} pour plus de détails. 
	
	\begin{figure}[H]
		\centering
		\caption{Développement en base $\beta = 3$, avec ensemble de chiffres $A = \{0, \frac{\pi}{4}, 1\}$.} \label{pi/4}
		\includegraphics{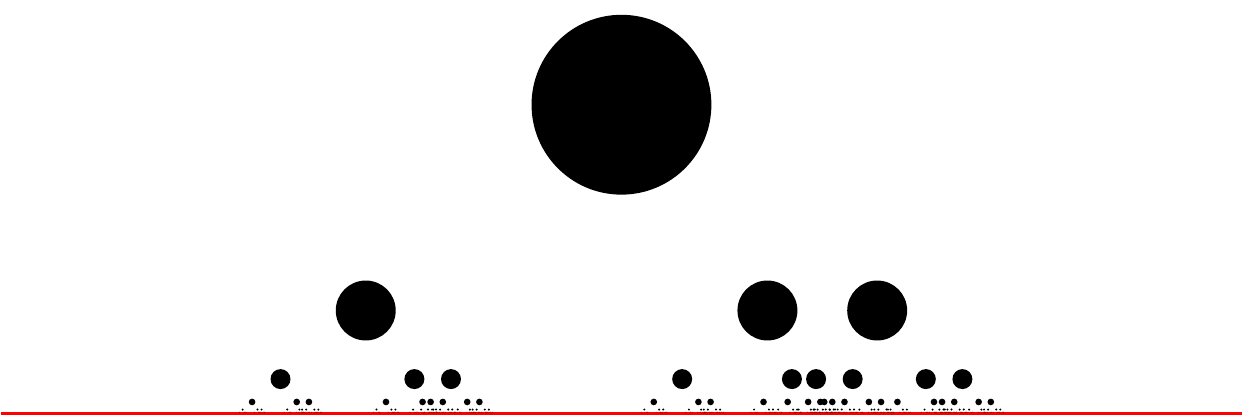}
	\end{figure}
	
	\begin{rem}
		L'application $t \mapsto h_{\Gamma_t} = \dim_H(\Lambda_{\Gamma_t})$ n'est pas continue.
		En effet, on sait que l'on a $\dim_H(\Lambda_{\Gamma_t}) = 1$ pour $t$ dans une partie dense de $\R$, et on a par exemple $\dim_H(\Lambda_{\Gamma_{2/3}}) = \delta_{\Gamma_{2/3}} < 1$ (voir figure \ref{2/3}).
	\end{rem}

	\begin{figure}[H]
		\centering
		\caption{Ensemble limite d'un sous-semi-groupe de $SL(2,\C)$}
		\includegraphics[scale=.13]{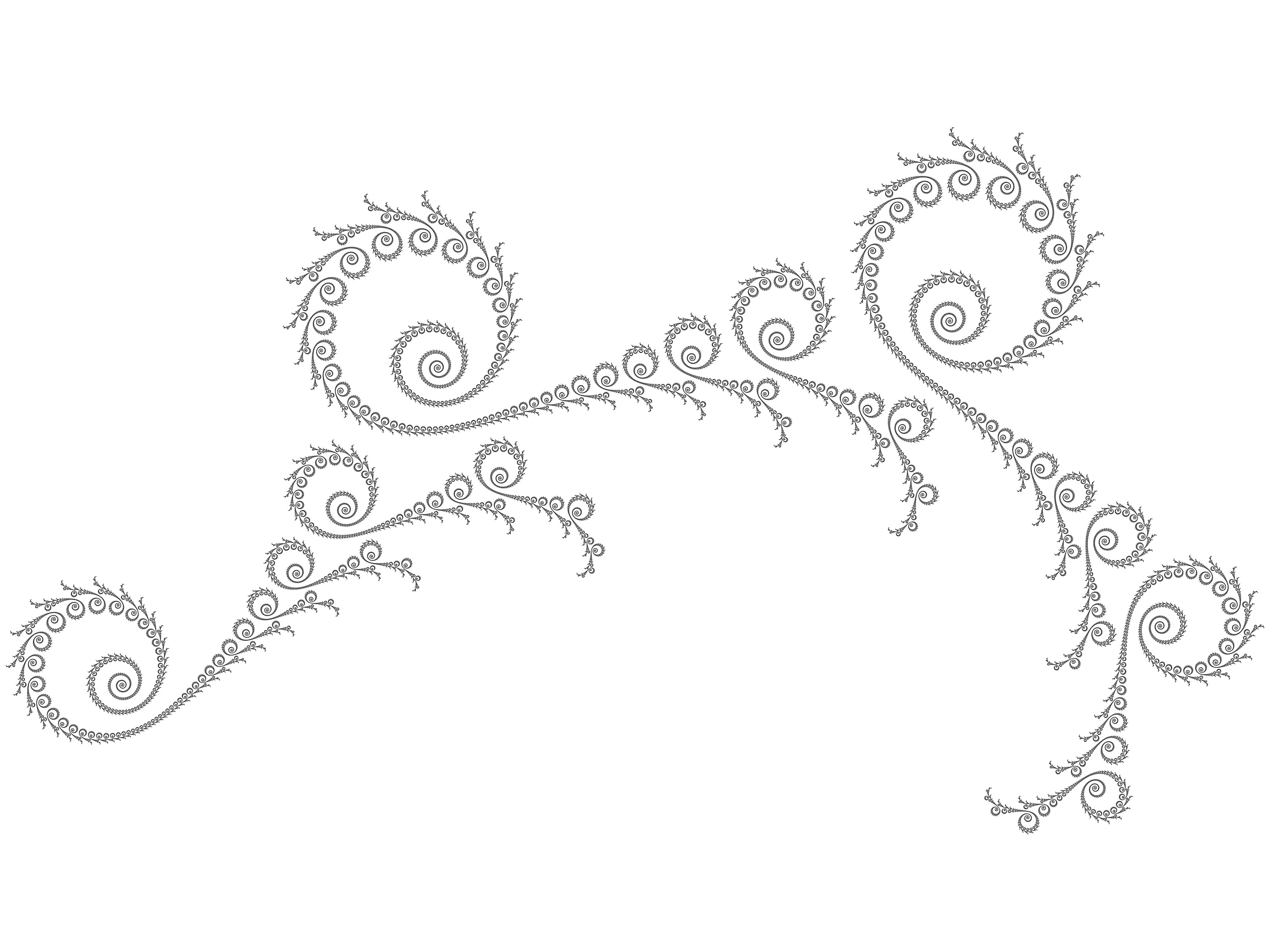}
	\end{figure}

	\begin{figure}[H]
		\centering
		\caption{Un groupe quasi-Schottky}
		\includegraphics{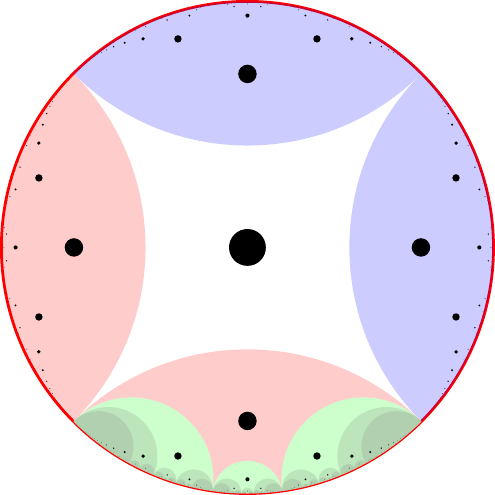}
	\end{figure}
	
	\immediate\write18{./dp -save}
\end{document}